\documentclass{amsart}
\title{Rigid and Schurian modules over cluster-tilted algebras of tame
type}
\date{26 January 2016}

\author[Marsh]{Robert J. Marsh}
\address{
School of Mathematics \\
University of Leeds \\
Leeds \\
LS2 9JT \\
United Kingdom
\vskip 0.05cm
Guest Professor at Department of Mathematical Sciences \\
NTNU \\
NO-7491 Trondheim \\
Norway, Aug-Dec 2014}
\email{marsh@maths.leeds.ac.uk}

\author[Reiten]{Idun Reiten}
\address{Department of Mathematical Sciences \\
NTNU \\
NO-7491 Trondheim \\
Norway}
\email{Idun.Reiten@math.ntnu.no}

\thanks{This work was supported by the Engineering and Physical Sciences Research Council [grant number EP/G007497/1], the Mathematical Sciences Research Institute, Berkeley and NFR FriNat
[grant number 231000].}


\usepackage{tikz,ifthen}
\usetikzlibrary{backgrounds}
\usetikzlibrary{calc,decorations.markings,arrows}
\usepackage[outline]{contour}
\usetikzlibrary{positioning} 
\usetikzlibrary{petri} 
\usetikzlibrary{decorations.pathmorphing} 
\usetikzlibrary{fit} 
\usetikzlibrary{backgrounds} 

\usepackage{enumerate}
\usepackage{xy}
\xyoption{matrix}
\xyoption{all}
\usepackage{array}
\usepackage{amssymb}
\usepackage{psfrag}
\usepackage{graphicx}
\usepackage{color}
\usepackage{verbatim}
\usepackage{pgf}
\usepackage{mathtools}

\theoremstyle{plain}
\newtheorem{lemma}{Lemma}[section]
\newtheorem{theorem}[lemma]{Theorem}
\newtheorem{corollary}[lemma]{Corollary}

\newtheorem{proposition}[lemma]{Proposition}

\theoremstyle{definition}
\newtheorem{definition}[lemma]{Definition}
\newtheorem{remark}[lemma]{Remark}
\newtheorem{example}[lemma]{Example}
\newtheorem{assumption}[lemma]{Assumption}

\numberwithin{equation}{section}

\newcommand{\coker}{\operatorname{coker}}
\newcommand{\End}{\operatorname{End}}
\newcommand{\Hom}{\operatorname{Hom}}
\newcommand{\Ext}{\operatorname{Ext}}

\newcommand{\ind}{\operatorname{ind}}
\newcommand{\add}{\operatorname{add}}

\newcommand{\Top}{\operatorname{Top}}
\newcommand{\im}{\operatorname{im}}
\newcommand{\ql}{\operatorname{ql}}
\newcommand{\rad}{\operatorname{rad}}
\newcommand{\opp}{\operatorname{opp}}
\newcommand{\id}{\operatorname{id}}

\newcommand{\C}{\mathcal{C}}
\newcommand{\T}{\mathcal{T}}
\newcommand{\W}{\mathcal{W}}
\newcommand{\X}{\mathcal{X}}
\newcommand{\XX}{\clu{X}}
\newcommand{\R}{\mathcal{R}}
\newcommand{\HH}{\mathcal{H}}
\newcommand{\M}{\mathcal{M}}
\newcommand{\I}{\mathcal{I}}
\newcommand{\Z}{\mathbb{Z}}
\newcommand{\HomH}[1]{\Hom_{#1}^H}
\newcommand{\HomF}[1]{\Hom_{#1}^F}
\newcommand{\PP}{P^{\Lambda}}
\newcommand{\II}{I^{\Lambda}}
\newcommand{\SL}{S^{\Lambda}}
\newcommand{\A}{\mathcal{A}}
\newcommand{\dimv}{d'}
\newcommand{\clu}[1]{\widetilde{#1}}
\newcommand{\wideR}[1]{\mathcal{S}_{#1}}

\newcommand{\kerphi}{L}
\newcommand{\field}{K}

\colorlet{shadecolor}{gray!70}

\newcommand{\qhole}{\times}
\newcommand{\taurigid}{\bullet}
\newcommand{\schuriannotrigid}{\circ}
\newcommand{\notschuriantaurigid}{\scriptscriptstyle \blacksquare}
\newcommand{\schurianrigidnottaurigid}{{\color{shadecolor} \bullet}}
\newcommand{\notschuriannotrigid}{\scriptscriptstyle \square}
\newcommand{\Dual}{D}

\begin{document}

\begin{abstract}
We give an example of a cluster-tilted algebra $\Lambda$ with quiver $Q$, such that
the associated cluster algebra $\A(Q)$ has a denominator vector which is not the dimension
vector of any indecomposable $\Lambda$-module. This answers a question posed by T. Nakanishi.
The relevant example is a cluster-tilted algebra associated with a tame hereditary algebra.
We show that for such a cluster-tilted algebra $\Lambda$, we can write any denominator
vector as a sum of the dimension vectors of at most three indecomposable rigid $\Lambda$-modules.
In order to do this it is necessary, and of independent interest, to first classify
the indecomposable rigid $\Lambda$-modules in this case.
\end{abstract}


\maketitle

\section*{Introduction}
In the theory of cluster algebras initiated by Fomin 
and Zelevinsky, the authors introduced some important kinds of vectors, 
amongst them the $d$-vectors (denominator vectors)~
\cite{FZ02} and the $c$-vectors~\cite{FZ07}.
These vectors have played an important role in the theory. In particular, they have been
important for establishing connections with the representation theory of finite
dimensional algebras.

Let $Q$ be a finite quiver with $n$ vertices, without loops or two-cycles, and
let $\A(Q)$ be the associated cluster algebra with initial cluster $\{x_1,\ldots ,x_n\}$.
Each non-initial cluster variable is known to be of the form $f/m$, where $m=x_1^{d_1}\cdots x_n^{d_n}$
for nonnegative integers $d_i$ and $f$ is not divisible by any $x_i$. Then the associated
$d$-vector is $(d_1,\ldots ,d_n)$. For the definition of $c$-vector we refer to~\cite{FZ07}.
On the other hand, we have the dimension vectors of the finite
dimensional rigid indecomposable $\field Q$-modules.

Assume first that $Q$ is acyclic. Then there are known interesting connections between
the $d$-vectors and the $c$-vectors on the one hand and the dimension vectors of the
indecomposable rigid $\field Q$-modules on the other hand. More specifically, there is a bijection
between the non-initial cluster variables and the indecomposable rigid $\field Q$-modules such that
the $d$-vector of a cluster variable coincides with the dimension vector of the corresponding
module (see~\cite{BMRT07,BCKMRT07,CK06}). Furthermore, the (positive) $c$-vectors of $\A(Q)$ and
the dimension vectors of the indecomposable rigid $\field Q$-modules coincide (see~\cite{chavez,ST13}).

However, when the initial quiver $Q$ is not acyclic, we do not have such nice connections
(see~\cite{AD13,BM10,BMR09} for work in this direction). Answering a question posed to us by
Nakanishi, we found an example showing the following:

(*) There is a cluster-tilted algebra $\Lambda$ with quiver $Q$ such that
$\A(Q)$ has a denominator vector which is not the dimension vector of any indecomposable
$\Lambda$-module.

Since we know that there are denominator vectors which are not dimension vectors, it is natural
to ask if the denominator vectors can be written as a sum of a small number of dimension vectors of
indecomposable rigid $\Lambda$-modules. We consider this question for cluster-tilted algebras associated
to tame hereditary algebras.
Note that by~\cite[Theorem 3.6]{GLS}, using~\cite[Theorem 5.2]{BIRS11} and~\cite[Theorem 5.1]{BMR08}, such cluster-tilted algebras are exactly the cluster-tilted algebras of tame representation type (noting that the cluster-tilted algebras of finite representation type are those arising from hereditary algebras of finite representation type by~\cite[Theorem A]{BMR07}).
In this case we show that it is possible to use at most $3$ summands. We do not know if it is always possible with $2$ summands.

In order to prove the results discussed in the previous paragraph we need to locate the indecomposable
rigid $\Lambda$-modules in the AR-quiver of $\Lambda$-mod. This investigation should be interesting
in itself. Closely related is the class of indecomposable Schurian modules, which we also describe.
If $H$ is a hereditary algebra, then every indecomposable rigid (equivalently, $\tau$-rigid) module
is Schurian. So one might ask what the relationships are
between the rigid, $\tau$-rigid and Schurian $\Lambda$-modules. In general there are $\tau$-rigid (hence rigid)
$\Lambda$-modules which are not Schurian. However, it
turns out that every indecomposable $\Lambda$-module which
is rigid, but not $\tau$-rigid, is Schurian.

In Section~\ref{s:setup}, we recall some basic definitions and results relating to cluster 
categories. In Section~\ref{s:tube} we discuss tubes in general. In Section~\ref{s:tubeproperties} we 
fix a cluster-tilting object $T$ in a cluster category associated to a tame hereditary algebra and 
investigate its properties in relation to a tube. Section~\ref{s:rigid} is devoted to identifying the 
rigid and Schurian $\End_{\C}(T)^{\opp}$-modules. In Section~\ref{s:wild}, we investigate
an example in the wild case which appears to behave in a
similar way to the tame case.
In Section~\ref{s:counterexample}, we give the example providing a negative answer to the
question of Nakanishi. Finally, in Section~\ref{s:three}, we also show that for
cluster-tilted algebras associated to tame hereditary algebras each denominator vector is
a sum of at most $3$ dimension vectors of indecomposable rigid $\Lambda$-modules.

We refer to~\cite{ASS06,ARS97} for standard facts from representation theory.
We would like to thank Otto Kerner for helpful conversations about wild hereditary algebras.

\section{Setup} \label{s:setup}
In this section we recall some definitions and results related to
cluster categories and rigid and $\tau$-rigid objects. We also include
some lemmas which are useful for showing that a module is Schurian or rigid.

For a modulus $N$, we choose representatives
$\mathbb{Z}_N=\{0,1,\ldots ,N-1\}$, writing
$[a]_N$ for the reduction of an integer $a$ mod $N$.
If $N=0$, we take $\mathbb{Z}_N$ to be the empty set.

We fix an algebraically closed field $\field$; all categories
considered will be assumed to be $\field$-additive.
For an object $X$ in a category $\X$, we denote
by $\add(X)$ the additive subcategory generated by $X$.
Suppose that $\X$ is a module category with AR-translate
$\tau$. Then we say that $X$ is \emph{rigid} if $\Ext^1(X,X)=0$,
$\tau$-\emph{rigid} if $\Hom(X,\tau X)=0$, \emph{Schurian} if 
$\End(X)\cong \field$, or \emph{strongly Schurian} if the multiplicity
of each simple module as a composition factor is at most one. 
Note that any strongly Schurian module is necessarily Schurian.

If $\X$ is a triangulated category with shift $[1]$ and AR-translate $\tau$,
we define rigid, $\tau$-rigid and Schurian objects similarly, where we
write $\Ext^1(X,Y)$ for $\Hom(X,Y[1])$.
For both module categories and triangulated categories, we shall consider
objects of the category up to isomorphism.

For modules $X,Y$ in a module category over a finite dimensional algebra,
we write $\overline{\Hom}(X,Y)$ for the injectively stable 
morphisms from $X$ to $Y$, i.e. the quotient of
$\Hom(X,Y)$ by the morphisms from $X$ to $Y$ which factorize 
through an injective module. We similarly write
$\underline{\Hom}(X,Y)$ for the projectively stable morphisms.
Then we have the AR-formula:
\begin{equation}
\label{e:ARformula}
D\overline{\Hom}(X,\tau Y)\cong \Ext^1(X,Y)\cong
D\underline{\Hom}(\tau^{-1}X,Y),
\end{equation}
where $D$ denotes the functor $\Hom(-,\field)$.

Let $Q=(Q_0,Q_1)$ and $Q'=(Q'_0,Q'_1)$ be quivers with vertices
$Q_0,Q'_0$ and arrows $Q_1,Q'_1$. Recall that a morphism
of quivers from $Q$ to $Q'$ is a pair of maps $f_i:Q_i\rightarrow Q'_i$, $i=0,1$,
such that whenever $\alpha:i\rightarrow j$ is an arrow in $Q$, we have that
$f_1(\alpha)$ starts at $f_0(i)$ and ends at $f_0(j)$.
In order to describe the modules we are working with, it is convenient
to use notation from~\cite{ringel98}, which we now recall.

\begin{definition} \label{d:colouredquiver}
Let $Q$ be a quiver with vertices $Q_0$.
A $Q$-\emph{coloured quiver} is a pair $(\Gamma,\pi)$,
where $\Gamma$ is a quiver and $\pi:\Gamma\rightarrow Q$ is a
morphism of quivers. We shall always assume that $\Gamma$ is a tree.
\end{definition}

As Ringel points out, a $Q$-coloured quiver $(\Gamma,\pi)$ can be regarded as a quiver $\Gamma$ in which each vertex is coloured by a vertex of $Q$ and each arrow is coloured by an arrow of $Q$.
In addition, if an arrow $\gamma:v\rightarrow w$ in $\Gamma$ is coloured by an arrow $\alpha:i\rightarrow j$ in $Q$ then $v$ must be coloured with $i$ and $w$ must be coloured with $j$.
We shall draw $Q$-coloured quivers in this way.
Thus each vertex $v$ of $\Gamma$ will be labelled with
its image $\pi(v)\in Q_0$, and each arrow $a$ of $\Gamma$ will
be labelled with its image $\pi(a)$ in $Q_1$.
But note that if $Q$ has no multiple arrows then we can omit the arrow labels, since the label of an arrow in $\Gamma$ is determined by the labels of its endpoints.

We shall also omit the orientation of the arrows in $\Gamma$, adopting the convention that the arrows always point down the page.

As in~\cite[Remark 4]{ringel98}, a $Q$-coloured quiver
$(\Gamma,\pi)$ determines a representation
$V=V(\Gamma,\pi)$ of $Q$ over $\field$ (and hence a $\field Q$-module) in the following way. For each
$i\in Q_0$, let $V_i$ be the vector space with basis given by
$B_i=\pi^{-1}(i)\subseteq \Gamma_0$. Given an arrow
$\alpha:i\rightarrow j$ in $Q$ and $b\in \pi^{-1}(i)$, we define
\begin{equation}
\label{e:linearmap}
\varphi_{\alpha}(b)=\sum_{b\xrightarrow{\alpha} c \text{ in }\Gamma} c,
\end{equation}
extending linearly.

If $A=\field Q/I$, where $I$ is an admissible ideal, and
$V$ satisfies the relations coming from the elements of $I$
then it is an $A$-module.
Note that, in general, not every $A$-module will arise in this way
(for example, over the Kronecker algebra).
Also, a given module may be definable using more than one $Q$-coloured quiver (by changing basis).

As an example of a coloured quiver, consider the quiver $Q$:
\vskip 0.05cm
\begin{equation}
\label{e:examplequiver}
\xymatrix{
1 \ar[r] \ar@/^1pc/[rrr] & 2 \ar[r] & 3 \ar[r] & 4.
}
\end{equation}
Then we have the following $Q$-coloured quivers and corresponding representations:
\begin{equation}
\label{e:T2}
T_2=
\begin{tikzpicture} [baseline,xscale=0.7,yscale=1,ext/.style={black,shorten <=-1pt, shorten >=-1pt}]
  \draw (0,0.5) node (A1) {$\scriptstyle 1$};
  \draw (0,0) node (A4) {$\scriptstyle 4$};
\draw[ext] (A1) -- (A4);
 \end{tikzpicture}, \quad\quad
\xymatrix{
\field \ar[r] \ar@/^1pc/[rrr]^{\id} & 0 \ar[r] & 0 \ar[r] & \field.
}
\end{equation}
\begin{equation}
\label{e:T3}
T_3=\begin{tikzpicture} [baseline,xscale=0.7,yscale=1,ext/.style={black,shorten <=-1pt, shorten >=-1pt}]
  \draw (0,0) node (A2) {$\scriptstyle 2$};
  \draw (0.25,0.5) node (A1) {$\scriptstyle 1$};
  \draw (0.5,0) node (A4) {$\scriptstyle 4$};
\draw[ext] (A1) -- (A2);
\draw[ext] (A1) -- (A4);
 \end{tikzpicture},\quad \quad
\xymatrix{
\field \ar[r]_{\id} \ar@/^1pc/[rrr]^{\id} & \field \ar[r] & 0 \ar[r] & \field.
}
\end{equation}

\begin{remark} \label{r:redraw}
We will sometimes label the vertices of a $Q_{\Lambda}$-coloured quiver $(\Gamma,\pi)$ by writing
$$\pi^{-1}(i)=\{b_{ik}\,:\,k=1,2,\ldots \}$$
for $i$ a vertex of $Q$.
Then, if $\alpha$ is an arrow from $i$ to $j$ in $Q$,~\eqref{e:linearmap} becomes:
$$\varphi_{\alpha}(b_{ik})=\sum_{l,\ b_{ik}\xrightarrow{\alpha} b_{jl}\text{ in }\Gamma} b_{jl}.$$
To aid with calculations, we may also redraw $\Gamma$, placing all of the basis elements
$b_{ij}$ (for fixed $i$) close together (according to a fixed embedding of $Q$ in the plane).
In this case, we must include the arrowheads on the arrows so that this information is not lost.
For an example, see Figure~\ref{f:redrawn}.
\end{remark}

\begin{figure}
$$
\begin{tikzpicture}[xscale=0.8,yscale=1,baseline=(bb.base),quivarrow/.style={black, -latex},translate/.style={black, dotted},relation/.style={black, dotted},ext/.style={black,shorten <=-1pt, shorten >=-1pt},exta/.style={black,->,shorten <=-1pt, shorten >=-1pt}]

\path (0,0) node (bb) {}; 

\draw (0,0) node (A1) {$\scriptstyle 1$};
\draw (1,1) node (A2) {$\scriptstyle 2$};
\draw (2,0) node (A3) {$\scriptstyle 3$};
\draw[quivarrow] (A1) -- (A2);
\draw[quivarrow] (A2) -- (A3);
\draw[quivarrow] (A3) -- (A1);

\begin{scope}[shift={(3.5,-0.7)}]
\draw (0,0) node (V2A) {$\scriptstyle 2$};
\draw (0,0.7) node (V1A) {$\scriptstyle 1$};
\draw (0,1.4) node (V3) {$\scriptstyle 3$};
\draw (0,2.1) node (V2B) {$\scriptstyle 2$};
\draw (0,2.8) node (V1B) {$\scriptstyle 1$};
\draw[ext] (V1B) -- (V2B);
\draw[ext] (V2B) -- (V3);
\draw[ext] (V3) -- (V1A);
\draw[ext] (V1A) -- (V2A);
\end{scope}

\begin{scope}[shift={(5,-0.7)}]
\draw (0,0) node (V2A) {$\scriptstyle b_{22}$};
\draw (0,0.7) node (V1A) {$\scriptstyle b_{12}$};
\draw (0,1.4) node (V3) {$\scriptstyle b_{31}$};
\draw (0,2.1) node (V2B) {$\scriptstyle b_{21}$};
\draw (0,2.8) node (V1B) {$\scriptstyle b_{11}$};
\draw[exta] (V1B) -- (V2B);
\draw[exta] (V2B) -- (V3);
\draw[exta] (V3) -- (V1A);
\draw[exta] (V1A) -- (V2A);
\end{scope}

\begin{scope}[shift={(6.8,0)}]
\draw (1,1) node (V2B) {$\scriptstyle b_{22}$};
\draw (0,0) node (V1B) {$\scriptstyle b_{12}$};
\draw (2,0) node (V3) {$\scriptstyle b_{31}$};
\draw (1,1.5) node (V2A) {$\scriptstyle b_{21}$};
\draw (0,0.5) node (V1A) {$\scriptstyle b_{11}$};
\draw[exta] (0.1,0.7) -- (0.7,1.3);
\draw[exta] (1.3,1.4) -- (2,0.2);
\draw[exta] (1.7,0) -- (0.32,0);
\draw[exta] (0.27,0.13) -- (0.95,0.8);
\end{scope}

\begin{scope}[shift={(10,0)}]
\draw (0,0) node (R1) {$\scriptstyle \field^2$};
\draw (1,1) node (R2) {$\scriptstyle \field^2$};
\draw (2,0) node (R3) {$\scriptstyle \field$};
\draw (0.1,0.75) node {$\left(\begin{smallmatrix} 1 & 0 \\ 0 & 1 \end{smallmatrix}\right)$};
\draw (2.12,0.55) node {$\left(\begin{smallmatrix} 1 & 0 \end{smallmatrix}\right)$};
\draw (1,-0.32) node {$\left(\begin{smallmatrix} 0 \\ 1 \end{smallmatrix}\right)$};
\draw[exta] (R1) -- (R2);
\draw[exta] (R2) -- (R3);
\draw[exta] (R3) -- (R1);
\end{scope}
\end{tikzpicture}
$$
\caption{A quiver $Q$, a $Q$-coloured quiver, together with the redrawing according to Remark~\ref{r:redraw} and the corresponding representation of $Q$.}
\label{f:redrawn}
\end{figure}
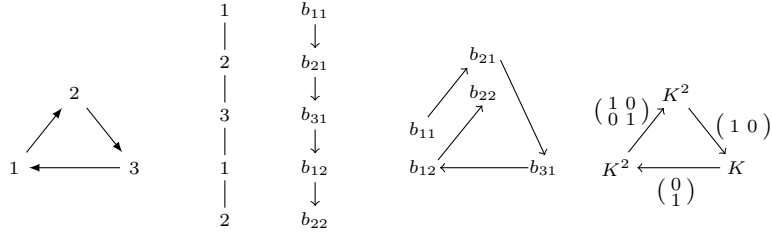

\begin{definition}
If $(\Gamma,\pi)$, $(\Gamma',\pi')$ are $Q$-coloured quivers then we call a map
$\varphi:\Gamma\rightarrow \Gamma'$ a morphism of $Q$-coloured quivers if it is
a morphism of quivers and $\pi=\pi'\varphi$.

If $\Gamma''$ is a full subquiver of $\Gamma$ and $\pi''$ is the restriction of $\pi$
to $\Gamma$ then $\Gamma''$ is called a $Q$-coloured subquiver of $(\Gamma,\pi)$; note that it is
again a $Q$-coloured quiver.
\end{definition}

\begin{remark} \label{r:colouredquivermorphisms}
If $(\Gamma',\pi')$ is a $Q$-coloured subquiver of $(\Gamma,\pi)$ with the property
that every arrow between a vertex in $\Gamma'$ and a vertex in $\Gamma$ not in $\Gamma'$
points towards the vertex in $\Gamma'$, then it is easy to see that there is a corresponding
embedding of modules $V(\Gamma',\pi')\hookrightarrow V(\Gamma,\pi)$. Similarly, if every such arrow
points towards $\Gamma'$, there is a corresponding quotient map
$V(\Gamma,\pi)\twoheadrightarrow V(\Gamma',\pi')$.

Let $(\Gamma(1),\pi(1))$ and $(\Gamma(2),\pi(2))$ be $Q$-coloured quivers.
Suppose that there is a $Q$-coloured quiver $(\Gamma,\pi)$ which is isomorphic to a $Q$-coloured
subquiver of $(\Gamma(1),\pi(1))$ with the second property above. Suppose in addition that
it is isomorphic to a $Q$-coloured subquiver of $(\Gamma(2),\pi(2))$ with the first property above.
Then there is a $\field Q$-module homomorphism $V(\Gamma(1),\pi(1))\rightarrow V(\Gamma(2),\pi(2))$ given by
the composition of the quotient map and the embedding given above.
\end{remark}

We fix a quiver $Q$ such that the path algebra $\field Q$ has
tame representation type. For example, we could take $Q$ to be
the quiver~\eqref{e:examplequiver}.
We denote by $\field Q$-mod the category of finite-dimensional
$\field Q$-modules, with AR-translate $\tau$.

We denote by $D^b(\field Q)$ the bounded derived category of $\field Q$-mod,
with AR-translate also denoted by $\tau$.
For objects $X$ and $Y$ in
$D^b(\field Q)$, we write $\Hom(X,Y)$ for $\Hom_{D^b(\field Q)}(X,Y)$ and
$\Ext(X,Y)$ for $\Ext_{D^b(\field Q)}(X,Y)$.
Note that if $X,Y$ are modules, these coincide with
$\Hom_{\field Q}(X,Y)$ and $\Ext_{\field Q}(X,Y)$ respectively.

The category $D^b(\field Q)$ is triangulated.
Let $\C=\C_Q$ denote the cluster category corresponding to
$Q$, i.e.\ the orbit category $\C_Q=D^b(\field Q)/F$, where
$F$ denotes the autoequivalence $\tau^{-1}[1]$ (see~\cite{BMRRT06}).
The category $\C$ is triangulated by~\cite[\S4]{keller05}.
Note that an object in $D^b(\field Q)$-mod can be regarded as an 
object in $\C$; in particular this applies to $\field Q$-modules,
which can be identified with complexes in $D^b(\field Q)$ 
concentrated in degree zero.

If $X,Y$ are ${\field}Q$-modules
regarded as objects in $\C$, then
$$\Hom_{\C}(X,Y)=\Hom(X,Y)\oplus \Hom(X,FY)$$
by~\cite[Prop.\ 1.5]{BMRRT06}.
We write $\HomH{\C}(X,Y)=\Hom(X,Y)$ and refer to elements of this space as $H$-\emph{maps} from $X$ to $Y$, and
we write $\Hom(X,FY)=\HomF{\C}(X,Y)$ and refer to elements of this space
as $F$-\emph{maps} from $X$ to $Y$. So, we have:
$$\Hom_{\C}(X,Y)=\HomH{\C}(X,Y)\oplus \HomF{\C}(X,Y).$$
Note that
\begin{equation}
\label{e:tau2}
\begin{split}
\HomF{\C}(X,Y) &=\Hom(X,FY)= \Hom(X,\tau^{-1} Y[1]) \\
&\cong \Ext(X,\tau^{-1}Y) \cong \Dual\Hom(\tau^{-1}Y,\tau X) \cong \Dual\Hom(Y,\tau^2 X),
\end{split}
\end{equation}
where $\Dual=\Hom(-,\field)$.
If $\chi$ is an additive subcategory of $\C$, we write:
$$\HomH{\C/\chi}(X,Y),\quad \HomF{\C/\chi}(X,Y)$$
for the quotients of $\HomH{\C}(X,Y)$ and $\HomF{\C}(X,Y)$ by the morphisms in
$\C$ factoring through $\chi$.

A rigid object $T$ in $\C$ is said to be
\emph{cluster-tilting} if, for any object
$X$ in $\C$, we have $\Ext_{\C}^1(T,X)=0$ if and only
if $X$ lies in $\add(T)$.

We fix a cluster-tilting object $T$ in $\C$.
We make the following assumption.
As explained in the proof of Theorem~\ref{t:classification}, to find the rigid
and Schurian modules for any cluster-tilted algebra
arising from $\C$, it is enough to find the rigid and Schurian modules in this case.
\begin{assumption} \label{a:preprojective}
The cluster-tilting object $T$ is induced by a $\field Q$-module (which we also denote by $T$). Furthermore, $T$ is
of the form $U\oplus T'$, where $U$ is preprojective and
$T'$ is regular. Note that the module $T$ is a tilting
module by~\cite{BMRRT06}.
\end{assumption}

\begin{example}
\label{ex:running}
For example, if $Q$ is the quiver in~\eqref{e:examplequiver}, we could take $T$ to be the tilting module:
\begin{equation}
T=P_1\oplus T_2\oplus T_3\oplus P_4,
\end{equation}
where $T_2$ and $T_3$ are the $\field Q$-modules defined in~\eqref{e:T2},~\eqref{e:T3}.
Note that $T$ can be obtained from
$P_1\oplus P_2\oplus P_3\oplus P_4$ by mutating (in the sense of~\cite{HU89,RS90}) first at
$P_2$ and then at $P_3$.
The modules $T_2$ and $T_3$ lie in a tube of rank $3$ in $\field Q$-mod; see Figure~\ref{f:tamehereditaryAR}.
\end{example}

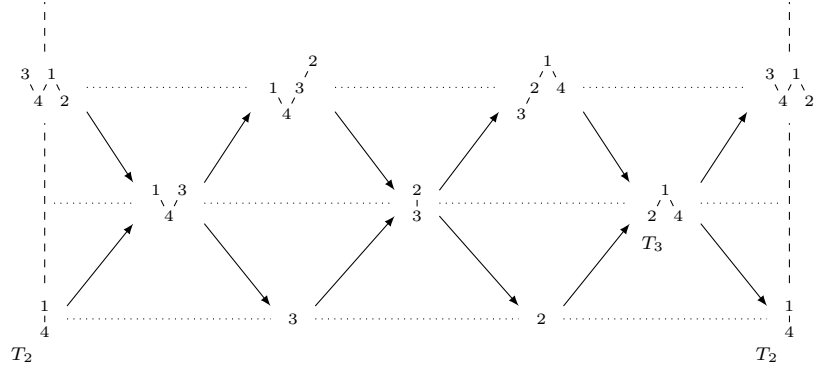
\begin{figure}
$$
\begin{tikzpicture}[xscale=0.33,yscale=0.22,baseline=(bb.base),quivarrow/.style={black, -latex},translate/.style={black, dotted}]

\path (0,0) node (bb) {}; 

\draw (0,0) node (X1) { \small
 \begin{tikzpicture} [scale=0.7,yscale=1,ext/.style={black,shorten <=-1pt, shorten >=-1pt}]
  \draw (0,0) node (X14) {$\scriptstyle 4$};
  \draw (0,0.5) node (X11) {$\scriptstyle 1$};
\draw[ext] (X11) -- (X14);
 \end{tikzpicture}
};

\draw (10,0) node (X2) { \small
 \begin{tikzpicture} [scale=0.7,yscale=1,ext/.style={black,shorten <=-1pt, shorten >=-1pt}]
  \draw (0,0) node (X23) {$\scriptstyle 3$};
 \end{tikzpicture}
};

\draw (20,0) node (X3) { \small
 \begin{tikzpicture} [scale=0.7,yscale=1,ext/.style={black,shorten <=-1pt, shorten >=-1pt}]
  \draw (0,0) node (X32) {$\scriptstyle 2$};
 \end{tikzpicture}
};

\draw (30,0) node (XX1) { \small
 \begin{tikzpicture} [scale=0.7,yscale=1,ext/.style={black,shorten <=-1pt, shorten >=-1pt}]
  \draw (0,0) node (XX14) {$\scriptstyle 4$};
  \draw (0,0.5) node (XX11) {$\scriptstyle 1$};
\draw[ext] (XX11) -- (XX14);
 \end{tikzpicture}
};

\draw (5,7) node (X4) { \small
 \begin{tikzpicture} [scale=0.7,yscale=1,ext/.style={black,shorten <=-1pt, shorten >=-1pt}]
  \draw (0,0) node (X44) {$\scriptstyle 4$};
  \draw (-0.25,0.5) node (X41) {$\scriptstyle 1$};
  \draw (0.25,0.5) node (X43) {$\scriptstyle 3$};
\draw[ext] (X41) -- (X44);
\draw[ext] (X44) -- (X43);
 \end{tikzpicture}
};

\draw (15,7) node (X5) { \small
 \begin{tikzpicture} [scale=0.7,yscale=1,ext/.style={black,shorten <=-1pt, shorten >=-1pt}]
  \draw (0,0) node (X53) {$\scriptstyle 3$};
  \draw (0,0.5) node (X52) {$\scriptstyle 2$};
\draw[ext] (X53) -- (X52);
 \end{tikzpicture}
};

\draw (25,7) node (X6) { \small
 \begin{tikzpicture} [scale=0.7,yscale=1,ext/.style={black,shorten <=-1pt, shorten >=-1pt}]
  \draw (0,0) node (X62) {$\scriptstyle 2$};
  \draw (0.5,0) node (X64) {$\scriptstyle 4$};
  \draw (0.25,0.5) node (X61) {$\scriptstyle 1$};
\draw[ext] (X62) -- (X61);
\draw[ext] (X64) -- (X61);
 \end{tikzpicture}
};

\draw (0,14) node (X7) { \small
 \begin{tikzpicture} [scale=0.7,yscale=1,ext/.style={black,shorten <=-1pt, shorten >=-1pt}]
  \draw (0,0) node (X74) {$\scriptstyle 4$};
  \draw (0.5,0) node (X72) {$\scriptstyle 2$};
  \draw (-0.25,0.5) node (X73) {$\scriptstyle 3$};
  \draw (0.25,0.5) node (X71) {$\scriptstyle 1$};
\draw[ext] (X74) -- (X73);
\draw[ext] (X74) -- (X71);
\draw[ext] (X72) -- (X71);
 \end{tikzpicture}
};

\draw (10,14) node (X8) { \small
 \begin{tikzpicture} [scale=0.7,yscale=1,ext/.style={black,shorten <=-1pt, shorten >=-1pt}]
  \draw (0,0) node (X84) {$\scriptstyle 4$};
  \draw (-0.25,0.5) node (X81) {$\scriptstyle 1$};
  \draw (0.25,0.5) node (X83) {$\scriptstyle 3$};
  \draw (0.5,1) node (X82) {$\scriptstyle 2$};
\draw[ext] (X84) -- (X81);
\draw[ext] (X84) -- (X83);
\draw[ext] (X83) -- (X82);
 \end{tikzpicture}
};

\draw (20,14) node (X9) { \small
 \begin{tikzpicture} [scale=0.7,yscale=1,ext/.style={black,shorten <=-1pt, shorten >=-1pt}]
  \draw (0,0) node (X93) {$\scriptstyle 3$};
  \draw (0.25,0.5) node (X92) {$\scriptstyle 2$};
  \draw (0.5,1) node (X91) {$\scriptstyle 1$};
  \draw (0.75,0.5) node (X94) {$\scriptstyle 4$};
\draw[ext] (X93) -- (X92);
\draw[ext] (X92) -- (X91);
\draw[ext] (X91) -- (X94);
 \end{tikzpicture}
};

\draw (30,14) node (XX7) { \small
 \begin{tikzpicture} [scale=0.7,yscale=1,ext/.style={black,shorten <=-1pt, shorten >=-1pt}]
  \draw (0,0) node (XX74) {$\scriptstyle 4$};
  \draw (0.5,0) node (XX72) {$\scriptstyle 2$};
  \draw (-0.25,0.5) node (XX73) {$\scriptstyle 3$};
  \draw (0.25,0.5) node (XX71) {$\scriptstyle 1$};
\draw[ext] (XX74) -- (XX73);
\draw[ext] (XX74) -- (XX71);
\draw[ext] (XX72) -- (XX71);
 \end{tikzpicture}
};

\draw (X1.south west) node {$\scriptstyle T_2$};
\draw (XX1.south west) node {$\scriptstyle T_2$};
\draw ($(X6.south)+(-0.5,-0.3)$) node {$\scriptstyle T_3$};

\draw[dashed] (X1.north) -- (X7.south);
\draw[dashed] (X7.north) -- ($(X7.north)+(0,3)$);
\draw[dashed] (XX1.north) -- (XX7.south);
\draw[dashed] (XX7.north) -- ($(XX7.north)+(0,3)$);

\draw[quivarrow] (X1.30) -- (X4.210);
\draw[quivarrow] (X4.330) -- (X2.150);
\draw[quivarrow] (X2.30) -- (X5.210);
\draw[quivarrow] (X5.330) -- (X3.150);
\draw[quivarrow] (X3.30) -- (X6.210);
\draw[quivarrow] (X6.330) -- (XX1.150);
\draw[quivarrow] (X7.330) -- (X4.150);
\draw[quivarrow] (X4.30) -- (X8.210);
\draw[quivarrow] (X8.330) -- (X5.150);
\draw[quivarrow] (X5.30) -- (X9.210);
\draw[quivarrow] (X9.330) -- (X6.150);
\draw[quivarrow] (X6.30) -- (XX7.210);

\draw (0,7) node (R2) {};
\draw (30,7) node (S2) {};

\draw[translate] (X1.east) -- (X2.west);
\draw[translate] (X2.east) -- (X3.west);
\draw[translate] (X3.east) -- (XX1.west);
\draw[translate] (R2) -- (X4.west);
\draw[translate] (X4.east) -- (X5.west);
\draw[translate] (X5.east) -- (X6.west);
\draw[translate] (X6.east) -- (S2);
\draw[translate] (X7.east) -- (X8.west);
\draw[translate] (X8.east) -- (X9.west);
\draw[translate] (X9.east) -- (XX7.west);

\end{tikzpicture}
$$
\caption{Part of the AR-quiver of $\field Q$-mod, where $Q$ is the quiver in~\eqref{e:examplequiver}.}
\label{f:tamehereditaryAR}
\end{figure}

We define $\Lambda=\Lambda_T=\End_{\C_Q}(T)$ to be the corresponding cluster-tilted algebra. For 
Example~\ref{ex:running}, $\Lambda$ is given by the quiver with relations shown
in Figure~\ref{f:tametiltedquiver} (we indicate how to compute such a quiver with relations explicitly for a similar 
example in Section~\ref{s:wild}). Note that this quiver can be obtained from $Q$ by mutating (in the sense 
of~\cite{FZ02}) first at $2$ and then at $3$.

\begin{figure}
$$
\begin{tikzpicture}[scale=0.8,baseline=(bb.base),quivarrow/.style={black, -latex},translate/.style={black, dotted},relation/.style={black, dotted, thick=2pt}]

\path (0,0) node (bb) {}; 

\draw (0,0) node (X1) {\small $1$};
\draw (2,0) node (X3) {\small $3$};
\draw (2,1.5) node (X2) {\small $2$};
\draw (1,-1.5) node (X4) {\small $4$};

\draw[quivarrow] (X2.south) -- (X3.north);
\draw[quivarrow] (X3.west) -- (X1.east);
\draw[quivarrow] (X4.north east) -- (X3.south west);
\draw[quivarrow] (0.15,-0.2) -- (0.85,-1.3);
\draw[quivarrow] (0,-0.34) -- (0.7,-1.44);

\draw[relation] (0.7,-0.2) to[out=200,in=135] (0.6,-0.6);
\draw[relation] (1.2,-0.2) to[out=340,in=45] (1.4,-0.6);
\draw[relation] (0.8,-0.95) to[out=305,in=240] (1.32,-0.85);

\draw (0.7,-0.68) node {\small $\ast$};

\end{tikzpicture}
$$
\caption{The endomorphism algebra $\End_{\C}(T)^{\opp}$ for the tilting module in Example~\ref{ex:running}.}
\label{f:tametiltedquiver}
\end{figure}
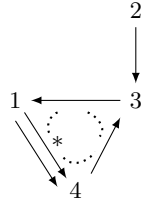

There is a natural functor $\Hom_{\C}(T,-)$ from $\C$ to $\Lambda$-mod. We have:

\begin{theorem} \cite[Thm.\ A]{BMR07} \label{t:equivalence}
The functor $\Hom_{\C}(T,-)$ induces
an equivalence from the additive quotient $\C/\add(\tau T)$ to
$\Lambda$-mod.
\end{theorem}

We denote the image of an object $X$ in $\C$ under
the functor $\Hom_{\C}(T,-)$ by $\clu{X}$.
We note the following:

\begin{proposition} \label{p:rigid}
Let $X$ be an object in $\C$ and $\XX$ the corresponding
$\Lambda$-module. Then
\begin{enumerate}[(a)]
\item
$\XX$ is Schurian if and only if
$$\Hom_{C/\add(\tau T)}(X,X)\cong \field.$$
\item
$\XX$ is rigid if and only if
$$\Hom_{\C/\add(\tau T\oplus \tau^2 T)}(X,\tau X)=0.$$
\end{enumerate}
\end{proposition}
\begin{proof}
Part (a) follows from the equivalence in Theorem~\ref{t:equivalence}. Part (b) follows from this combined
with the AR-formula~\eqref{e:ARformula}, noting that the
injective modules in $\Lambda$-mod are the objects in the
subcategory $\add \Hom_{\C}(T,\tau^2 T)$
(see~\cite{BMR07},~\cite[\S2]{kellerreiten07}).
\end{proof}

The following statement follows from~\cite[Thm.\ 4.1]{AIR14}.

\begin{theorem} \label{t:taurigid} \cite{AIR14}
The functor $\Hom_{\C}(T,-)$ induces a bijection between
isomorphism classes of indecomposable rigid objects in $\C$ which
are not summands of $\tau T$ and isomorphism classes of indecomposable
$\tau$-rigid $\Lambda$-modules.
\end{theorem}

Since a $\field Q$-module is rigid if and only if the induced object of
$\C$ is rigid (by~\cite[Prop.\ 1.7]{BMRRT06}), we have:

\begin{corollary} \label{c:taurigid}
If $X$ is a $\field Q$-module not in $\add(\tau T)$
then $X$ is rigid in $\field Q$-mod if and only if
$\clu{X}$ is $\tau$-rigid in $\Lambda$-mod.
\end{corollary}

Since (for modules over any finite-dimensional algebra) every $\tau$-rigid
module is rigid, we have that $\clu{X}$ is a rigid $\Lambda$-module
for any rigid $\field Q$-module $X$.

\begin{remark} \label{r:transjective}
Suppose that $X$ is an indecomposable object of $D^b(\field Q)$ which is either a preprojective
$\field Q$-module, a preinjective $\field Q$-module or the shift of a projective $\field Q$-module.
Assume also that $X$ is not a direct summand of $\tau T$.
Then $X$ is rigid in $D^b(\field Q)$, hence (by~\cite[Prop.\ 1.7]{BMRRT06}) rigid in $\C$.
By Theorem~\ref{t:taurigid}, $\clu{X}$ is $\tau$-rigid in $\Lambda$-mod.
Furthermore, $X$ is Schurian in $D^b(\field Q)$. We also have $\HomF{\C}(X,X)\cong \Dual\Hom(X,\tau^2 X)=0$
by~\eqref{e:tau2}, so $X$ is a Schurian object of $\C$. It follows that $\clu{X}$ is a Schurian
$\Lambda$-module by Proposition~\ref{p:rigid}(a). Thus we see that, for any
indecomposable transjective object of $\C$ (not a summand of $\tau T$),
the corresponding $\Lambda$-module is Schurian and $\tau$-rigid,
hence rigid. Thus the main work in classifying indecomposable Schurian and
($\tau$-)rigid $\Lambda$-modules concerns those which arise from tubes in $\field Q$-mod.
\end{remark}

Finally, we include some lemmas which will be useful for
checking whether a given $\Lambda$-module is Schurian or rigid.

\begin{lemma}
\label{l:Ccombine}
Let $X,Y,Z$ be ${\field}Q$-modules,
regarded as objects in $\C$.
Let $f\in \HomF{\C}(X,Y)=\Hom(X,FY)$.
Then $f$ factorizes in $\C$ through $Z$ if and only if
it factorizes in $D^b(\field Q)$ through $Z$ or $F(Z)$.
\end{lemma}
\begin{proof}
Since $f$ is an $F$-map, it can only factorize through
$Z$ in $\C$ as an $H$-map followed by an $F$-map or an
$F$-map followed by an $H$-map.
The former case corresponds to factorizing through $Z$
in $D^b(\field Q)$ and the latter case corresponds to factorizing
through $F(Z)$ in $D^b(\field Q)$.
\end{proof}

\begin{proposition}
\label{p:dual}
Let $A,B,C$ be objects in $D^b(\field Q)$.
\begin{enumerate}[(a)]
\item
Let $\alpha:A\rightarrow C$ and
$$\Hom(B,\tau \alpha):\Hom(B,\tau A)\rightarrow \Hom(B,\tau C)$$
and
$$\Hom(\alpha,B[1]):\Hom(C,B[1])\rightarrow \Hom(A,B[1])$$
Then $\Hom(B,\tau \alpha)$ is nonzero
(respectively, injective, surjective, or an isomorphism)
if and only if $\Hom(\alpha,B[1])$ is nonzero (respectively,
surjective, injective or an isomorphism). We illustrate the
maps $\Hom(B,\tau \alpha)$ and $\Hom(\alpha,B[1])$ below for
ease of reference.
$$
\xymatrix{
B \ar[dr]_{\gamma} \ar[rr]^{\Hom(B,\tau\alpha)(\gamma)} && \tau C \\
& \tau A \ar[ur]_{\tau\alpha}
}
\quad\quad
\xymatrix{
A \ar[dr]_{\alpha} \ar[rr]^{\Hom(\alpha,B[1])(\delta)} && B[1] \\
& C \ar[ur]_{\delta}
}
$$
\item
Let $\beta:C\rightarrow B$ and consider the induced maps:
$$\Hom(\beta,\tau A):\Hom(B,\tau A)\rightarrow \Hom(C,\tau A)$$
and
$$\Hom(A,\beta[1]):\Hom(A,C[1])\rightarrow \Hom(A,B[1])$$
Then $\Hom(\beta,\tau A)$ is nonzero
(respectively, injective, surjective, or an isomorphism)
if and only if $\Hom(A,\beta[1])$ is nonzero (respectively,
surjective, injective or an isomorphism).
We illustrate the maps $\Hom(\beta,\tau A)$ and $\Hom(A,\beta[1])$ below for ease of
reference.
$$
\xymatrix{
C \ar[dr]_{\beta} \ar^{\Hom(\beta,\tau A)(\gamma)}[rr] && \tau A \\
& B \ar[ur]_{\gamma}
}
\quad\quad
\xymatrix{
A \ar[dr]_{\delta} \ar[rr]^{\Hom(A,\beta[1])(\delta)} && B[1] \\
& C[1] \ar[ur]_{\beta[1]}
}
$$

\end{enumerate}
\end{proposition}
\begin{proof}
Part (a) follows from the commutative diagram:
$$
\xymatrix{
\Hom(C,B[1]) \ar[r]_(0.55){\sim} \ar[d]_{\Hom(\alpha,B[1])} & \Ext(C,B) \ar_(0.45){\sim}[r] & D\Hom(B,\tau C) \ar[d]^{D\Hom(B,\tau\alpha)} \\
\Hom(A,B[1]) \ar[r]_(0.55){\sim} & \Ext(A,B) \ar[r]_(0.45){\sim} &
D\Hom(B,\tau A).
}
$$
Part (b) follows from the commutative
diagram:
$$
\xymatrix{
\Hom(A,C[1]) \ar[r]_(0.55){\sim} \ar[d]_{\Hom(A,\beta[1])} & \Ext(A,C) \ar_(0.45){\sim}[r] & D\Hom(C,\tau A) \ar[d]^{D\Hom(\beta,\tau A)} \\
\Hom(A,B[1]) \ar[r]_(0.55){\sim} & \Ext(A,B) \ar[r]_(0.45){\sim} &
D\Hom(B,\tau A)
}
$$
\end{proof}

\begin{proposition} \label{p:Cfactor}
Let $A,B$ and $C$ be indecomposable $\field Q$-modules and suppose that \linebreak $\Hom(A,B[1])\cong \field$.
Let $\varepsilon:A\rightarrow B[1]$ be a nonzero map.
\begin{itemize}
\item[(a)]
The map $\varepsilon$ factors through $C$ if and only if
there is a map $\alpha\in \Hom(A,C)$ such that $\Hom(B,\tau \alpha)\not=0$.
\item[(b)]
The map $\varepsilon$ factors through $C[1]$ if and only if
there is a map $\beta\in\Hom(C,B)$ such that $\Hom(\beta,\tau A)\not=0$.
\end{itemize}
\end{proposition}
\begin{proof}
Since $\Hom(A,B[1])\cong \field$, the map $\varepsilon$ factors through
$C$ if and only if $\Hom(\alpha,B[1])\not=0$ for some
$\alpha\in \Hom(A,C)$. Part (a) then follows from Proposition~\ref{p:dual}(a).
Similarly, $\varepsilon$ factors through $C[1]$ if and only
if $\Hom(A,\beta[1])\not=0$ for some $\beta\in \Hom(C,B)$.
Part (b) then follows from Proposition~\ref{p:dual}(b).
\end{proof}

\section{Tubes}
\label{s:tube}
In this section we recall some facts concerning tubes in $\field Q$-mod.
We fix such a tube $\T$, of rank $r$.
Note that $\T$ is standard, i.e.\ the subcategory
of $\T$ consisting of the indecomposable objects is equivalent to the
mesh category of the AR-quiver of $\T$.

Let $Q_i$, for $i\in \mathbb{Z}_r$ be the quasisimple modules in $\T$.
Then, for each $i\in \mathbb{Z}_r$ and $l\in \mathbb{N}$, there
is an indecomposable module $M_{i,l}$ in $\T$ with socle
$Q_i$ and quasilength $l$; these modules exhaust the
indecomposable modules in $\T$. For $i\in \mathbb{Z}$ and $l\in \mathbb{N}$,
we define $Q_i=Q_{[i]_r}$ and $M_{i,l}=M_{[i]_r,l}$. Note that the socle of
$M_{i,l}$ is $M_{i,1}$. We denote the quasilength
$l$ of a module $M=M_{i,l}$ by $\ql(M)$.
The AR-quiver of $\T$
is shown in Figure~\ref{f:arquiver3} (for the case $r=3$).

\begin{figure}
$$
\begin{tikzpicture}[scale=1,baseline=(bb.base),quivarrow/.style={black, -latex, shorten >=10pt,shorten <=10pt},translate/.style={black, dotted, shorten >=13pt, shorten <=13pt}]
  
\path (0,0) node (bb) {}; 

\draw (0,0) node {$M_{0,1}$};
\draw (2,0) node {$M_{1,1}$};
\draw (4,0) node {$M_{2,1}$};
\draw (6,0) node {$M_{0,1}$};
\draw (1,1) node {$M_{0,2}$};
\draw (3,1) node {$M_{1,2}$};
\draw (5,1) node {$M_{2,2}$};
\draw (0,2) node {$M_{2,3}$};
\draw (2,2) node {$M_{0,3}$};
\draw (4,2) node {$M_{1,3}$};
\draw (6,2) node {$M_{2,3}$};

\draw[quivarrow] (0,0) -- (1,1);
\draw[quivarrow] (2,0) -- (3,1);
\draw[quivarrow] (4,0) -- (5,1);
\draw[quivarrow] (1,1) -- (2,0);
\draw[quivarrow] (3,1) -- (4,0);
\draw[quivarrow] (5,1) -- (6,0);
\draw[quivarrow] (1,1) -- (2,2);
\draw[quivarrow] (3,1) -- (4,2);
\draw[quivarrow] (5,1) -- (6,2);
\draw[quivarrow] (0,2) -- (1,1);
\draw[quivarrow] (2,2) -- (3,1);
\draw[quivarrow] (4,2) -- (5,1);

\draw[translate] (0,0) -- (2,0);
\draw[translate] (2,0) -- (4,0);
\draw[translate] (4,0) -- (6,0);
\draw[translate] (1,1) -- (3,1);
\draw[translate] (3,1) -- (5,1);
\draw[translate] (0,2) -- (2,2);
\draw[translate] (2,2) -- (4,2);
\draw[translate] (4,2) -- (6,2);

\draw[dotted,shorten >=13pt] (0,1) -- (1,1);
\draw[dotted,shorten <=13pt] (5,1) -- (6,1);

\draw[dashed,shorten <=10pt, shorten >=10pt] (0,0) -- (0,2);
\draw[dashed,shorten <=10pt, shorten >=10pt] (6,0) -- (6,2);
\draw[dashed,shorten <=10pt, shorten >=10pt] (0,2) -- (0,4);
\draw[dashed,shorten <=10pt, shorten >=10pt] (6,2) -- (6,4);

\end{tikzpicture}
$$
\caption{The AR-quiver of a tube of rank $3$}
\label{f:arquiver3}
\end{figure}
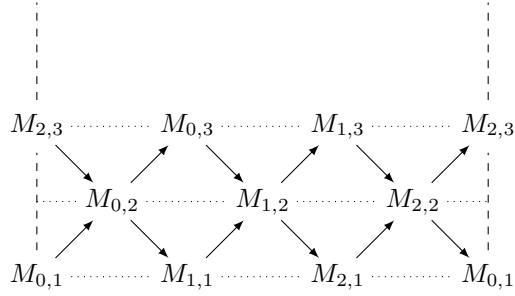

\begin{lemma}
\label{l:downwardarrows}
Let $X$ be an object in $\T$ of quasilength $\ell$.
Then any path in the AR-quiver of $\T$ with at least $\ell$ downward
arrows must be zero in $\field Q$-mod.
\end{lemma}
\begin{proof}
By applying mesh relations if necessary, we can rewrite
the path as a product of $\ell-1$ downward arrows (the
maximum number possible) followed by an upward arrow and
a downward arrow (followed, possibly, by more arrows). Hence the path is zero.
\end{proof}

The following is well-known.

\begin{lemma} \label{l:homlow}
Let $M_{i,l}$, with $0\leq l\leq r-1$, and
$M_{j,m}$ be objects in $\T$.
Then we have the following: (see Figure~\ref{f:hammock} for
an example).

\begin{itemize}
\renewcommand*{\arraystretch}{0.2}
\item[(a)]
$$\Hom(M_{i,l},M_{j,m})\cong \begin{cases}
\field, & \parbox{10cm}{if $1\leq m\leq l-1$ and $j$ is congruent to a member of \\ $[i+l-m,i+l-1]$ mod $r$;} \\[0.6em]
\field, & \text{if }m\geq l\text{ and $j$ is congruent to a member of $[i,i+l-1]$ mod $r$;} \\[0.2em]
0, & otherwise.
\end{cases}
$$
\item[(b)]
$$\Hom(M_{j,m},M_{i,l})\cong \begin{cases}
\field, & \parbox{10cm}{if $1\leq m\leq l-1$ and $j$ is congruent to a member of $[i-m+1,i]$ mod $r$;} \\[0.8em]
\field, & \parbox{10cm}{if $m\geq l$ and $j$ is congruent to a member of $[i-m+1,i-m+l]$ mod $r$;} \\[0.3em]
0, & otherwise.
\end{cases}
$$
\end{itemize}
\end{lemma}

\begin{proof}
We first consider part (a). Note that, since the quasilength of $M_{i,l}$ is assumed to be at most
$r$, the rays starting at $M_{i+p,l-p}$ for $0\leq p\leq l-1$ do not
intersect each other. It is then easy to see that, up to mesh relations, there is exactly one
path in the AR-quiver of $\T$ from $M_{i,l}$ to the objects in these rays
and no path to any other object in $\T$. The result then follows from
the fact that $\T$ is standard. A similar proof gives part (b).
\end{proof}

Let $M_{i,l}$ be an indecomposable module in $\T$. The \emph{wing}
$\W_{M_{i,l}}$ of $M_{i,l}$ is given by:
$$\W_{M_{i,l}}=\{M_{j,m}\,:\, i\leq j\leq i+l-1,\ 1\leq m\leq l+i-j\}.$$
Now fix $M_{i,l}\in \T$ with $l\leq r$.  It follows from Lemma~\ref{l:homlow} that 
if the quasisocle of $X\in \T$ does not lie in $\W_{M_{i,l}}$ then $\Hom(M_{i,l},X)=0$.
Similarly, if the quasitop of $X$ does not lie in $\W_{M_{i,l}}$ then
$\Hom(X,M_{i,l})=0$.
This implies the following, which we state here as we shall use it often.

\begin{corollary} \label{c:wingzero}
Let $M,N,X$ be indecomposable objects in $\T$, and suppose that
$M$ has quasilength at most $r$, and $M\in \W_N$.
\begin{itemize}
\item[(a)] If the quasisocle of $X$ does not lie in $\W_N$ then $\Hom(M,X)=0$.
\item[(b)] If the quasitop of $X$ does not lie in $\W_N$ then $\Hom(X,M)=0$.
\end{itemize}
\end{corollary}

\begin{figure}
$$
\begin{tikzpicture}[scale=0.4,baseline=(bb.base),quivarrow/.style={black, -latex, shorten >=10pt,shorten <=10pt},translate/.style={black, dotted, shorten >=13pt, shorten <=13pt}]

\newcommand{\rank}{5}
\newcommand{\height}{5}

\pgfmathparse{\rank-1}\let\rankm\pgfmathresult;
\pgfmathparse{\height-1}\let\heightm\pgfmathresult;

\path (0,0) node (bb) {}; 

\foreach \i in {0,...,\rank}{
\foreach \j in {0,...,\height}{
\node [fill=white,inner sep=0pt] at (\i*2,\j*2) {$\circ$};}}

\foreach \i in {0,...,\rankm}{
\foreach \j in {0,...,\heightm}{
\draw[->,shorten <=4pt,shorten >=4pt] (\i*2,\j*2) -- (\i*2+1,\j*2+1);
}}

\foreach \i in {0,...,\rankm}{
\foreach \j in {1,...,\height}{
\draw[->,shorten <=4pt,shorten >=4pt] (\i*2,\j*2) -- (\i*2+1,\j*2-1);
}}

\foreach \i in {0,...,\rankm}{
\foreach \j in {0,...,\heightm}{
\node [fill=white,inner sep=0pt] at (\i*2+1,\j*2+1) {$\circ$};}}

\foreach \i in {0,...,\rankm}{
\foreach \j in {0,...,\heightm}{
\draw[->,shorten <=4pt, shorten >=4pt] (\i*2+1,\j*2+1) -- (\i*2+2,\j*2+2);
\draw[->,shorten <=4pt, shorten >=4pt] (\i*2+1,\j*2+1) -- (\i*2+2,\j*2);
}}

\begin{scope}[on background layer]
\draw[dashed] (0,0) -- (0,2*\height+2);
\draw[dashed] (2*\rank,0) -- (2*\rank,2*\height+2);

\draw[draw=none,fill=gray!30,opacity=0.4] (4,2) -- (6,0) -- (10,4) -- (10,8) -- (4,2);

\draw[draw=none,fill=gray!30,opacity=0.4] (0,4) -- (6,10) -- (2,10) -- (0,8) -- (0,4);
\end{scope}

\draw (4,2) node {$\bullet$};

\draw (7,-1) node {$\scriptstyle M_{i+l-1,1}$};

\end{tikzpicture}
\quad\quad\quad\quad
\begin{tikzpicture}[scale=0.4,baseline=(bb.base),quivarrow/.style={black, -latex, shorten >=10pt,shorten <=10pt},translate/.style={black, dotted, shorten >=13pt, shorten <=13pt}]

\path (0,0) node (bb) {}; 

\newcommand{\rank}{5}
\newcommand{\height}{5}

\pgfmathparse{\rank-1}\let\rankm\pgfmathresult;
\pgfmathparse{\height-1}\let\heightm\pgfmathresult;

\foreach \i in {0,...,\rank}{
\foreach \j in {0,...,\height}{
\node [fill=white,inner sep=0pt] at (\i*2,\j*2) {$\circ$};}}

\foreach \i in {0,...,\rankm}{
\foreach \j in {0,...,\heightm}{
\draw[->,shorten <=4pt,shorten >=4pt] (\i*2,\j*2) -- (\i*2+1,\j*2+1);
}}

\foreach \i in {0,...,\rankm}{
\foreach \j in {1,...,\height}{
\draw[->,shorten <=4pt,shorten >=4pt] (\i*2,\j*2) -- (\i*2+1,\j*2-1);
}}

\foreach \i in {0,...,\rankm}{
\foreach \j in {0,...,\heightm}{
\node [fill=white,inner sep=0pt] at (\i*2+1,\j*2+1) {$\circ$};}}

\foreach \i in {0,...,\rankm}{
\foreach \j in {0,...,\heightm}{
\draw[->,shorten <=4pt, shorten >=4pt] (\i*2+1,\j*2+1) -- (\i*2+2,\j*2+2);
\draw[->,shorten <=4pt, shorten >=4pt] (\i*2+1,\j*2+1) -- (\i*2+2,\j*2);
}}

\begin{scope}[on background layer]
\draw[dashed] (0,0) -- (0,2*\height+2);
\draw[dashed] (2*\rank,0) -- (2*\rank,2*\height+2);

\draw[draw=none,fill=gray!30,opacity=0.4] (4,2) -- (2,0) -- (0,2) -- (0,6) -- (4,2);

\draw[draw=none,fill=gray!30,opacity=0.4] (10,6) -- (6,10) -- (10,10) -- (10,6);
\end{scope}

\draw (4,2) node {$\bullet$};

\draw (3,-1) node {$\scriptstyle M_{i-l+1,1}$};

\end{tikzpicture}
$$
\caption{The left hand figure shows the modules $X$ in $\T$ for which $\Hom(M_{i,l},X)\not=0$ (in the shaded region), for the
case $r=5$. The module $M_{i,l}$ is denoted by a filled-in circle. The right hand figure shows the modules $X$ with $\Hom(X,M_{i,l})\not=0$.}
\label{f:hammock}
\end{figure}
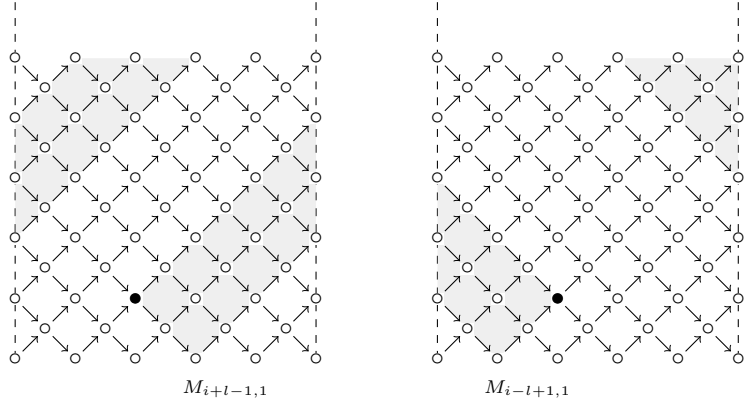

\begin{lemma} \label{l:dimhom}
Let $M_{i,l}$ be an indecomposable module in $\T$. Then we have:
$$\dim \End(M_{i,l})=
\begin{cases} 1, & 1\leq l\leq r; \\
2, & r+1\leq l\leq 2r;
\end{cases}
\quad\quad
\dim \Hom(M_{i,l},\tau M_{i,l})=
\begin{cases} 0, & 1\leq l\leq r-1; \\
1, & r\leq l\leq 2r-1;
\end{cases}
$$
$$
\dim \Hom(M_{i,l},\tau^2 M_{i,l})=
\begin{cases} 0, & 1\leq l\leq r-2; \\
1, & r-1\leq l\leq 2r-2.
\end{cases}
$$
\end{lemma}
\begin{proof}
The formulas are easily checked using the fact that $\T$ is standard.
\end{proof}

The last lemma in this section also follows from the fact that $\T$ is standard
(since the mesh relations are homogeneous).

\begin{lemma} \label{l:basis}
Let $X,Y$ be indecomposable objects in $\T$, and let
$\pi_1(X,Y),\ldots ,\pi_t(X,Y)$ be representatives
for the paths in $\T$ from $X$ to $Y$ up to equivalence
via the mesh relations. Then the corresponding maps
$f_1(X,Y),\ldots ,f_t(X,Y)$ form a basis for $\Hom(X,Y)$.
\end{lemma}

\section{Properties of $T$ with respect to a tube}
\label{s:tubeproperties}
In this section, we collect together some useful facts
that we shall use in Section~\ref{s:rigid} to determine
the rigid and Schurian $\Lambda$-modules.

Recall that we have fixed a tube $\T$ in $\field Q$-mod of rank $r$.
Let $T_{\T}$ be the direct sum of the indecomposable summands of $T$ lying in
$\T$ (we include the case $T_{\T}=0$).
Let $T_k$, $k\in \mathbb{Z}_s$ be the indecomposable summands of
$T_{\T}$ which are not contained in the wing of any other 
indecomposable summand of $T_{\T}$,
numbered in order cyclically around $\T$.
The indecomposable summands of $T_{\T}$
are contained in $\cup_{k\in \mathbb{Z}_s} \W_{T_k}$, where $\W_{T_k}$ 
denotes the wing of $T_k$.
Note that if $T_{\T}=0$ then $s=0$ and $\mathbb{Z}_s$ is the empty set.

A key role is played by the modules $\tau T_k$.
Let $i_k\in \{0,1,\ldots ,r-1\}$ and $l_k\in \mathbb{N}$ be integers such that
$$\tau T_k\cong M_{i_k,l_k}.$$
Note that $l_k\leq r-1$, since $T_k$ is rigid.
Then we have the wings
\begin{align}
\label{e:wings}
\W_{T_k} &=\{M_{i,l}\,:\,i_k+1\leq i\leq i_k+l_k,1\leq l\leq l_k+i_k+1-i\}, \\
\W_{\tau T_k} &=\{M_{i,l}\,:\,i_k\leq i\leq i_k+l_k-1,1\leq l\leq l_k+i_k-i\}, \\
\W_{\tau^2 T_k} &=\{M_{i,l}\,:\,i_k-1\leq i\leq i_k+l_k-2,1\leq l\leq l_k+i_k-1-i\},
\end{align}
For $k\in \mathbb{Z}_s$, the quasisimple objects in $\W_{\tau T_k}$ are
the $Q_i$ for $i_k\leq i\leq i_k+l_k-1$.
Note that, since $\Ext^1(T,T)=0$, we have $[i_{k+1}-(i_k+l_k-1)]_r\not=0,1$
(by Lemma~\ref{l:homlow} and the AR-formula).
In other words, two successive wings $\W_{\tau T_k}$ and $\W_{\tau T_{k+1}}$
are always separated by at least one quasisimple module.

For $k\in \mathbb{Z}_s$, we define $\Top_k$ to be the module $M_{i_k,r+l_k}$.
Note that $\Top_k$ is the module of smallest quasilength in the intersection
of the ray through the injective objects in $\W_{\tau T_k}$ and the coray
through the projective objects in $\W_{\tau T_k}$. Let $\HH_k$ be the
part of $\W_{\Top_k}$ consisting of injective or projective objects in
$\W_{\Top_k}$ of quasilength at least $r$. So
\begin{equation}
\HH_k=\{M_{i_k,l}\,:\,r\leq l\leq r+l_k\}\cup \{M_{i_k+p,r+l_k-p}\,:\,0\leq p\leq l_k\}.
\label{e:hk}
\end{equation}
The unique object in both of these sets is $\Top_k=M_{i_k,r+l_k}$, the unique projective-injective
object in $\W_{\Top_k}$.

Let $\R_k$ (respectively, $\wideR{k}$) be the part of $\W_{\Top_k}$ consisting of
non-projective, non-injective objects in $\W_{\Top_k}$ of quasilength at least $r$ (respectively, at least
$r-1$). Note that $\R_k\subseteq \wideR{k}$.
We have:
\begin{align}
\label{e:defrk}
\R_k &= \{M_{i,l}\,:\,i_k+1\leq i\leq i_k+l_k-1,\,r\leq l\leq r+l_k+i_k-i-1\}; \\
\wideR{k} &= \{M_{i,l}\,:\,i_k+1\leq i\leq i_k+l_k,\,r-1\leq l\leq r+l_k+i_k-i-1\}.
\end{align}
An example is shown in Figure~\ref{f:wings}.

\begin{lemma} \label{l:rktop}
The quasisocles of the indecomposable objects in $\R_k$
(respectively, $\wideR{k}$) are the $Q_i$ where
$i_k+1\leq i\leq i_k+l_k-1$ (respectively, $i_k+1\leq i\leq i_k+l_k$)
and the quasitops are the $Q_i$ where $i_k\leq i\leq i_k+l_k-2$ (respectively,
$i_k-1\leq i\leq i_k+l_k-2$).
In particular, the quasisocle of an indecomposable object in $\R_k$ lies in
$\W_{T_k}\cap \W_{\tau T_k}$ (respectively, in $\W_{T_k}$).
The quasitop of an indecomposable object in $\R_k$ (respectively, $\wideR{k}$)
lies in $\W_{\tau T_k}\cap \W_{\tau^2 T_k}$ (respectively, $\W_{\tau^2 T_k}$).
\end{lemma}
\begin{proof}
The first statement follows from~\eqref{e:defrk}. The quasitop of $M_{i,l}$ is
$Q_{i+l-1}$. Hence, the quasitops of the indecomposable objects in $\R_k$ are the
$Q_i$ with
$$(i_k+1)+r-1\leq i\leq (i_k+l_k-1)+(r+l_k+i_k-(i_k+l_k-1)-1)-1,$$
i.e.
$$i_k+r\leq i\leq r+l_k+i_k-2.$$
i.e.\ the $Q_i$ with $i_k\leq i\leq i_k+l_k-2$, since we are working mod $r$.
Similarly, the quasitops of the indecomposable objects in $\wideR{k}$ are the
$Q_i$ with
$$(i_k+1)+(r-1)-1\leq i\leq (i_k+l_k)+(r+l_k+i_k-(i_k+l_k)-1)-1,$$
i.e.
$$i_k+r-1\leq i\leq r+l_k+i_k-2,$$
i.e.\ the $Q_i$ with
$i_k-1\leq i\leq i_k+l_k-2$.
The last statements follow from the descriptions of the wings
$\W_{T_k}$, $\W_{\tau T_k}$ and $\W_{\tau^2 T_k}$ above~\eqref{e:wings}.
It is easy to observe the result in this lemma in Figure~\ref{f:wings}, where
the regions $\R_k$ and $\wideR{k}$ are indicated.
\end{proof}

\begin{figure}
$$
\begin{tikzpicture}[xscale=0.3,yscale=0.4,baseline=(bb.base),quivarrow/.style={black, -latex, shorten >=10pt,shorten <=10pt},translate/.style={black, dotted, shorten >=13pt, shorten <=13pt}]

\newcommand{\rank}{11}
\newcommand{\height}{7}
\newcommand{\copies}{2}

\pgfmathparse{\rank-1}\let\rankm\pgfmathresult;
\pgfmathparse{\height-1}\let\heightm\pgfmathresult;

\pgfmathparse{\rank*\copies}\let\totwidth\pgfmathresult;
\pgfmathparse{\rank*\copies-1}\let\totwidthm\pgfmathresult;

\path (0,0) node (bb) {}; 

\foreach \i in {0,...,\totwidth}{
\foreach \j in {0,...,\height}{
\node [fill=white,inner sep=0pt] at (\i*2,\j*2) {$\circ$};}}

\foreach \i in {0,...,\totwidthm}{
\foreach \j in {0,...,\heightm}{
\node [fill=white,inner sep=0pt] at (\i*2+1,\j*2+1) {$\circ$};}}


\begin{scope}[on background layer]
\draw[dashed] (0,0) -- (0,2*\height+2);
\draw[dashed] (2*\rank,0) -- (2*\rank,2*\height+2);
\draw[dashed] (2*\totwidth,0) -- (2*\totwidth,2*\height+2);
\end{scope}


\draw[draw=none,fill=gray!30,opacity=0.4] (4,0) -- (7,3) -- (10,0) -- (4,0);
\draw[draw=none,fill=gray!30,opacity=0.4] (14,0) -- (16,2) -- (18,0) -- (14,0);

\draw[draw=none,fill=gray!30,opacity=0.4] (4+2*\rank,0) -- (7+2*\rank,3) -- (10+2*\rank,0) -- (4+2*\rank,0);
\draw[draw=none,fill=gray!30,opacity=0.4] (14+2*\rank,0) -- (16+2*\rank,2) -- (18+2*\rank,0) -- (14+2*\rank,0);



\draw (1.5,8.5) -- (5,12) -- (8.5,8.5) -- (1.5,8.5);
\draw (2.5,9.5) -- (7.5,9.5);

\draw (23.5,8.5) -- (27,12) -- (30.5,8.5) -- (23.5,8.5);
\draw (24.5,9.5) -- (29.5,9.5);


\draw (13.5,8.5) -- (18,13) -- (22.5,8.5) -- (13.5,8.5);
\draw (14.5,9.5) -- (21.5,9.5);

\draw (35.5,8.5) -- (40,13) -- (44.5,8.5) -- (35.5,8.5);
\draw (36.5,9.5) -- (43.5,9.5);

\begin{scope}[on background layer]

\draw[dotted] (4,0) -- (18,14) -- (32,0);
\draw[dotted] (14,0) -- (27,13) -- (40,0);

\draw[dotted] (0,8) -- (5,13) -- (18,0);
\draw[dotted] (26,0) -- (40,14) -- (44,10);

\draw[dotted] (10,0) -- (0,10);
\draw[dotted] (36,0) -- (44,8);
\end{scope}


\draw (7,3.5) node[black] {$\scriptstyle \tau T_0$};
\draw (29,3.5) node[black] {$\scriptstyle \tau T_0$};

\draw (16,2.5) node[black] {$\scriptstyle \tau T_1$};
\draw (38,2.5) node[black] {$\scriptstyle \tau T_1$};

\draw (5,13.5) node[black] {$\scriptstyle \text{Top}_1$};
\draw (27,13.5) node[black] {$\scriptstyle \text{Top}_1$};

\draw (18,14.5) node[black] {$\scriptstyle \text{Top}_0$};
\draw (40,14.5) node[black] {$\scriptstyle \text{Top}_0$};



\foreach \i/\j in {2/10,3/11,4/12,5/13,6/12,7/11,8/10}
{\draw (\i,\j) node {$\bullet$};}


\foreach \i/\j in {14/10,15/11,16/12,17/13,18/14,19/13,20/12,21/11,22/10}
{\draw (\i,\j) node {$\bullet$};}


\foreach \i/\j in {24/10,25/11,26/12,27/13,28/12,29/11,30/10}
{\draw (\i,\j) node {$\bullet$};}


\foreach \i/\j in {36/10,37/11,38/12,39/13,40/14,41/13,42/12,43/11,44/10}
{\draw (\i,\j) node {$\bullet$};}

\end{tikzpicture}
$$
\caption{The wings $\W_{\tau T_k}$ shown as shaded regions in two copies of $\T$ in the
case $r=11$. The elements in the $\HH_k$ are drawn as filled dots.
The elements in the regions $\R_k$ and $\wideR{k}$ are enclosed in
triangles.}
\label{f:wings}
\end{figure}
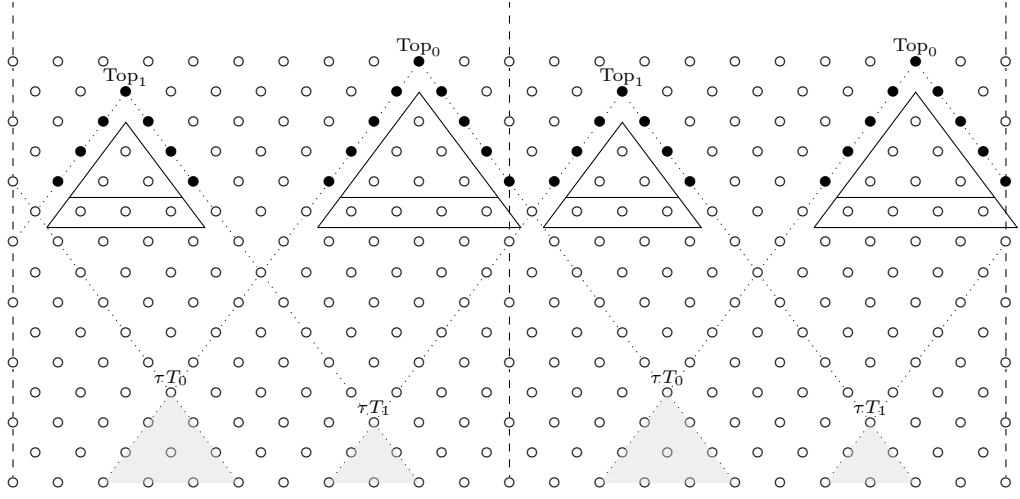

Recall that $U$ denotes the maximal preprojective direct summand of $T$.

\begin{lemma} \label{l:additive}
Let $k\in \mathbb{Z}_s$ and let $X$ be an indecomposable object in $\W_{\tau T_k}$.
Then $\Hom(U,X)=0$.
\end{lemma}
\begin{proof}
Since $U$ is preprojective, $\Hom(U,-)$ is exact on short
exact sequences of modules in $\T$, so $\dim \Hom(U,-)$ is
additive on such sequences. This includes, in particular,
almost split sequences in $\T$, and it follows that:
\begin{equation}
\label{e:additive}
\dim \Hom(U,X)=\sum_{Y\in \W_X,Y\text{ quasisimple}}\dim \Hom(U,Y).
\end{equation}
Since $\Hom(U,\tau T_k)=0$, we must have
$\Hom(U,Y)=0$ for all quasisimple modules in $\W_{\tau T_k}$,
and the result now follows from~\eqref{e:additive}.
\end{proof}

\begin{lemma} \label{l:outsidehelp}
Suppose that $Y\in \ind(\T)$ satisfies $\Hom_{\C}(T_{\T},Y)=0$ and
$\Hom(U,Y)=0$. Then $Y\in \cup_{k\in \Z_s} \W_{\tau T_k}$.
\end{lemma}
\begin{proof}
Suppose $Y$ satisfies the assumptions above.
Then, if $V$ is an indecomposable summand of $T$ in a tube distinct from $\T$,
we have $\Hom(V,Y)=0$ and $\Hom(Y,\tau^2 V)=0$, so $\Hom_{\C}(V,Y)=0$.
We also have that $\Hom(Y,\tau^2 U)=0$, since $\tau^2 U$ is preprojective,
so $\Hom_{\C}(U,Y)=0$. Hence, we have $\Hom_{\C}(T,Y)=0$, so $\Ext_{\C}(Y,\tau T)=0$.
Since $T$ (and hence $\tau T$) is a cluster-tilting object
in $\C$, this implies that $Y$ lies in $\add \tau T$ and therefore in
$\cup_{k\in \mathbb{Z}_s}\W_{\tau T_k}$ as required.
\end{proof}

\begin{proposition} \label{p:outsidenonzero}
Let $X$ be an indecomposable object in $\T$ not lying in $\cup_{k\in \mathbb{Z}_s}\W_{\tau T_k}$.
Then $\Hom(U,X)\not=0$.
\end{proposition}
\begin{proof}
Since $\dim\Hom(U,-)$ is additive on $\T$, we can assume that $X$ is
quasisimple. We assume, for a contradiction, that $\Hom(U,X)=0$.
If we can find a module $Y\in \T\setminus \cup_{k\in \mathbb{Z}_s}\W_{\tau T_k}$
such that $\Hom_{\C}(T_{\T},Y)=0$ and $\Hom(U,Y)=0$ then, by Lemma~\ref{l:outsidehelp},
we have a contradiction. We now construct such a module $Y$, considering various cases
for $X$.

\noindent \textbf{Case 1}: Assume that
$X\not\cong Q_{i_k-1}$ and $X\not\cong Q_{i_k+l_k}$ for any $k\in \mathbb{Z}_s$, i.e.
that $X$ is not immediately adjacent to any of the wings $\W_{\tau T_k}$, $k\in \Z_s$.
There is a single module of this kind in the example in Figure~\ref{f:wings};
this is denoted by $X_1$ in Figure~\ref{f:outsidenonzero}. In this case we take $Y=X$.

If $V$ is an indecomposable summand of $T_{\T}$, then $V\in \W_{T_k}$ for some
$k\in \mathbb{Z}_s$. Since the quasisimple module $X$
does not lie in $\cup_{k\in \mathbb{Z}_s} \W_{T_k}$, we have $\Hom(V,X)=0$
by Corollary~\ref{c:wingzero}.
Similarly, $\tau^2 V\in \W_{\tau^2 T_k}$ for some $k\in \mathbb{Z}_s$. Since
the quasitop of $X$ (i.e.\ $X$) does not lie in $\cup_{k\in \mathbb{Z}_s} \W_{T_k}$,
we have $\Hom(X,\tau^2 V)=0$ by Corollary~\ref{c:wingzero}.
Hence $\Hom_{\C}(T_{\T},X)=0$, completing this case.

We next suppose that $X\cong M_{i_k+l_k,1}$ for some
$k\in \mathbb{Z}_s$ (the case $X\cong M_{i_k-1,1}$ is similar).
Recall that there is always at least one quasisimple module between two
wings $\W_{\tau T_k}$ and $\W_{\tau T_{k\pm 1}}$.

\noindent \textbf{Case 2}:
Assume first that there are at least two quasisimple
modules between the wings $\W_{\tau T_k}$ and $\W_{\tau T_{k+1}}$, so
that $X$ is not adjacent to the wing $\W_{\tau T_{k+1}}$.
In the example in Figure~\ref{f:outsidenonzero}, the object $X_2$
is an example of this type (with $k=1$). In this case, we take $Y=M_{i_k,l_k+1}$
(indicated by $Y_2$ in Figure~\ref{f:outsidenonzero}).

By Lemma~\ref{l:additive}, $\Hom(U,M_{i_k+l_k-1,1})=0$.
By assumption, $\Hom(U,X)=0$.
Since $\dim \Hom(U,-)$ is additive on $\T$, we have $\Hom(U,M_{i_k+l_k-1,2})=0$.
If $l_k=1$ then $Y=M_{i_k+l_k-1,2}$ and $\Hom(U,Y)=0$.
If $l_k>1$ then, since $\dim \Hom(U,-)$ is additive on the short exact sequence:
$$0\rightarrow \tau T_k\rightarrow T'_k\oplus M_{i_k+l_k-1,1}\rightarrow M_{i_k+l_k-1,2}\rightarrow 0,$$
it follows that $\Hom(U,Y)=0$ in this case also.

Since the quasisocle $Q_{i_k}$ of $Y$ does not lie in
$\cup_{k'\in \mathbb{Z}_s} \W_{T_{k'}}$, we have $\Hom(V,Y)=0$ for
all indecomposable summands $V$ of $T_{\T}$ by Corollary~\ref{c:wingzero}.
Since there are at least two quasisimple modules between the wings
$\W_{\tau T_k}$ and $\W_{\tau T_{k+1}}$, the quasitop of $Y$ does not lie
in $\cup_{k'\in \mathbb{Z}_s} \W_{\tau^2 T_{k'}}$.
Hence $\Hom(Y,\tau^2 V)=0$ for all summands $V$ of $T_{\T}$,
by Corollary~\ref{c:wingzero}.
So $\Hom_{\C}(T_{\T},Y)=0$, completing this case.

\textbf{Case 3}:
We finally consider the case where
there is exactly one quasisimple module between the wings
$\W_{\tau T_k}$ and $\W_{\tau T_{k+1}}$.
In the example in Figure~\ref{f:outsidenonzero}, the object
$X_3$ is an example of this type. In this case, we take $Y=M_{i_k,l_k+l_{k+1}+1}$
(indicated by $Y_3$ in Figure~\ref{f:outsidenonzero}).

The quasisimples in $\W_Y$ are
the quasisimples in $\W_{\tau T_k}$, the quasisimples in $\W_{\tau T_{k+1}}$ and $X$.
For a quasisimple module $Q$ in one of the first two sets,
$\Hom(U,Q)=0$ by Lemma~\ref{l:additive}.
By assumption, $\Hom(U,X)=0$. Hence, arguing as in Lemma~\ref{l:additive}
and using the additivity of $\dim\Hom(U,-)$ on $\T$, we have $\Hom(U,Y)=0$.

Since the quasisocle of $Y$ is $Q_{i_k}$, which does not lie in
$\cup_{k'\in \mathbb{Z}_s}\W_{T_{k'}}$, we see that $\Hom(V,Y)=0$ for any
indecomposable summand of $\T_{\T}$ by Corollary~\ref{c:wingzero}.
Similarly, the quasitop of $Y$ is $Q_{i_{k+1}+l_{k+1}-1}$,
which does not lie in $\cup_{k'\in \mathbb{Z}_s}\W_{\tau^2 T_{k'}}$.
Hence $\Hom(Y,\tau^2 V)=0$ for any indecomposable summand of $T_{\T}$
by Corollary~\ref{c:wingzero}. So $\Hom_{\C}(T_{\T},Y)=0$, completing
this case.
\end{proof}

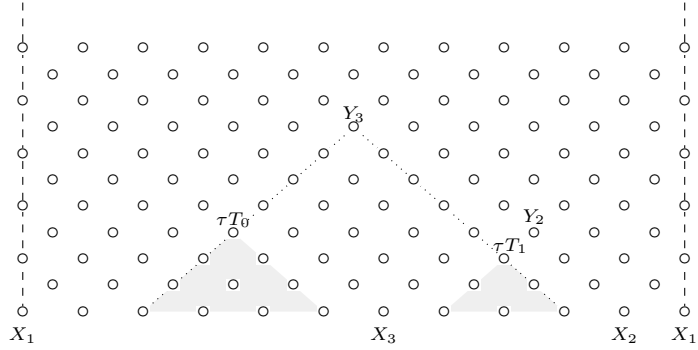
\begin{figure}
$$
\begin{tikzpicture}[xscale=0.4,yscale=0.35,baseline=(bb.base),quivarrow/.style={black, -latex, shorten >=10pt,shorten <=10pt},translate/.style={black, dotted, shorten >=13pt, shorten <=13pt}]

\newcommand{\rank}{11}
\newcommand{\height}{5}
\newcommand{\copies}{1}

\pgfmathparse{\rank-1}\let\rankm\pgfmathresult;
\pgfmathparse{\height-1}\let\heightm\pgfmathresult;

\pgfmathparse{\rank*\copies}\let\totwidth\pgfmathresult;
\pgfmathparse{\rank*\copies-1}\let\totwidthm\pgfmathresult;

\path (0,0) node (bb) {}; 

\foreach \i in {0,...,\totwidth}{
\foreach \j in {0,...,\height}{
\node [fill=white,inner sep=0pt] at (\i*2,\j*2) {$\circ$};}}

\foreach \i in {0,...,\totwidthm}{
\foreach \j in {0,...,\heightm}{
\node [fill=white,inner sep=0pt] at (\i*2+1,\j*2+1) {$\circ$};}}

\begin{scope}[on background layer]

\draw[dashed] (0,0) -- (0,2*\height+2);
\draw[dashed] (2*\rank,0) -- (2*\rank,2*\height+2);

\draw[draw=none,fill=gray!30,opacity=0.4] (4,0) -- (7,3) -- (10,0) -- (4,0);
\draw[draw=none,fill=gray!30,opacity=0.4] (14,0) -- (16,2) -- (18,0) -- (14,0);

\draw[dotted] (4,0) -- (11,7) -- (18,0);
\end{scope}

\draw (7,3.5) node[black] {$\scriptstyle \tau T_0$};
\draw (16.2,2.5) node[black] {$\scriptstyle \tau T_1$};

\draw (11,7.5) node[black] {$\scriptstyle Y_3$};
\draw (17,3.6) node[black] {$\scriptstyle Y_2$};

\draw (0,-0.9) node [black] {$\scriptstyle X_1$};
\draw (22,-0.9) node [black] {$\scriptstyle X_1$};
\draw (20,-0.9) node [black] {$\scriptstyle X_2$};
\draw (12,-0.9) node [black] {$\scriptstyle X_3$};

\end{tikzpicture}
$$
\caption{Proof of Proposition~\ref{p:outsidenonzero}. For the quasisimple module $X_1$, we take $Y=X_1$;
for the module $X_2$, we take $Y=Y_2$, and for the module $X_3$, we take $Y=Y_3$.}
\label{f:outsidenonzero}
\end{figure}

\begin{lemma} \label{l:onezero}
Let $P$ be an indecomposable projective $\field Q$-module, and
suppose that we have $\Hom(P,X_0)=0$ for some indecomposable module $X_0$ on the border of
$\T$. Then $\dim \Hom(P,X)\leq 1$ for all indecomposable modules $X$ on the border
of $\T$.
Furthermore, if there is some indecomposable module $X_1$ on the border of $\T$
such that $\Hom(P,X)=0$ for all indecomposable modules $X\not\cong X_1$ on the border of $\T$, then
$\dim \Hom(P,X_1)=1$.
\end{lemma}
\begin{proof}
This can be checked using the tables in~\cite[XIII.2]{simsonskowronski07}.
\end{proof}

\begin{proposition} \label{p:leq1}
Suppose that $T_{\T}\not=0$ and
let $X\in \T\setminus \cup_{k\in \mathbb{Z}_s}\W_{\tau T_k}$ be an
indecomposable module on
the border of $\T$ and $V$ an indecomposable summand of $U$. Then
$\dim\Hom(V,X)\leq 1$.
Furthermore, if $k=0$ and $T_0$ has quasilength $r-1$, then $\dim\Hom(V,X)=1$.
\end{proposition}
\begin{proof}
By applying a power of $\tau$ if necessary, we can assume that $V$ is projective.
By assumption, $\T$ contains a summand of $T$, so there is at least one
quasisimple module $X_0$ in $\cup_{k\in \mathbb{Z}_s} \W_{\tau T_k}$.
By Lemma~\ref{l:additive}, we have that $\Hom(V,X_0)=0$.
The first part of the lemma then follows from
Lemma~\ref{l:onezero}. If $k=0$ and $T_0$ has quasilength $r-1$, then
$\Hom(V,X_0)=0$ for every quasisimple in $\W_{\tau T_0}$. Since $X$ is the
unique quasisimple in $\T$ not in $\W_{\tau T_0}$, the second part now follows
from Lemma~\ref{l:onezero} also.
\end{proof}

Abusing notation, we denote the down-arrows in $\T$ by $x$ and the up-arrows by
$y$. So, for example, $x^r$ means the composition of $r$ down-arrows from
a given vertex.

\begin{proposition} \label{p:Hfactor}
Let $X=M_{i,l}$ be an indecomposable module in $\T$.
\begin{itemize}
\item[(a)] Suppose that $r+1\leq l$. Let $u_X=y^rx^r:X\rightarrow X$.
Then $u_X$ factors through $\add(\tau T_{\T})$ if and only
if $X\in \cup_{k\in \mathbb{Z}_s} \HH_k\cup \R_k$.
\item[(b)] Suppose that $r\leq l$. Let $v_X=y^{r-1}x^{r-1}:X\rightarrow \tau X$ be
the unique nonzero map (up to a scalar), as in Lemma~\ref{l:dimhom}.
Then $v_X$ factors through $\add(\tau T_{\T}\oplus \tau^2 T_{\T})$
if and only if $X\in \cup_{k\in \mathbb{Z}_s} (\HH_k\cup \R_k\setminus \{\Top_k\})$. Furthermore, $v_X$ factors through both
$\add(\tau T_{\T})$ and $\add(\tau^2 T_{\T})$ if and only if $X\in \cup_{k\in \mathbb{Z}_s} \R_k$.
\item[(c)] Suppose that $r\leq l$. Let $w_X=y^{r-2}x^{r-2}:\tau^{-1} X\rightarrow \tau X$.
Then $w_X$ factors through $\add(\tau T_{\T})$
if and only if $X\in \cup_{k\in \mathbb{Z}_s} \wideR{k}$.
\end{itemize}
\end{proposition}
\begin{proof}
We start with part (a). Note that $u_X$ lies in the basis for $\Hom(X,X)$
given in Lemma~\ref{l:basis}. Also, by the mesh relations, $u_X=x^ry^r$.
Let $D_X$ be the diamond-shaped region in $\T$ bounded by the paths
$x^ry^r$ and $y^rx^r$ starting at $X$. It is clear that
$u_X$ factors through any indecomposable module in $D_X$.
For an example, see Figure~\ref{f:DX}, where part of one copy of $D_X$ has been drawn.

If $Y$ lies outside $D_X$, then any path from
$X$ to $X$ in $\T$ via $Y$ must contain more than $r$ downward
arrows. By Lemma~\ref{l:basis} it is a linear combination of basis elements
distinct from $u$. So $u$ cannot factor through the direct sum of any
collection of objects outside this region.

Hence $u_X$ factors through $\add(\tau T_{\T})$ if and only if
some indecomposable summand of $\tau T_{\T}$ lies in $D_X$.
Since the indecomposable summands of $\tau T_{\T}$ lie in
$\cup_{k\in \mathbb{Z}_s} \W_{\tau T_k}$, we see that $u_X$
factors through $\tau T_{\T}$ if and only if $M_{i+r,l-r}$ (the
module in $D_X$ with minimal quasilength) lies in
$\cup_{k\in \mathbb{Z}_s} \W_{\tau T_k}$.

The corners of the triangular region $\HH_k\cup \R_k$ are
$M_{i_k,r}$, $\Top_k=M_{i_k,r+l_k}$ and $M_{i_k+l_k,r}$.
The part of $\HH_k\cup \R_k$ consisting of modules with
quasilength at least $r+1$ is the triangle with corners
$M_{i_k,r+1}$, $M_{i_k,r+l_k}$ and $M_{i_k+l_k-1,r+1}$.
Hence, $X=M_{i,l}$ lies in $\HH_k\cup \R_k$ if and only if $M_{i+r,l-r}$
lies in the triangular region of $\T$ with corners
$M_{i_k,1}$, $M_{i_k,l_k}$ and $M_{i_k+l_k-1,1}$, i.e.
$\W_{\tau T_k}$. The result follows.

For part (b), we consider the diamond-shaped region $E_X$ bounded by the paths
$y^{r-1}x^{r-1}$ and $x^{r-1}y^{r-1}$ starting at $X$.
We have, using an argument similar to the above, that $v_X$ factors
through $\add(\tau T_{\T}\oplus \tau^2 T_{\T})$ if and only if some indecomposable direct
summand of $\tau T_{\T}\oplus \tau^2 T_{\T}$ lies in $E_X$.
Hence, $v_X$ factors through $\add(\tau T_{\T}\oplus \tau^2 T_{\T})$ if and only
if $M_{i+r-1,l-r+1}$ lies in $\cup_{k\in \mathbb{Z}_s} (\W_{\tau T_k}\cup
\W_{\tau^2 T_k})$.
The corners of the trapezoidal region $\HH_k\cup \R_k\setminus \{\Top_k\}$ are
$M_{i_k,r}$, $M_{i_k,r+l_k-1}$, $M_{i_k+1,r+l_k-1}$, $M_{i_k+l_k,r}$.
Hence $X\in \HH_k\cup \R_k\setminus \{\Top_k\}$ if and only if $M_{i+r-1,l-r+1}$
lies in the trapezoidal region with corners
$M_{i_k+r-1,1}$, $M_{i_k+r-1,l_k}$, $M_{i_k+r,l_k}$, $M_{i_k+l_k+r-1,1}$, i.e.
$M_{i_k-1,1}$, $M_{i_k-1,l_k}$, $M_{i_k,l_k}$, $M_{i_k+l_k-1,1}$ which is the
union $\W_{\tau T_k}\cup \W_{\tau^2 T_k}$.
This gives the first part of (b).

We have that $v_X$ factors through $\add (\tau T_{\T})$
(respectively, $\add(\tau^2 T_{\T})$) if and only if $M_{i+r-1,l-r+1}$
lies in
$\cup_{k\in \mathbb{Z}_s} \W_{\tau T_k}$ (respectively,
$\cup_{k\in \mathbb{Z}_s} \W_{\tau^2 T_k}$).
Hence $v_X$ factors through both $\add(\tau T_{\T})$
and $\add(\tau^2 T_{\T})$ if and only if $M_{i+r-1,l-r+1}$ lies in
$\cup_{k\in \mathbb{Z}_s} (\W_{\tau T_k}\cap
\W_{\tau^2 T_k})$.
The corners of the triangular region $\R_k$ are
$M_{i_k+1,r}$, $M_{i_k+1,r+l_k-2}$ and $M_{i_k+l_k-1,r}$.
Hence $X\in \R_k$ if and only if $M_{i+r-1,l-r+1}$
lies in the triangular region with corners
$M_{i_k+r,1}$, $M_{i_k+r,l_k-1}$ and $M_{i_k+l_k+r-2,1}$, i.e.
$M_{i_k,1}$, $M_{i_k,l_k-1}$ and $M_{i_k+l_k-2,1}$, which is the
intersection $\W_{\tau T_k}\cap \W_{\tau^2 T_k}$. This
gives the second part of (b).

For part (c), we consider the diamond-shaped region $F_{\tau^{-1}X}$ bounded by the
paths $y^{r-2}x^{r-2}$ and $x^{r-2}y^{r-2}$ starting at $\tau^{-1}X$.
We have, using an argument similar to the above, that $w_X$ factors
through $\add(\tau T_{\T})$ if and only if some indecomposable direct
summand of $\tau T_{\T}$ lies in $F_{\tau^{-1}X}$.
Hence, $w_X$ factors through $\add(\tau T_{\T})$ if and only
if $M_{i+1+r-2,l-r+2}=M_{i+r-1,l-r+2}$ lies in $\cup_{k\in \mathbb{Z}_s} \W_{\tau T_k}$.
The corners of the triangular region $\wideR{k}$ are
$M_{i_k+1,r-1}$, $M_{i_k+1,r+l_k-2}$ and $M_{i_k+l_k,r-1}$.
Hence $X\in \wideR{k}$ if and only if $M_{i+r-1,l-r+2}$ lies in the
the triangular region with corners
$M_{i_k+1+r-1,1}$, $M_{i_k+1+r-1,l_k}$ and $M_{i_k+l_k+r-1,1}$,
i.e.
$M_{i_k,1}$, $M_{i_k,l_k}$ and $M_{i_k+l_k-1,1}$, which is the wing
$\W_{\tau T_k}$. Part (c) follows.
\end{proof}

\begin{figure}
$$
\begin{tikzpicture}[xscale=0.3,yscale=0.35,baseline=(bb.base),quivarrow/.style={black, -latex, shorten >=10pt,shorten <=10pt},translate/.style={black, dotted, shorten >=13pt, shorten <=13pt}]

\newcommand{\rank}{11}
\newcommand{\height}{8}
\newcommand{\copies}{2}

\pgfmathparse{\rank-1}\let\rankm\pgfmathresult;
\pgfmathparse{\height-1}\let\heightm\pgfmathresult;

\pgfmathparse{\rank*\copies}\let\totwidth\pgfmathresult;
\pgfmathparse{\rank*\copies-1}\let\totwidthm\pgfmathresult;

\path (0,0) node (bb) {}; 

\foreach \i in {0,...,\totwidth}{
\foreach \j in {0,...,\height}{
\node [fill=white,inner sep=0pt] at (\i*2,\j*2) {$\circ$};}}

\foreach \i in {0,...,\totwidthm}{
\foreach \j in {0,...,\heightm}{
\node [fill=white,inner sep=0pt] at (\i*2+1,\j*2+1) {$\circ$};}}

\begin{scope}[on background layer]
\draw[dashed] (0,0) -- (0,2*\height+2);
\draw[dashed] (2*\rank,0) -- (2*\rank,2*\height+2);
\draw[dashed] (2*\totwidth,0) -- (2*\totwidth,2*\height+2);

\draw[draw=none,fill=gray!30,opacity=0.4] (4,0) -- (7,3) -- (10,0) -- (4,0);
\draw[draw=none,fill=gray!30,opacity=0.4] (14,0) -- (16,2) -- (18,0) -- (14,0);

\draw[draw=none,fill=gray!30,opacity=0.4] (4+2*\rank,0) -- (7+2*\rank,3) -- (10+2*\rank,0) -- (4+2*\rank,0);
\draw[draw=none,fill=gray!30,opacity=0.4] (14+2*\rank,0) -- (16+2*\rank,2) -- (18+2*\rank,0) -- (14+2*\rank,0);

\draw[dotted] (4,16) -- (1,13) -- (12,2) -- (23,13) -- (20,16);
\draw[draw=none,fill=gray!30,opacity=0.4] (4,16) -- (1,13) -- (12,2) -- (23,13) -- (20,16);
\end{scope}

\draw (1,13.6) node[black] {$\scriptstyle X$};
\draw (23,13.6) node[black] {$\scriptstyle X$};

\draw (7,3.5) node[black] {$\scriptstyle \tau T_0$};
\draw (29,3.5) node[black] {$\scriptstyle \tau T_0$};

\draw (16,2.5) node[black] {$\scriptstyle \tau T_1$};
\draw (38,2.5) node[black] {$\scriptstyle \tau T_1$};

\end{tikzpicture}
$$
\caption{Proof of Proposition~\ref{p:Hfactor}: the shaded region indicates part of (one copy of) the
diamond-shaped region $D_X$. In this case, $u_X$ does not factor through $\add(\tau T_{\T})$.}
\label{f:DX}
\end{figure}
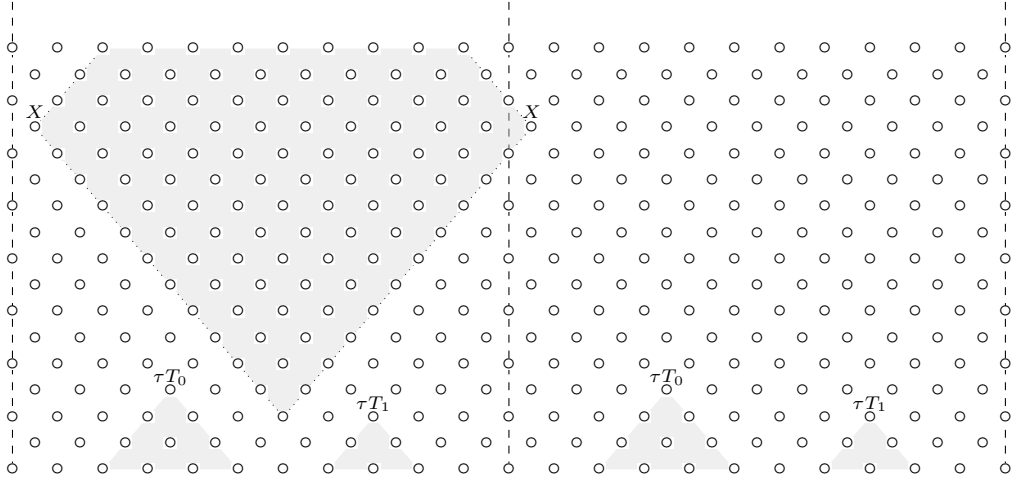

\section{Rigid and Schurian $\Lambda$-modules}
\label{s:rigid}
We determine which objects $X$ in $\T$ give rise to Schurian and rigid $\Lambda$-modules $\XX$.

\begin{lemma} \label{l:not}
Let $X=M_{i,l}$ be an indecomposable module in $\T$. Then:
\begin{itemize}
\item[(a)] If $r+1\leq l$ and $X\not\in \cup_{k\in \mathbb{Z}_s} \HH_k\cup \R_k$
then $\XX$ is not Schurian.
\item[(b)] If $r\leq l$ and $X\not \in \cup_{k\in \mathbb{Z}_s} \HH_k\cup \R_k\setminus \{\Top_k\}$
then $\XX$ is not rigid.
\end{itemize}
\end{lemma}
\begin{proof}
Firstly note that in both (a) and (b), $X$ cannot be a summand of $\tau T$.
For part (a), let $u_X=y^rx^r:X\rightarrow X$. Since $U$ is preprojective, any
composition of maps in $\C$ from $X$ to $X$ factoring through $U$ is zero.
By Proposition~\ref{p:Hfactor}(a) and Lemma~\ref{l:Ccombine}, $u_X$ does not
factor through $\add(\tau T_{\T})$. It follows that $u_X$ does not factor
through $\add(\tau T)$ and hence $\HomH{\C/\add(\tau T)}(X,X)\not\cong \field$,
so $\XX$ is not Schurian.
A similar argument, using Proposition~\ref{p:Hfactor}(b), gives part (b).
\end{proof}

\begin{lemma} \label{l:lowquasilength}
Let $X$ be an indecomposable object in $\T$ which is not a summand of $\tau T$.
Then:
\begin{itemize}
\item[(a)] $\XX$ is a $\tau$-rigid $\Lambda$-module if and only if
the quasilength of $X$ is at most $r-1$;
\item[(b)] If the quasilength of $X$ is at most $r-2$,
then $\XX$ is Schurian.
\end{itemize}
\end{lemma}
\begin{proof}
It is well-known (and follows from the fact that $\T$ is
standard) that $X$ is a rigid $\field Q$-module if and only if its
quasilength is at most $r-1$, so part (a) follows from
Corollary~\ref{c:taurigid}.
If the quasilength of $X$ is at most $r-2$, then $\Hom(X,X)\cong \field$
and $\Hom(X,\tau^2 X)=0$ by Lemma~\ref{l:dimhom},
so
$$\Hom_{\C}(X,X)\cong \Hom(X,X)\oplus \Dual\Hom(X,\tau^2 X)\cong \field,$$
giving part (b).
\end{proof}

We need the following.

\begin{lemma} \label{l:sesargument}
Let $A,B$
be indecomposable $\field Q$-modules, and assume that
$\Hom(A,B[1])\cong \field$.
Let $\varepsilon\in \Hom(A,B[1])$ be nonzero.
\begin{itemize}
\item[(a)] If there is a map $\varphi\in \Hom(B,\tau A)$
such that $\im(\varphi)$ has an indecomposable direct summand which
does not lie in $\cup_{k\in \mathbb{Z}_s}\W_{\tau T_k}$
then $\varepsilon$ factors in $D^b(\field Q)$ through $U[1]$.
\item[(b)] If there is a map $\varphi\in \Hom(B,\tau A)$
such that $\im(\varphi)$ has an indecomposable direct summand which
does not lie in $\cup_{k\in \mathbb{Z}_s}\W_{\tau^2 T_k}$
then $\varepsilon$ factors in $D^b(\field Q)$ through $\tau U[1]$.
\end{itemize}
\end{lemma}
\begin{proof}
We write $\varphi$ as $\varphi_2\varphi_1$ where $\varphi_1:B\rightarrow \im(\varphi)$
and $\varphi_2:\im(\varphi)\rightarrow \tau A$. We have the short exact sequence:
\begin{equation}
\label{e:phi}
\xymatrix{
0 \ar[r] & \ker(\varphi) \ar[r] & B \ar[r]^(0.4){\varphi_1} & \im(\varphi) \ar[r] & 0
}
\end{equation}
For part (a), we apply $\Hom(U,-)$ to this sequence (noting that, since $U$ is
preprojective, it is exact on $\T$), to obtain the exact sequence:
\begin{equation*}
\label{e:Uphi}
\xymatrix{
0 \ar[r] & \Hom(U,\ker(\varphi)) \ar[r] & \Hom(U,B) \ar[rr]^(0.45){\Hom(U,\varphi_1)} && \Hom(U,\im(\varphi)) \ar[r] & 0
}
\end{equation*}
Since $\im\varphi$ has an indecomposable direct summand which does
not lie in $\cup_{k\in \mathbb{Z}_s}\W_{\tau T_k}$, it follows from
Proposition~\ref{p:outsidenonzero} that $\Hom(U,\im\varphi)\not=0$.
Hence, the epimorphism $\Hom(U,\varphi_1)$ is nonzero. Since $\varphi_2$ is a monomorphism,
$\Hom(U,\varphi)\not=0$, so there is a map $\beta\in \Hom(U,B)$ such that
$\Hom(U,\varphi)(\beta)=\varphi\beta\not=0$.
Hence $\Hom(\beta,\tau A)(\varphi)=\varphi\beta\not=0$,
so $\Hom(\beta,\tau A)\not=0$. Part (a) now follows from Proposition~\ref{p:Cfactor}(b), taking
$C=U$.

For part (b), we apply $\Hom(\tau U,-)$ to the sequence~\eqref{e:phi}. Note that
$\Hom(\tau U,\im(\varphi))\cong \Hom(U,\tau^{-1}\im(\varphi))\not=0$ by
Proposition~\ref{p:outsidenonzero}, and the argument goes through as in part (a).
\end{proof}

\begin{lemma}
\label{l:hatrigid}
Fix $k\in \mathbb{Z}_s$ and let $X\in \HH_k\setminus \{\Top_k\}$.
Then $\XX$ is rigid.
\end{lemma}
\begin{proof}
Firstly note that $X$ cannot be a direct summand of $\tau T$.
By the assumption, the quasilength of $X$ lies in the interval $[r,2r-1]$, so,
by Lemma~\ref{l:dimhom}, $\Hom(X,\tau X)\cong \field$.
Let $u=y^{r-1}x^{r-1}$ be a nonzero element of $\Hom(X,\tau X)$.
Then by Proposition~\ref{p:Hfactor}(b), $u$ factors through
$\add(\tau T_{\T}\oplus \tau^2 T_{\T})$, so
$\HomH{\C/\add(\tau T\oplus \tau^2 T)}(X,\tau X)=0$.

Suppose that $X\cong M_{i_k,l}$ where $r\leq l\leq r+l_k-1$.
We have $$\HomF{\C}(X,\tau X)=\Hom(X,X[1])\cong  \Dual\Hom(X,\tau X)\cong \field.$$
We apply Lemma~\ref{l:sesargument}(a) in the case $A=X$, $B=X$.
We take $\varphi=u$ and $\varepsilon$
to be a nonzero element of $\Hom(X,X[1])$.
Then $\im(\varphi)\cong M_{i_k-1,l-r+1}\not\in \cup_{k\in \mathbb{Z}_s}\W_{\tau T_k}$.
By Lemma~\ref{l:sesargument}(a), we have that $\varepsilon$ factors through $U[1]$.
Hence, regarded as an $F$-map in $\C$, $\varepsilon$ factors through $\tau U$. It follows that
$$\HomF{\C/\add(\tau T\oplus \tau ^2 T)}(X,\tau X)=0.$$

Suppose that $X\cong M_{i_k+p,r+l_k-p}$ where $1\leq p\leq l_k$.
We have
$$\HomF{\C}(X,\tau X)=\Hom(X,X[1])\cong \Dual\Hom(X,\tau X)\cong \field.$$
We apply Lemma~\ref{l:sesargument}(b) in the case $A=X$, $B=X$.
We take $\varphi=u$ and $\varepsilon$ to be a nonzero element of $\Hom(X,X[1])$.
Then $\im(\varphi)=M_{i_k+p-1,l_k-p-1}\not\in \cup_{k\in \mathbb{Z}_s}\W_{\tau T_k}$.
By Lemma~\ref{l:sesargument}(b), $\varepsilon$ factors through $\tau U[1]$. Hence, regarded
as an $F$-map in $\C$, $\varepsilon$ factors through $\tau^2 U$. It follows that
$$\HomF{\C/\add(\tau T\oplus \tau ^2 T)}(X,\tau X)=0.$$

In either case, we have shown that $\Hom_{\C/\add(\tau T\oplus \tau^2 T)}(X,\tau X)=0$,
and it follows that $\clu{X}$ is rigid by Proposition~\ref{p:rigid}(b).
\end{proof}

If $T_{\T}$ contains an indecomposable direct summand of quasilength $r-1$ then
$s=1$, $l_0=r-1$ and, by~\eqref{e:hk},
\begin{equation}
\HH_0=\{M_{i_0,l}\,:\,r\leq l\leq 2r-1\}\cup \{M_{i_0+p,2r-1-p}\,:,0\leq p\leq r-1\}.
\label{e:hkmax}
\end{equation}
In particular, $\Top_k=M_{i_0,2r-1}$ has quasilength $2r-1$. In all other cases,
$\Top_k$ has smaller quasilength.

\begin{lemma} \label{l:schurian}
Fix $k\in \mathbb{Z}_s$. Suppose that $X$ is an indecomposable object of $\T$
which is not a summand of $\tau T$ and satisfies either
\begin{itemize}
\item[(a)] $X\in \HH_k$ and $\ql(X)\leq 2r-2$, or
\item[(b)] $\ql(X)\in \{r-1,r\}$ and $X\not\in \cup_{k\in \mathbb{Z}_s} \HH_k\cup \wideR{k}$.
\end{itemize}
Then $\XX$ is Schurian.
\end{lemma}
\begin{proof}
In case (a), $r\leq \ql(X)\leq 2r-2$, and in case (b), $r-1\leq \ql(X)\leq r$.
If $\ql(X)\leq r$ then $\Hom(X,X)\cong \field$ by Lemma~\ref{l:dimhom}.
If $\ql(X)>r$ then $\Hom(X,X)\cong \field^2$. A basis is given by
the identity map and the map $u_X$ in Proposition~\ref{p:Hfactor}(a).
By Proposition~\ref{p:Hfactor}(a), $u_X$ factors through $\add(\tau T)$.
Hence, in either case, $\HomH{\C/\add(\tau T)}(X,X)\cong \field$.

Since the quasilength of $X$ lies in $[r-1,2r-2]$, we have, by Lemma~\ref{l:dimhom}, that
$$\HomF{\C}(X,X)=\Hom(X,\tau^{-1}X[1])\cong \Ext(X,\tau^{-1}X)\cong D\Hom(\tau^{-1}X,\tau X)\cong \field.$$
We apply Lemma~\ref{l:sesargument}(a) in the case $A=X$, $B=\tau^{-1}X$.
We take $\varphi$ to be the map $w_{\tau^{-1}X}$ in Proposition~\ref{p:Hfactor}(c),
the unique nonzero element of $\Hom(\tau^{-1}X,\tau X)$ up to a scalar, and $\varepsilon$
to be a nonzero element of $\Hom(X,\tau^{-1}X[1])$.

In case (a), there are two possibilities. If $X\cong M_{i_k,l}$ where $r\leq l\leq r+l_k-1$,
then $\im(\varphi)\cong M_{i_k+r-1,l-r+2}\not\in \cup_{k\in \mathbb{Z}_s}\W_{\tau T_k}$.
If $X\cong M_{i_k+p,r+l_k-p}$ where $1\leq p\leq l_k$, then
$\im(\varphi)\cong M_{i_k+p+r-1,2+l_k-p}\not\in \cup_{k\in \mathbb{Z}_s}\W_{\tau T_k}$.
In case (b), there are also two possibilities. If $\ql(X)=r-1$, then $X\cong M_{i,r-1}$
where $i\not\in \cup_{k\in \mathbb{Z}_s} [i_k+1,i_k+l_k]$. Then
$$\im(\varphi)\cong M_{i+1+(r-2),r-1-(r-2)}=M_{i-1,1}\not\in \cup_{k\in \mathbb{Z}_s} \W_{\tau T_k}.$$
If $\ql(X)=r$, then $X\cong M_{i,r}$ where $i\not\in \cup_{k\in \mathbb{Z}_s} [i_k,i_k+l_k]$. Then
$$\im(\varphi)\cong M_{i+1+(r-2),r-(r-2)}=M_{i-1,2}\not\in \cup_{k\in \mathbb{Z}_s} \W_{\tau T_k}.$$

Applying Lemma~\ref{l:sesargument}(a), we see that $\varepsilon$ factors through $U[1]$.
Hence, regarded as an $F$-map in $\C$, $\varepsilon$ factors through $\tau U$. It follows that
$$\HomF{\C/\add \tau T}(X,X)=0.$$
We have shown that $\Hom_{\C/\add(\tau T)}(X,X)\cong \field$,
and it follows that $\XX$ is Schurian by Proposition~\ref{p:rigid}(a).
\end{proof}

\begin{lemma} \label{l:belownotrigid}
Fix $k\in \mathbb{Z}_s$, and let $X\in \R_k$. Then $\XX$ is not rigid.
\end{lemma}
\begin{proof}
Since $X\in \R_k$, we have $r\leq \ql(X)\leq r+l_k-2\leq 2r-3$. In particular,
this implies that $X$ is not a direct summand of $\tau T$.
By Lemma~\ref{l:dimhom}, we have $\Hom(X,\tau X)\cong \field$.
Let $u$ be a nonzero map in $\Hom(X,\tau X)$, unique up to a
nonzero scalar. We have
$$\HomF{\C}(X,\tau X)=\Hom(X,X[1])\cong \Dual\Hom(X,\tau X)\cong \field.$$
Let $v\in \Hom(X,X[1])$ be a nonzero map, unique up to a nonzero scalar.

We show first that $v$ cannot factor through $V$ for any indecomposable summand
$V$ of $\tau T$ or $\tau^2 T$. If $\Hom(X,V)=0$ then we are done,
so we may assume that $\Hom(X,V)\not=0$.
Hence, $V$ lies in $\T$ and $\ql(V)\leq r-1$.

By Lemma~\ref{l:homlow}, we have that $\Hom(X,V)\cong \field$. Let $f\in \Hom(X,V)$ be any nonzero map.
Then the number of downward arrows in a path for $f$
(and hence for $\tau f$) is at least $\ql(X)-\ql(V)\geq \ql(X)-r+1$.
The number of downward arrows in a path for $u$ is $r-1$, so the number of downward arrows
in a path for $\tau f\circ u$ is at least $\ql(X)$, so $\tau f\circ u=0$ by Lemma~\ref{l:downwardarrows}.
Since $\{u\}$ is a basis for $\Hom(X,\tau X)$, it follows that $\Hom(X,\tau f)=0$. Therefore,
by Proposition~\ref{p:Cfactor}(a), $v$ cannot factor through $V$.

We next show that $v$ cannot factor through $\tau^{-1}V[1]$ for any indecomposable summand $V$ of
$\tau T$. If $\Hom(\tau^{-1}V,X)=0$ then $\Hom(\tau^{-1}V[1],X[1])=0$ and we are done.
Therefore, we may assume that
$\Hom(\tau^{-1}V,X)\not=0$.

Suppose first that $V\in \T$, so $\ql(V)\leq r-1$.
By Lemma~\ref{l:homlow}, we have that $\Hom(\tau^{-1}V,X)\cong \field$. Let $g\in \Hom(\tau^{-1}V,X)$ be any nonzero map. The number of downward arrows in a path for $u$ is $r-1$, hence
the number of downward arrows in a path for $ug$ is at least $r-1$. As $\ql(\tau^{-1}V)\leq r-1$,
it follows from Lemma~\ref{l:downwardarrows}
that $ug=0$. Since $\{u\}$ is a basis for $\Hom(X,\tau X)$, it follows that $\Hom(g,\tau X)=0$.

Secondly, suppose that $V$ is an indecomposable direct summand of $\tau U$ or $\tau^2 U$.
Let $h\in \Hom(\tau^{-1}V,X)$.
By Proposition~\ref{p:Hfactor}(b), we have that $u$ factors through both $\add(\tau T_{\T})$ and $\add(\tau^2 T_{\T})$,
so $uh=0$ as $\tau^{-1}V$ is a direct
summand of $T\oplus \tau T$. Since $\{u\}$ is a basis for $\Hom(X,\tau X)$, it follows that $\Hom(h,\tau X)=0$.

Applying Proposition~\ref{p:Cfactor}(b) to the triple $A=B=X$, $C=\tau^{-1}V$ and $\beta=g$ or $h$, we obtain
that $v$ does not factor through $\tau^{-1}V[1]$.

We have shown that $v$ does not factor in $D^b(\field Q)$
through $V$ or $\tau^{-1}V[1]$ for
any indecomposable summand $V$ of $\tau T\oplus \tau^2 T$.
Since $\Hom(X,X[1])\cong \field$, it follows that $v$ does not factor
in $D^b(\field Q)$ through
$\add(\tau T\oplus \tau^2 T)$ or
$\add(\tau^{-1}(\tau T\oplus \tau^2 T)[1])$.
By Lemma~\ref{l:Ccombine}, the morphism $v$, regarded as a morphism in $\Hom_{\C}(X,\tau X)$, does not factor in $\C$
through $\add(\tau T\oplus \tau^2 T)$.
Hence: $$\Hom_{\C/\add(\tau T\oplus \tau^2 T)}(X,\tau X)\not=0.$$
Therefore $\XX$ is not rigid by Proposition~\ref{p:rigid}.
\end{proof}

Note that the objects in $\wideR{k}$ (see~\eqref{e:defrk}) have quasilength
at least $r-1$, so if $T$ has no indecomposable direct summand in $\T$ of quasilength
$r-1$, the objects in $\wideR{k}$ are not summands of $\tau T$. It is easy to check
directly that this holds in the case where $T$ has an indecomposable direct summand $T_0$
in $\T$ of quasilength $r-1$, since all the indecomposable direct summands of $\tau T$ in
$\T$ lie in $\W_{\tau T_0}$ (see Figure~\ref{f:regionD}).

\begin{lemma} \label{l:belownotschurian}
Fix $k\in \mathbb{Z}_s$, and let $X\in \wideR{k}$. Then $\XX$ is not Schurian.
\end{lemma}
\begin{proof}
Firstly note that, by the above, $X$ is not an indecomposable direct summand of $\tau T$.
Since $X\in \wideR{k}$, we have
$r-1\leq \ql(X)\leq r+l_k-2\leq 2r-3$, so by Lemma~\ref{l:dimhom}, we have
$\Hom(\tau^{-1}X,\tau X)\cong \field$. Let $u$ be a nonzero map in
$\Hom(\tau^{-1}X,\tau X)$, unique up to a nonzero scalar. We have
$$\HomF{\C}(X,X)=\Hom(X,\tau^{-1}X[1])\cong \Dual\Hom(\tau^{-1}X,\tau X)\cong \field.$$
Let $v\in \Hom(X,\tau^{-1}X[1])$ be a nonzero map, unique up to a nonzero scalar.

We will first show that $v$ cannot factor through $V$ for any indecomposable summand $V$ of $\tau T$.
If $\Hom(X,V)=0$ then we are done,
so we may assume that $\Hom(X,V)\not=0$. In particular, we may assume that $V$ lies in $\T$.
By Lemma~\ref{l:homlow}, $\Hom(X,V)\cong \field$. Let $f\in \Hom(X,V)$ be a nonzero map,
unique up to a nonzero scalar.

If $\ql(V)\leq r-2$ then the number of downward arrows in a 
path for $f$ (and hence for $\tau f$) is at least $\ql(X)-\ql(V)\geq 
\ql(X)-r+2$.
If $\ql(V)=r-1$ then $s=k=1$ and $V=\tau T_1$.
Then, since no object in $\wideR{1}$ is
in the coray through $\tau T_1$, the number of downward arrows
in a path for $f$ (and hence for $\tau f$) is at least
$\ql(X)-\ql(V)+1\geq \ql(X)-r+2$.
The number of downward arrows in a path for $u$ is $r-2$.
Hence in either case the number of downward arrows
in a path for $\tau f\circ u$ is at least $\ql(X)$, so
$\tau f\circ u=0$ by Lemma~\ref{l:downwardarrows}.

Since $\{u\}$ is a basis for $\Hom(X,\tau X)$, it follows that
$\Hom(X,\tau f)=0$. Therefore, by Proposition~\ref{p:Cfactor}(a),
$v$ cannot factor through $V$.

We next show that $v$ cannot factor through $\tau^{-1}V[1]$ for 
any indecomposable summand $V$ of
$\tau T$. If $\Hom(\tau^{-1}V,\tau^{-1}X)=0$ then
$\Hom(\tau^{-1}V[1],\tau^{-1}X[1])=0$ and we are done,
so we may assume that $\Hom(\tau^{-1}V,\tau^{-1}X)\not=0$.

Suppose first that $V\in \T$.
By Lemma~\ref{l:homlow}, $\Hom(\tau^{-1}V,\tau^{-1}X)\cong \field$.
Let $g$ be a non-zero map in $\Hom(\tau^{-1}V,\tau^{-1}X)$, unique
up to a nonzero scalar.

If $\ql(V)\leq r-2$, then the number of downward arrows in a path
for $g$ is at least $\ql(X)-\ql(V)\geq \ql(X)-r+2$.
Since the number of downward arrows in a path for $u$ is $r-2$,
the number of downward arrows in a path for $ug$ is at least
$\ql(X)\geq r-1>\ql(\tau^{-1}V)$, so $ug=0$ by Lemma~\ref{l:downwardarrows}.

If $\ql(V)=r-1$ then $s=k=1$ and $V=\tau T_1$.
Since no element of $\tau^{-1}\wideR{1}$ lies in the ray through
$\tau^{-1} V\cong T_1$, a path for $g$ has at least one
downward arrow.
It follows that a path for $ug$ has at least $r-1=\ql(\tau^{-1}V)$ 
downward arrows, so $ug=0$ in this case also. 
Since $\{u\}$ is a basis for $\Hom(X,\tau X)$, 
it follows that, in either case, $\Hom(g,\tau X)=0$.

Secondly, suppose that $V$ is an indecomposable direct summand of
$\tau U$, and let $h\in \Hom(\tau^{-1}V,\tau^{-1}X)$.
By Proposition~\ref{p:Hfactor}(c), $u$ factors through $\tau T_k$, since $X\in \wideR{k}$.
Hence, $uh=0$ as $\tau^{-1}V$ is a direct summand of $T$.
Since $\{u\}$ is a basis for $\Hom(X,\tau X)$, it follows that $\Hom(h,\tau X)=0$.

Applying Proposition~\ref{p:Cfactor}(b) to the triple $A=B=X$, $C=\tau^{-1}V$, we obtain
that $v$ does not factor through $\tau^{-1}V[1]$.

We have shown that $v$ does not factor in $D^b(\field Q)$
through $V$ or $\tau^{-1}V[1]$ for any indecomposable summand 
$V$ of $\tau T$. Since $\Hom(X,\tau^{-1}X[1])\cong \field$, it follows 
that $v$ does not factor in $D^b(\field Q)$ through $\add(\tau T)$
or $\add(T[1])$.
By Lemma~\ref{l:Ccombine}, the morphism $v$, regarded as a morphism in $\HomF{\C}(X,X)$, does not factor through
$\add(\tau T)$.
Hence $\Hom_{\C/\add(\tau T)}(X,X)\not \cong \field$.
Therefore $\XX$ is not Schurian by Proposition~\ref{p:rigid}.
\end{proof}

Recall (equation~\ref{e:hkmax}) that if $T_{\T}$ contains an indecomposable direct
summand of quasilength $r-1$ then
$$\HH_0=\{M_{i_0,l}\,:\,r\leq l\leq 2r-1\}\cup \{M_{i_0+p,2r-1-p}\,:,0\leq p \leq r-1\}.$$
and $\Top_k=M_{i_0,2r-1}$. The following lemma shows, in particular, that
$\clu{\Top_k}$ is Schurian.

\begin{lemma} \label{l:maximalschurian}
Suppose that $T_{\T}$ contains an indecomposable direct summand $T_0$ of
quasilength $r-1$. Let $X\in \HH_0$. Then $\XX$ is a strongly Schurian,
and hence Schurian, $\Lambda$-module.
\end{lemma}
\begin{proof}
Firstly note that $\ql(X)\geq r$, so $X$ is not a summand of $\tau T$.
Let $V$ be an indecomposable direct summand of $T$.
Note that the entry in the dimension vector of $\XX$ corresponding
to $V$ is equal to $\dim\Hom_{\C}(V,X)$.

Suppose first that $V$ is an indecomposable summand of $U$.
Then by Lemma~\ref{l:additive}, we have that $\Hom(V,Y)=0$ for all objects
$Y$ in $\W_{\tau T_0}$. By Proposition~\ref{p:leq1}, $\dim\Hom(V,Y)\leq 1$
if $Y=M_{i_0-1,1}$ is the unique object on the border of $\T$ not in $\W_{\tau T_0}$.
Using the additivity of $\dim\Hom(V,-)$ on $\T$, we see that
$\dim\Hom(V,X)\leq 1$. Since $V$ is preprojective, $\dim\Hom(X,\tau^2 V)=0$,
so, since
$$\Hom_{\C}(V,X)\cong \Hom(V,X)\oplus \Dual\Hom(X,\tau^2 V),$$
we have $\dim\Hom_{\C}(V,X)\leq 1$.
If $V$ lies in a tube other than $\T$ then $\Hom_{\C}(V,X)=0$.
So we are left with the case where $V$ lies in $\T$.

If $X\cong M_{i_0,l}$ for some $l$ with $r\leq l\leq 2r-1$ then the quasisocle
of $X$ is $Q_{i_0}$, which does not lie in $\W_T$.
So, by Corollary~\ref{c:wingzero}, $\Hom(V,X)=0$.
Since $\ql(V)\leq r-1$, it follows from
Lemma~\ref{l:homlow} that $\dim\Hom(X,\tau^2 V)\leq 1$.
Hence $\dim\Hom_{\C}(V,X)\leq 1$.

If $X\cong M_{i_0+p,2r-1-p}$ for some $p$ with $0\leq p\leq r-1$ then
the quasitop of $X$ is $Q_{i_0+p+2r-1-p-1}=Q_{i_0-2}$,
which does not lie in $\W_{\tau^2 T}$. So, by Corollary~\ref{c:wingzero},
we have that $\Hom(X,\tau^2 V)=0$. Since $\ql(V)\leq r-1$, it follows from
Lemma~\ref{l:homlow} that $\dim\Hom(V,X)\leq 1$.
Hence $\dim\Hom_{\C}(V,X)\leq 1$.

We have shown that $\XX$ is strongly Schurian as required.
Since any strongly Schurian module is Schurian, we are done.
\end{proof}

\begin{corollary} \label{c:schurian}
Let $X\in \cup_{k\in \mathbb{Z}_s}\HH_k$. Then $\clu{X}$ is Schurian.
\end{corollary}

\begin{proof}
Firstly note that, since $\ql(X)\geq r$, $X$ is not a direct summand of $\tau T$.
Suppose $k\in \mathbb{Z}_s$ and $X\in \HH_k$.
If $\ql(X)\leq 2r-2$ then this follows from Lemma~\ref{l:schurian}.
The maximal quasilength of an object in $\HH_k$ is $\ql(\Top_k)=\ql(M_{i_k,r+l_k})=r+l_k$.
This is only greater than $2r-2$ when $l_k$ is maximal, i.e. equal to $r-1$. Then $s=1$
(i.e. there is only one indecomposable direct summand of $T_{\T}$ not contained in the wing of
another indecomposable direct summand of $T_{\T}$).
We must have $k=0$ and the result follows from Lemma~\ref{l:maximalschurian}.
\end{proof}

We have now determined whether $\clu{X}$ is rigid or Schurian for
all indecomposable modules $X$ in $\T$ which are not direct summands of $\tau T$.
We summarize this with the following theorem.
Note that, by Theorem~\ref{t:equivalence}, every indecomposable $\Lambda$-module
is of the form $\clu{X}$ for $X$ an indecomposable object in $\C$ which is not a direct summand
of $\tau T$.
Note also that part (a) of the following is a consequence of Lemma~\ref{l:lowquasilength}(a),
which was shown using~\cite{AIR14}.

\begin{theorem} \label{t:classification}
Let $Q$ be a quiver of tame representation type, and $\C$ the corresponding
cluster category. Let $T$ be an arbitrary cluster-tilting object in $\C$.
Let $X$ be an indecomposable object of $\C$ which is not a summand of $\tau T$
and let $\clu{X}$ the corresponding $\Lambda$-module.
\begin{itemize}
\item[(a)] The $\Lambda$-module $\clu{X}$ is $\tau$-rigid if and only if
$X$ is transjective or $X$ is regular and $\ql(X)\leq r-1$.
\item[(b)] The $\Lambda$-module $\clu{X}$ is rigid if and only if
either
\begin{itemize}
\item[(i)] $X$ is transjective, or
\item[(ii)] $X$ is regular and $\ql(X)\leq r-1$ or
\item[(iii)] $X$ is regular and $X\in \cup_{k\in\mathbb{Z}_s} \HH_k\setminus \{\Top_k\}$.
\end{itemize}
\item[(c)] The $\Lambda$-module $\clu{X}$ is Schurian if and only if
either
\begin{itemize}
\item[(i)] $X$ is transjective, or
\item[(ii)] $X$ is regular and $\ql(X)\leq r-2$, or
\item[(iii)] $X$ is regular, $\ql(X)\in \{r-1,r\}$ and $X\not\in \cup_{k\in \mathbb{Z}_s}\wideR{k}$, or
\item[(iv)] $X$ is regular, $\ql(X)\geq r+1$ and $X\in \cup_{k\in \mathbb{Z}_s} \HH_k$.
\end{itemize}
\end{itemize}
\end{theorem}

\begin{proof}
If $X$ is transjective, the result follows from Remark~\ref{r:transjective}, so we may assume
that $X$ lies in a tube $\T$. Let $r$ be the rank of $\T$.
Replacing $T$ with $\tau^{mr}T$ for some $m\in\mathbb{Z}$ if necessary, we may assume that
$T$ is of the form $U\oplus T'$ where $U$ is a preprojective module and $T'$ is regular, i.e.
that Assumption~\ref{a:preprojective} holds (note that $\tau$ is an autoequivalence of $\C$).
For part (b), note that if $\ql(X)\leq r-1$, then $\XX$ is $\tau$-rigid by (a), 
hence rigid. If $\ql(X)\geq r$ and $X\not\in \cup_{k\in\mathbb{Z}_s} \HH_k\cup \R_k\setminus \{\Top_k\}$
then $\clu{X}$ is not rigid by Lemma~\ref{l:not}. If $\ql(X)\geq r$ and $X\in \R_k$ then $\clu{X}$
is not rigid by Lemma~\ref{l:belownotrigid}. And if $\ql(X)\geq r$ and $X\in \cup_{k\in\mathbb{Z}_s} \HH_k\setminus \{\Top_k\}$ then $\clu{X}$ is rigid by Lemma~\ref{l:hatrigid}.

For part (c), note that if $\ql(X)\leq r-2$ then $\clu{X}$ is Schurian by Lemma~\ref{l:lowquasilength}.
If $\ql(X)\geq r+1$ and $X\not\in \cup_{k\in \mathbb{Z}_s} \HH_k\cup R_k$ then $\clu{X}$ is not
Schurian by Lemma~\ref{l:not}. If $\ql(X)\geq r+1$ and $X\in R_k$ then $X\in \wideR{k}$
so $\clu{X}$ is not Schurian by Lemma~\ref{l:belownotschurian}. If $\ql(X)\geq r+1$ and
$X\in \HH_k$ then $\clu{X}$ is Schurian by Corollary~\ref{c:schurian}.

If $\ql(X)\in \{r-1,r\}$ and $X\not\in \cup_{k\in \mathbb{Z}_s} \HH_k\cup \wideR{k}$ then
$\clu{X}$ is Schurian by Lemma~\ref{l:schurian}. If $\ql(X)\in \{r-1,r\}$ and $X\in \HH_k$ then
$\clu{X}$ is Schurian by Corollary~\ref{c:schurian}. If $\ql(X)\in \{r-1,r\}$ and $X\in \wideR{k}$
then $\clu{X}$ is not Schurian by Lemma~\ref{l:belownotschurian}.
\end{proof}

\begin{corollary} \label{c:rigidschurian}
Let $Q$ be a quiver of finite or tame representation type and $\Lambda$
a cluster-tilted algebra arising from the cluster category of $Q$. Then
every indecomposable $\Lambda$-module which is rigid, but not $\tau$-rigid,
is Schurian.
\end{corollary}
\begin{proof}
If $Q$ is of finite representation type, then it is known that every indecomposable
object in $D^b({\field}Q$) is rigid. Hence, by Theorem~\ref{c:taurigid}, every indecomposable
$\Lambda$-module is $\tau$-rigid and the statement is vacuous in this case.

Suppose that $Q$ is of tame representation type. Let $\Lambda=\End_{\C}(T)^{\opp}$,
where $T$ is a cluster-tilting object in the cluster category $\C$ of $Q$.
Let $X$ be an indecomposable object in $\C$ which is not a summand of $\tau T$.
If $\XX$ is rigid, but not $\tau$-rigid, then by Theorem~\ref{t:classification},
we have that $X$ is regular and
$X\in \cup_{k\in\mathbb{Z}_s} \HH_k\setminus \{\Top_k\}$. If $\ql(X)=r$,
then $\XX$ is Schurian by Theorem~\ref{t:classification}(c)(iii), since
$\cup_{k\in \mathbb{Z}_s}\HH_k\cap \cup_{k\in \mathbb{Z}_s}\wideR{k}$ is empty.
If $\ql(X)\geq r+1$, then $\XX$ is Schurian by Theorem~\ref{t:classification}(c)(iv).
\end{proof}

In Figure~\ref{f:tameAR}, we show part of the AR-quiver of $\Lambda$-mod for Example~\ref{ex:running}.
The part shown consists of modules coming from the tube in $\field Q$-mod shown in Figure~\ref{f:tamehereditaryAR}.
We give a $Q_{\Lambda}$-coloured quiver for each module,
where $Q_{\Lambda}$ is the quiver of $\Lambda$.
Note that we need to distinguish between the two arrows between vertices $1$ and $4$.
We do this by decorating the arrow which is involved in the relations with an asterisk.
Recall that this then has the following interpretation (see the text after Definition~\ref{d:colouredquiver}). Suppose that $\varphi$ is the linear map corresponding to the decorated (respectively, undecorated) arrow in $Q_{\Lambda}$.
Then the image of a basis element $b\in B_1$ (the basis of the vector space at the vertex $1$)
under $\varphi$ is the sum of the basis elements $c\in B_4$ which are
at the end of an arrow starting at $b$ labelled with (respectively, without) an asterisk.
The diagram on the right shows which of these modules are $\tau$-rigid, rigid and Schurian.

In Figure~\ref{f:schurianrigid}, we illustrate the $\tau$-rigid, rigid and Schurian $\Lambda$-modules given by
Theorem~\ref{t:classification} for the example in Figure~\ref{f:wings} (choosing specific indecomposable
summands of $T$ in the wings of the $T_i$).

\begin{figure}
$$
\begin{tikzpicture}[xscale=0.27,yscale=0.22,baseline=(bb.base),quivarrow/.style={black, -latex},translate/.style={black, dotted}]

\path (0,0) node (bb) {}; 

\draw (0,0) node (P2) { \small
 \begin{tikzpicture} [scale=0.7,yscale=1,ext/.style={black,shorten <=-1pt, shorten >=-1pt}]
  \draw (0,0) node (P23A) {$\scriptstyle 3$};
  \draw (0,0.5) node (P24) {$\scriptstyle 4$};
  \draw (0,1) node (P21) {$\scriptstyle 1$};
  \draw (0,1.5) node (P23B) {$\scriptstyle 3$};
  \draw (0,2) node (P22) {$\scriptstyle 2$};
\draw[ext] (P23A) -- (P24);
\draw[ext] (P24) -- (P21);
\draw[ext] (P21) -- (P23B);
\draw[ext] (P23B) -- (P22);
 \end{tikzpicture}
};

\draw (10,0) node (I2) { \small
 \begin{tikzpicture} [scale=1]
  \draw (0,0) node {$\scriptstyle 2$};
 \end{tikzpicture}
};

\draw (20,0) node { \small
 \begin{tikzpicture} [scale=1]
  \draw (0,0) node {$\scriptstyle \bullet$};
 \end{tikzpicture}
};

\draw (30,0) node (PP2) { \small
 \begin{tikzpicture} [scale=0.7,yscale=1,ext/.style={black,shorten <=-1pt, shorten >=-1pt}]
  \draw (0,0) node (PP23A) {$\scriptstyle 3$};
  \draw (0,0.5) node (PP24) {$\scriptstyle 4$};
  \draw (0,1) node (PP21) {$\scriptstyle 1$};
  \draw (0,1.5) node (PP23B) {$\scriptstyle 3$};
  \draw (0,2) node (PP22) {$\scriptstyle 2$};
\draw[ext] (PP23A) -- (PP24);
\draw[ext] (PP24) -- (PP21);
\draw[ext] (PP21) -- (PP23B);
\draw[ext] (PP23B) -- (PP22);
 \end{tikzpicture}
};

\draw (5,7) node (I3) { \small
 \begin{tikzpicture} [scale=0.7,yscale=1,ext/.style={black,shorten <=-1pt, shorten >=-1pt}]
  \draw (0,0) node (I33A) {$\scriptstyle 3$};
  \draw (0.25,0.5) node (I34) {$\scriptstyle 4$};
  \draw (0.25,1) node (I31) {$\scriptstyle 1$};
  \draw (0.25,1.5) node (I33B) {$\scriptstyle 3$};
  \draw (0.25,2) node (I32A) {$\scriptstyle 2$};
  \draw (-0.25,0.5) node (I32B) {$\scriptstyle 2$};
\draw[ext] (I33A) -- (I34);
\draw[ext] (I34) -- (I31);
\draw[ext] (I31) -- (I33B);
\draw[ext] (I33B) -- (I32A);
\draw[ext] (I32B) -- (I33A);
 \end{tikzpicture}
};

\draw (15,7) node { \small
 \begin{tikzpicture} [scale=1]
  \draw (0,0) node {$\scriptstyle \bullet$};
 \end{tikzpicture}
};

\draw (25,7) node (X1) { \small
 \begin{tikzpicture} [scale=0.7,yscale=1,ext/.style={black,shorten <=-1pt, shorten >=-1pt}]
  \draw (0,0) node (X13A) {$\scriptstyle 3$};
  \draw (0,0.5) node (X14) {$\scriptstyle 4$};
  \draw (0,1) node (X11) {$\scriptstyle 1$};
  \draw (0,1.5) node (X13B) {$\scriptstyle 3$};
\draw[ext] (X13A) -- (X14);
\draw[ext] (X14) -- (X11);
\draw[ext] (X11) -- (X13B);
 \end{tikzpicture}
};

\draw (0,14) node (X2) { \small
 \begin{tikzpicture} [scale=0.7,yscale=1,ext/.style={black,shorten <=-1pt, shorten >=-1pt}]
  \draw (0,0) node (X23A) {$\scriptstyle 3$};
  \draw (0.25,0.5) node (X24) {$\scriptstyle 4$};
  \draw (0.25,1) node (X21) {$\scriptstyle 1$};
  \draw (0.25,1.5) node (X23B) {$\scriptstyle 3$};
  \draw (-0.25,0.5) node (X22) {$\scriptstyle 2$};
\draw[ext] (X23A) -- (X24);
\draw[ext] (X24) -- (X21);
\draw[ext] (X21) -- (X23B);
 \draw[ext] (X22) -- (X23A);
 \end{tikzpicture}
};

\draw (10,14) node (X3) { \small
 \begin{tikzpicture} [scale=0.7,yscale=1,ext/.style={black,shorten <=-1pt, shorten >=-1pt}]
  \draw (0,0) node (X34) {$\scriptstyle 4$};
  \draw (0,0.5) node (X31) {$\scriptstyle 1$};
  \draw (0,1) node (X33) {$\scriptstyle 3$};
  \draw (0,1.5) node (X32) {$\scriptstyle 2$};
\draw[ext] (X34) -- (X31);
\draw[ext] (X31) -- (X33);
\draw[ext] (X33) -- (X32);
 \end{tikzpicture}
};

\draw (20,14) node (X4) { \small
 \begin{tikzpicture} [scale=0.7,yscale=1,ext/.style={black,shorten <=-1pt, shorten >=-1pt}]
  \draw (0,0) node (X43) {$\scriptstyle 3$};
  \draw (0,0.5) node (X44) {$\scriptstyle 4$};
  \draw (0,1) node (X41) {$\scriptstyle 1$};
\draw[ext] (X43) -- (X44);
\draw[ext] (X44) -- (X41);
 \end{tikzpicture}
};

\draw (30,14) node (XX2) { \small
 \begin{tikzpicture} [scale=0.7,yscale=1,ext/.style={black,shorten <=-1pt, shorten >=-1pt}]
  \draw (0,0) node (XX23A) {$\scriptstyle 3$};
  \draw (0.25,0.5) node (XX24) {$\scriptstyle 4$};
  \draw (0.25,1) node (XX21) {$\scriptstyle 1$};
  \draw (0.25,1.5) node (XX23B) {$\scriptstyle 3$};
  \draw (-0.25,0.5) node (XX22) {$\scriptstyle 2$};
\draw[ext] (XX23A) -- (XX24);
\draw[ext] (XX24) -- (XX21);
\draw[ext] (XX21) -- (XX23B);
 \draw[ext] (XX22) -- (XX23A);
 \end{tikzpicture}
};

\draw (5,21) node (X5) { \small
 \begin{tikzpicture} [scale=0.7,yscale=1,ext/.style={black,shorten <=-1pt, shorten >=-1pt}]
  \draw (0,0) node (X54) {$\scriptstyle 4$};
  \draw (0,0.5) node (X51) {$\scriptstyle 1$};
  \draw (0,1) node (X53) {$\scriptstyle 3$};
\draw[ext] (X54) -- (X51);
\draw[ext] (X51) -- (X53);
 \end{tikzpicture}
};

\draw (15,21) node (X6) { \small
 \begin{tikzpicture} [scale=0.7,yscale=1,ext/.style={black,shorten <=-1pt, shorten >=-1pt}]
  \draw (0,0) node (X64A) {$\scriptstyle 4$};
  \draw (0.06,0.26) node {$\scriptstyle \ast$}; 
  \draw (-0.25,0.5) node (X61A) {$\scriptstyle 1$};
  \draw (-0.25,1) node (X63A) {$\scriptstyle 3$};
  \draw (-0.25,1.5) node (X62) {$\scriptstyle 2$};
  \draw (0.25,0.5) node (X61B) {$\scriptstyle 1$};
  \draw (0.5,0) node (X64B) {$\scriptstyle 4$};
  \draw (0.5,-0.5) node (X63B) {$\scriptstyle 3$};
\draw[ext] (X62) -- (X63A);
\draw[ext] (X63A) -- (X61A);
\draw[ext] (X61A) -- (X64A);
\draw[ext] (X64A) -- (X61B);
\draw[ext] (X61B) -- (X64B);
\draw[ext] (X64B) -- (X63B);
 \end{tikzpicture}
};

\draw (25,21) node (X7) { \small
 \begin{tikzpicture} [scale=0.7,yscale=1,ext/.style={black,shorten <=-1pt, shorten >=-1pt}]
  \draw (0,0) node (X73) {$\scriptstyle 3$};
  \draw (0.25,0.5) node (X74) {$\scriptstyle 4$};
  \draw (0.25,1) node (X71) {$\scriptstyle 1$};
  \draw (-0.25,0.5) node (X72) {$\scriptstyle 2$};
\draw[ext] (X73) -- (X74);
\draw[ext] (X74) -- (X71);
\draw[ext] (X73) -- (X72);
 \end{tikzpicture}
};

\draw (0,28) node (X8) { \small
 \begin{tikzpicture} [scale=0.7,yscale=1,ext/.style={black,shorten <=-1pt, shorten >=-1pt}]
  \draw (0,0) node (X84) {$\scriptstyle 4$};
  \draw (0,0.5) node (X81) {$\scriptstyle 1$};
\draw[ext] (X84) -- (X81);
 \end{tikzpicture}
};

\draw (10,28) node (X9) { \small
 \begin{tikzpicture} [scale=0.7,yscale=1,ext/.style={black,shorten <=-1pt, shorten >=-1pt}]
  \draw (0,0) node (X94A) {$\scriptstyle 4$};
  \draw (0.06,0.26) node {$\scriptstyle \ast$}; 
  \draw (-0.25,0.5) node (X91A) {$\scriptstyle 1$};
  \draw (-0.25,1) node (X93A) {$\scriptstyle 3$};
  \draw (0.25,0.5) node (X91B) {$\scriptstyle 1$};
  \draw (0.5,0) node (X94B) {$\scriptstyle 4$};
  \draw (0.5,-0.5) node (X93B) {$\scriptstyle 3$};
\draw[ext] (X93A) -- (X91A);
\draw[ext] (X91A) -- (X94A);
\draw[ext] (X94A) -- (X91B);
\draw[ext] (X91B) -- (X94B);
\draw[ext] (X94B) -- (X93B);
 \end{tikzpicture}
};

\draw (20,28) node (X10) { \small
 \begin{tikzpicture} [scale=0.7,yscale=1,ext/.style={black,shorten <=-1pt, shorten >=-1pt}]
  \draw (0,0) node (X104A) {$\scriptstyle 4$};
  \draw (0.06,0.26) node {$\scriptstyle \ast$}; 
  \draw (-0.25,0.5) node (X101A) {$\scriptstyle 1$};
  \draw (-0.25,1) node (X103A) {$\scriptstyle 3$};
  \draw (-0.25,1.5) node (X102A) {$\scriptstyle 2$};
  \draw (0.25,0.5) node (X101B) {$\scriptstyle 1$};
  \draw (0.5,0) node (X104B) {$\scriptstyle 4$};
  \draw (0.75,-0.5) node (X103B) {$\scriptstyle 3$};
  \draw (1,0) node (X102B) {$\scriptstyle 2$};
\draw[ext] (X102A) -- (X103A);
\draw[ext] (X103A) -- (X101A);
\draw[ext] (X101A) -- (X104A);
\draw[ext] (X104A) -- (X101B);
\draw[ext] (X101B) -- (X104B);
\draw[ext] (X104B) -- (X103B);
\draw[ext] (X103B) -- (X102B);
 \end{tikzpicture}
};

\draw (30,28) node (XX8) { \small
 \begin{tikzpicture} [scale=0.7,yscale=1,ext/.style={black,shorten <=-1pt, shorten >=-1pt}]
  \draw (0,0) node (XX84) {$\scriptstyle 4$};
  \draw (0,0.5) node (XX81) {$\scriptstyle 1$};
\draw[ext] (XX84) -- (XX81);
 \end{tikzpicture}
};

\draw (P2.south west) node {$\scriptstyle P_2$};
\draw (PP2.south west) node {$\scriptstyle P_2$};
\draw (I2.south west) node {$\scriptstyle I_2$};
\draw (I3.south west) node {$\scriptstyle I_3$};
\draw (X1.south west) node {$\scriptstyle P_3$};

\draw[dashed] (P2.north) -- (X2.south);
\draw[dashed] (X2.north) -- (X8.south);
\draw[dashed] (X8.north) -- ($(X8.north)+(0,3)$);
\draw[dashed] (PP2.north) -- (XX2.south);
\draw[dashed] (XX2.north) -- (XX8.south);
\draw[dashed] (XX8.north) -- ($(XX8.north)+(0,3)$);

\draw[quivarrow] (P2.30) -- (I3.210);
\draw[quivarrow] (I3.330) -- (I2.150);
\draw[quivarrow] (X1.330) -- (PP2.150);
\draw[quivarrow] (X2.330) -- (I3.150);
\draw[quivarrow] (I3.30) -- (X3.210);
\draw[quivarrow] (X4.330) -- (X1.150);
\draw[quivarrow] (X1.30) -- (XX2.210);
\draw[quivarrow] (X2.30) -- (X5.210);
\draw[quivarrow] (X5.330) -- (X3.150);
\draw[quivarrow] (X3.30) -- (X6.210);
\draw[quivarrow] (X6.330) -- (X4.150);
\draw[quivarrow] (X4.30) -- (X7.210);
\draw[quivarrow] (X7.330) -- (XX2.150);
\draw[quivarrow] (X8.330) -- (X5.150);
\draw[quivarrow] (X5.30) -- (X9.210);
\draw[quivarrow] (X9.330) -- (X6.150);
\draw[quivarrow] (X6.30) -- (X10.210);
\draw[quivarrow] (X10.330) -- (X7.150);
\draw[quivarrow] (X7.30) -- (XX8.210);

\draw (0,7) node (R2) {};
\draw (30,7) node (S2) {};
\draw (0,21) node (R4) {};
\draw (30,21) node (S4) {};

\draw[translate] (P2.east) -- (I2.west);
\draw[translate] (R2) -- (I3.west);
\draw[translate] (X1.east) -- (S2);
\draw[translate] (X2.east) -- (X3.west);
\draw[translate] (X3.east) -- (X4.west);
\draw[translate] (X4.east) -- (XX2.west);
\draw[translate] (R4) -- (X5.west);
\draw[translate] (X5.east) -- (X6.west);
\draw[translate] (X6.east) -- (X7.west);
\draw[translate] (X7.east) -- (S4);
\draw[translate] (X8.east) -- (X9.west);
\draw[translate] (X9.east) -- (X10.west);
\draw[translate] (X10.east) -- (XX8.west);

\end{tikzpicture} \quad\quad
\begin{tikzpicture}[scale=0.33,baseline=-10ex,quivarrow/.style={black, -latex, shorten >=10pt,shorten <=10pt},translate/.style={black, dotted, shorten >=13pt, shorten <=13pt}]

\path (0,0) node (bb) {}; 

\draw (0,0) node {$\taurigid$};
\draw (0,2) node {$\notschuriannotrigid$};
\draw (0,4) node {$\schuriannotrigid$};

\draw (2,0) node {$\taurigid$};
\draw (4,0) node {$\qhole$};
\draw (6,0) node {$\taurigid$};
\draw (1,1) node {$\notschuriantaurigid$};
\draw (3,1) node {$\qhole$};
\draw (5,1) node {$\notschuriantaurigid$};
\draw (2,2) node {$\schurianrigidnottaurigid$};
\draw (4,2) node {$\schurianrigidnottaurigid$};
\draw (6,2) node {$\notschuriannotrigid$};
\draw (1,3) node {$\schurianrigidnottaurigid$};
\draw (3,3) node {$\notschuriannotrigid$};
\draw (5,3) node {$\schurianrigidnottaurigid$};
\draw (2,4) node {$\notschuriannotrigid$};
\draw (4,4) node {$\notschuriannotrigid$};
\draw (6,4) node {$\schuriannotrigid$};

\draw[dotted] (0,0) -- (0,6);
\draw[dotted] (6,0) -- (6,6);
\end{tikzpicture}
$$
\caption{The left hand diagram shows part of the AR-quiver of $\Lambda$-mod.
The right hand diagram shows the same objects. The symbol for a module
is circular if it is Schurian, filled-in with gray if it
is rigid but not $\tau$-rigid, and filled-in with black if it is $\tau$-rigid.
The symbol $\times$ represents a gap in the AR-quiver (corresponding to an
indecomposable direct summand of $\tau T$)}.
\label{f:tameAR}
\end{figure}
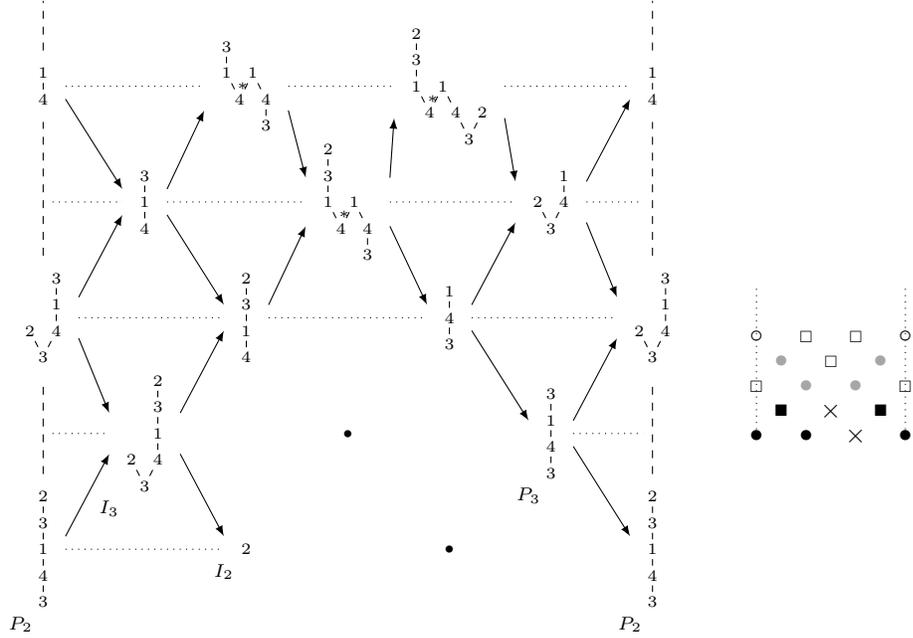

\begin{figure}
$$
\begin{tikzpicture}[xscale=0.3,yscale=0.4,baseline=(bb.base),quivarrow/.style={black, -latex, shorten >=10pt,shorten <=10pt},translate/.style={black, dotted, shorten >=13pt, shorten <=13pt}]

\newcommand{\rank}{11}
\newcommand{\height}{7}
\newcommand{\copies}{2}

\pgfmathparse{\rank-1}\let\rankm\pgfmathresult;
\pgfmathparse{\height-1}\let\heightm\pgfmathresult;

\pgfmathtruncatemacro{\taulimita}{(\rank-3)/2}
\pgfmathtruncatemacro{\taulimitb}{(\rank-4)/2}

\pgfmathparse{\rank*\copies}\let\totwidth\pgfmathresult;
\pgfmathparse{\rank*\copies-1}\let\totwidthm\pgfmathresult;

\path (0,0) node (bb) {}; 


\foreach \i in {0,...,\totwidth}{
\foreach \j in {2,...,\taulimita}{\draw (\i*2,\j*2) node[black] {$\taurigid$};}}

\foreach \i in {0,...,\totwidthm}{
\foreach \j in {2,...,\taulimitb}{\draw (\i*2+1,\j*2+1) node[black] {$\taurigid$};}}

\foreach \i/\j in {0,2,4,8,12,16,18,20}{
\draw (\i,0) node[black] {$\taurigid$};
\draw(\i+22,0) node[black] {$\taurigid$};}

\foreach \i/\j in {6,10,14}{
\draw (\i,0) node[black] {$\qhole$};
\draw (\i+22,0) node[black] {$\qhole$};}

\foreach \i/\j in {1,3,5,7,9,11,13,17,19,21}{
\draw (\i,1) node[black] {$\taurigid$};
\draw(\i+22,1) node[black] {$\taurigid$};}

\foreach \i/\j in {15}{
\draw (\i,1) node[black] {$\qhole$};
\draw (\i+22,1) node[black] {$\qhole$};}

\foreach \i/\j in {0,2,4,6,10,12,14,18,20}{
\draw (\i,2) node[black] {$\taurigid$};
\draw(\i+22,2) node[black] {$\taurigid$};}

\foreach \i/\j in {8,16}{
\draw (\i,2) node[black] {$\qhole$};
\draw (\i+22,2) node[black] {$\qhole$};}

\foreach \i/\j in {1,3,5,9,11,13,15,17,19,21}{
\draw (\i,3) node[black] {$\taurigid$};
\draw(\i+22,3) node[black] {$\taurigid$};}

\foreach \i/\j in {7}{
\draw (\i,3) node[black] {$\qhole$};
\draw (\i+22,3) node[black] {$\qhole$};}

\draw (44,0) node [black] {$\taurigid$};
\draw (44,2) node [black] {$\taurigid$};

\begin{scope}[on background layer]

\draw[dashed] (0,0) -- (0,2*\height+2);
\draw[dashed] (2*\rank,0) -- (2*\rank,2*\height+2);
\draw[dashed] (2*\totwidth,0) -- (2*\totwidth,2*\height+2);


\draw[draw=none,fill=gray!30,opacity=0.4] (4,0) -- (7,3) -- (10,0) -- (4,0);
\draw[draw=none,fill=gray!30,opacity=0.4] (14,0) -- (16,2) -- (18,0) -- (14,0);

\draw[draw=none,fill=gray!30,opacity=0.4] (4+2*\rank,0) -- (7+2*\rank,3) -- (10+2*\rank,0) -- (4+2*\rank,0);
\draw[draw=none,fill=gray!30,opacity=0.4] (14+2*\rank,0) -- (16+2*\rank,2) -- (18+2*\rank,0) -- (14+2*\rank,0);



\draw (1.5,8.5) -- (5,12) -- (8.5,8.5) -- (1.5,8.5);
\draw (2.5,9.5) -- (7.5,9.5);

\draw (23.5,8.5) -- (27,12) -- (30.5,8.5) -- (23.5,8.5);
\draw (24.5,9.5) -- (29.5,9.5);


\draw (13.5,8.5) -- (18,13) -- (22.5,8.5) -- (13.5,8.5);
\draw (14.5,9.5) -- (21.5,9.5);

\draw (35.5,8.5) -- (40,13) -- (44.5,8.5) -- (35.5,8.5);
\draw (36.5,9.5) -- (43.5,9.5);


\draw[dotted] (4,0) -- (18,14) -- (32,0);
\draw[dotted] (14,0) -- (27,13) -- (40,0);

\draw[dotted] (0,8) -- (5,13) -- (18,0);
\draw[dotted] (26,0) -- (40,14) -- (44,10);

\draw[dotted] (10,0) -- (0,10);
\draw[dotted] (36,0) -- (44,8);
\end{scope}


\draw (7,3.5) node[black] {$\scriptstyle \tau T_0$};
\draw (29,3.5) node[black] {$\scriptstyle \tau T_0$};

\draw (16,2.5) node[black] {$\scriptstyle \tau T_1$};
\draw (38,2.5) node[black] {$\scriptstyle \tau T_1$};

\draw (5,13.4) node[black] {$\scriptstyle \Top_1$};
\draw (27,13.4) node[black] {$\scriptstyle \Top_1$};

\draw (18,14.5) node[black] {$\scriptstyle \Top_0$};
\draw (40,14.5) node[black] {$\scriptstyle \Top_0$};


\foreach \i/\j in {1/9,9/9,11/9,13/9,23/9,31/9,33/9,35/9}
{\draw (\i,\j) node {$\taurigid$};}


\foreach \i/\j in {5/13, 18/14, 10/10, 12/10, 27/13, 40/14, 32/10, 34/10}
{\node [fill=white,inner sep=0pt] at (\i,\j) {$\schuriannotrigid$};}


\foreach \i/\j in {3/9,5/9,7/9,15/9,17/9,19/9,21/9,25/9,27/9,29/9,37/9,39/9,41/9,43/9}
{\draw (\i,\j) node {$\notschuriantaurigid$};}



\foreach \i/\j in {14/10,15/11,16/12,17/13,19/13,20/12,21/11,22/10}
{\draw (\i,\j) node {$\schurianrigidnottaurigid$};}


\foreach \i/\j in {0/10,36/10,37/11,38/12,39/13,41/13,42/12,43/11,44/10}
{\draw (\i,\j) node {$\schurianrigidnottaurigid$};}


\foreach \i/\j in {1/11,9/11,11/11,13/11,23/11,31/11,33/11,35/11}
{\draw (\i,\j) node {$\notschuriannotrigid$};}

\foreach \i/\j in {0/12,2/12,8/12,10/12,12/12,14/12,22/12,24/12,30/12,32/12,34/12,36/12,44/12}
\node [fill=white,inner sep=1pt] at (\i,\j) {$\notschuriannotrigid$};

\foreach \i/\j in {1/13,3/13,7/13,9/13,11/13,13/13,15/13,21/13,23/13,25/13,29/13,31/13,33/13,35/13,37/13,43/13}
{\draw (\i,\j) node {$\notschuriannotrigid$};}

\foreach \i/\j in {0/14,2/14,4/14,6/14,8/14,10/14,12/14,14/14,16/14,20/14,22/14,24/14,
26/14,28/14,30/14,32/14,34/14,36/14,38/14,42/14,44/14}
\node [fill=white,inner sep=0pt] at (\i,\j) {$\notschuriannotrigid$};

\foreach \i/\j in {4/10, 6/10, 5/11, 16/10, 18/10, 20/10, 17/11, 19/11, 18/12,
26/10,28/10,27/11,38/10,40/10,42/10,39/11,41/11,40/12}
{\draw (\i,\j) node {$\notschuriannotrigid$};}


\foreach \i/\j in {2/10,3/11,4/12,6/12,7/11,8/10}
{\draw (\i,\j) node {$\schurianrigidnottaurigid$};}


\foreach \i/\j in {24/10,25/11,26/12,28/12,29/11,30/10}
{\draw (\i,\j) node {$\schurianrigidnottaurigid$};}

\end{tikzpicture}
$$
\caption{Schurian and rigid $\Lambda$-modules for a particular choice of tilting module
$T$. The notation is as in Figure~\ref{f:tameAR}.}
\label{f:schurianrigid}
\end{figure}
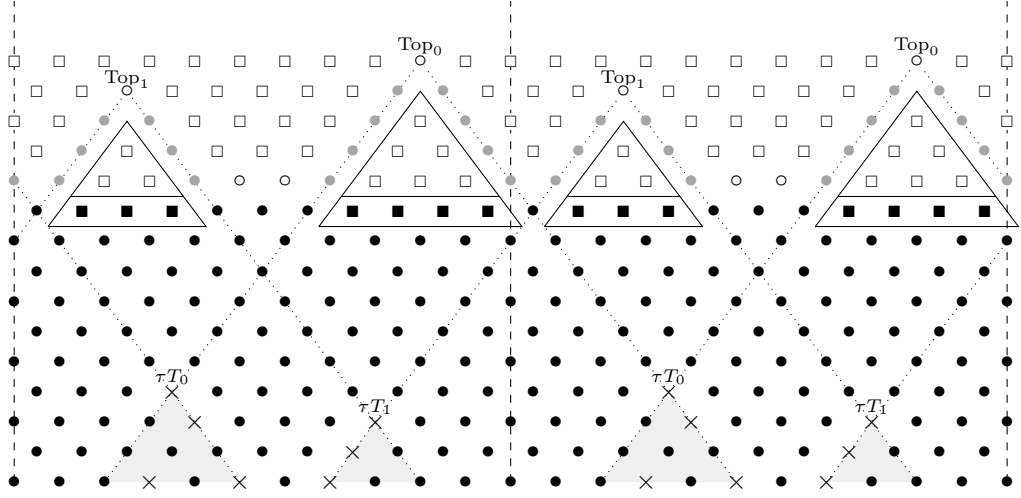

\section{Wild case}
\label{s:wild}
In this section we determine whether some modules are rigid or Schurian for a specific quiver of wild representation type. We will see that there are some similarities with the
tame case.

Let $Q$ be the quiver:
$$
\xymatrix{
0 \ar[r] & 1 \ar[r] \ar@/^1pc/[rrr] & 2 \ar[r] & 3 \ar[r] & 4 
}
$$
\vskip 0.2cm
\noindent and $\field Q$ the corresponding path algebra, of wild representation type.
Let $P_0,\ldots ,P_4$ be the indecomposable projective
$\field Q$-modules (with $Q$-coloured quivers as in~\eqref{e:hereditaryprojectives}),
and $S_0,\ldots ,S_4$ their simple tops.
The simple module $S_2$ is a quasisimple object in a regular component $\R$
of type $\Z A_{\infty}$ in the AR-quiver of $\field Q$-mod. Figure~\ref{f:hereditaryAR}
depicts part of this component.

\begin{equation}
P_0=
 \begin{tikzpicture} [baseline=3ex, xscale=0.7,yscale=1,ext/.style={black,shorten <=-1pt, shorten >=-1pt}]
  \draw (0,0) node (A4B) {$\scriptstyle 4$};
  \draw (0,0.5) node (A3) {$\scriptstyle 3$};
  \draw (0,1) node (A2) {$\scriptstyle 2$};
  \draw (-0.25,1.5) node (A1) {$\scriptstyle 1$};
  \draw (-0.25,2) node (A0) {$\scriptstyle 0$};
  \draw (-0.5,1) node (A4A) {$\scriptstyle 4$};
\draw[ext] (A1) -- (A4A);
\draw[ext] (A1) -- (A2);
\draw[ext] (A2) -- (A3);
\draw[ext] (A3) -- (A4B);
\draw[ext] (A1) -- (A0);
\end{tikzpicture}, \quad
P_1=
 \begin{tikzpicture} [baseline=2.5ex, xscale=0.7,yscale=1,ext/.style={black,shorten <=-1pt, shorten >=-1pt}]
  \draw (0,0) node (A4B) {$\scriptstyle 4$};
  \draw (0,0.5) node (A3) {$\scriptstyle 3$};
  \draw (0,1) node (A2) {$\scriptstyle 2$};
  \draw (-0.25,1.5) node (A1) {$\scriptstyle 1$};
  \draw (-0.5,1) node (A4A) {$\scriptstyle 4$};
\draw[ext] (A1) -- (A4A);
\draw[ext] (A1) -- (A2);
\draw[ext] (A2) -- (A3);
\draw[ext] (A3) -- (A4B);
\end{tikzpicture}, \quad
P_2=
 \begin{tikzpicture} [baseline=2ex, xscale=0.7,yscale=1,ext/.style={black,shorten <=-1pt, shorten >=-1pt}]
  \draw (0,0) node (A4B) {$\scriptstyle 4$};
  \draw (0,0.5) node (A3) {$\scriptstyle 3$};
  \draw (0,1) node (A2) {$\scriptstyle 2$};
\draw[ext] (A2) -- (A3);
\draw[ext] (A3) -- (A4B);
\end{tikzpicture}, \quad
P_3=
 \begin{tikzpicture} [baseline=1.2ex, xscale=0.7,yscale=1,ext/.style={black,shorten <=-1pt, shorten >=-1pt}]
  \draw (0,0) node (A4B) {$\scriptstyle 4$};
  \draw (0,0.5) node (A3) {$\scriptstyle 3$};
\draw[ext] (A3) -- (A4B);
\end{tikzpicture}, \quad
P_4=
 \begin{tikzpicture} [baseline=-0.4ex, xscale=0.7,yscale=1,ext/.style={black,shorten <=-1pt, shorten >=-1pt}]
  \draw (0,0) node (A4B) {$\scriptstyle 4$};
\end{tikzpicture}.
\label{e:hereditaryprojectives}
\end{equation}

\begin{figure}
$$
\begin{tikzpicture}[xscale=0.3,yscale=0.15,baseline=(bb.base),quivarrow/.style={black, -latex},ext/.style={black,shorten <=2pt, shorten >=2pt}, translate/.style={black, dotted}]

\path (0,0) node (bb) {}; 

\draw (10,0) node (B) { \small
 \begin{tikzpicture} [xscale=0.7,yscale=1]
  \draw (0,0) node {$\scriptstyle 3$};
 \end{tikzpicture}
};

\draw (20,0) node (C) { \small
 \begin{tikzpicture} [xscale=0.7,yscale=1]
  \draw (0,0) node {$\scriptstyle 2$};
 \end{tikzpicture}
};

\draw (30,0) node (D) { \small
 \begin{tikzpicture} [xscale=0.7,yscale=1,ext/.style={black,shorten <=-1pt, shorten >=-1pt}]
  \draw (0,0) node (D4) {$\scriptstyle 4$};
  \draw (0,0.5) node (D1) {$\scriptstyle 1$};
   \draw[ext] (D1) -- (D4);
 \end{tikzpicture}
};

\draw (5,7) node (E) { \small
 \begin{tikzpicture} [xscale=0.7,yscale=1,ext/.style={black,shorten <=-1pt, shorten >=-1pt}]
  \draw (0,0) node (E4) {$\scriptstyle 4$};
  \draw (-0.25,0.5) node (E1) {$\scriptstyle 1$};
  \draw (-0.25,1) node (E0) {$\scriptstyle 0$};
  \draw (0.25,0.5) node (E3) {$\scriptstyle 3$};
\draw[ext] (E0) -- (E1);
\draw[ext] (E1) -- (E4);
\draw[ext] (E3) -- (E4);
 \end{tikzpicture}
};

\draw (15,7) node (F) { \small
 \begin{tikzpicture} [xscale=0.7,yscale=1,ext/.style={black,shorten <=-1pt, shorten >=-1pt}]
  \draw (0,0) node (F3) {$\scriptstyle 3$};
  \draw (0,0.5) node (F2) {$\scriptstyle 2$};
  \draw[ext] (F2) -- (F3);
 \end{tikzpicture}
};

\draw (25,7) node (G) { \small
 \begin{tikzpicture} [xscale=0.7,yscale=1,ext/.style={black,shorten <=-1pt, shorten >=-1pt}]
  \draw (0,0) node (G4) {$\scriptstyle 4$};
  \draw (-0.25,0.5) node (G1) {$\scriptstyle 1$};
  \draw (-0.5,0) node (G2) {$\scriptstyle 2$};
\draw[ext] (G1) -- (G2);
\draw[ext] (G1) -- (G4);
 \end{tikzpicture}
};

\draw (10,14) node (H) { \small
 \begin{tikzpicture} [xscale=0.7,yscale=1,ext/.style={black,shorten <=-1pt, shorten >=-1pt}]
  \draw (0,0) node (H4) {$\scriptstyle 4$};
  \draw (-0.25,0.5) node (H1) {$\scriptstyle 1$};
  \draw (-0.25,1) node (H0) {$\scriptstyle 0$};
  \draw (0.25,0.5) node (H3) {$\scriptstyle 3$};
  \draw (0.25,1) node (H2) {$\scriptstyle 2$};
\draw[ext] (H0) -- (H1);
\draw[ext] (H1) -- (H4);
\draw[ext] (H2) -- (H3);
\draw[ext] (H3) -- (H4);
 \end{tikzpicture}
};

\draw (20,14) node (I) { \small
 \begin{tikzpicture} [xscale=0.7,yscale=1,ext/.style={black,shorten <=-1pt, shorten >=-1pt}]
  \draw (0,0) node (I3) {$\scriptstyle 3$};
  \draw (0,0.5) node (I2) {$\scriptstyle 2$};
  \draw (0.25,1) node (I1) {$\scriptstyle 1$};
  \draw (0.5,0.5) node (I4) {$\scriptstyle 4$};
\draw[ext] (I1) -- (I2);
\draw[ext] (I2) -- (I3);
\draw[ext] (I1) -- (I4);
 \end{tikzpicture}
};

\draw (31.7,0) node {$T_2$};
\draw (27.5,7) node {$T_3$};

\draw[quivarrow] (E.330) -- (B.150);
\draw[quivarrow] (G.330) -- (D.150);
\draw[quivarrow] (E.30) -- (H.210);
\draw[quivarrow] (I.330) -- (G.150);
\draw[quivarrow] (B.30) -- (F.210);
\draw[quivarrow] (F.330) -- (C.150);
\draw[quivarrow] (C.30) -- (G.210);
\draw[quivarrow] (H.330) -- (F.150);
\draw[quivarrow] (F.30) -- (I.210);

\draw[translate] (H.east) -- (I.west);
\draw[translate] (B.east) -- (C.west);
\draw[translate] (C.east) -- (D.west);
\draw[translate] (E.east) -- (F.west);
\draw[translate] (F.east) -- (G.west);

\end{tikzpicture}
$$
\caption{Part of the AR-quiver of $\field Q$-mod.}
\label{f:hereditaryAR}
\end{figure}
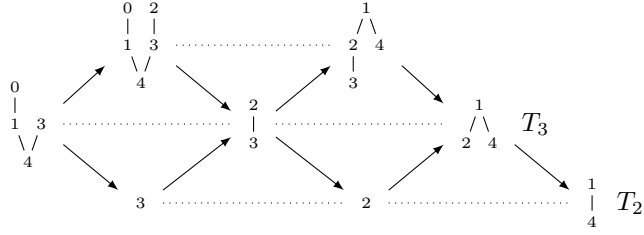

\begin{lemma} \label{l:hereditaryrigid}
Let $X$ be an indecomposable module in $\R$. Then $X$ is rigid if and only if it has quasilength
less than or equal to $2$.
\end{lemma}
\begin{proof}
By~\cite[Thm.\ 2.6]{hoshino84}, every rigid module in a
regular AR-component of a hereditary algebra has
quasilength at most $n-2$, where $n$ is the number of simple modules. In this case, there are $5$ simple modules, so
no indecomposable module in $\R$ with quasilength at least
$4$ is rigid.

Since $\field Q$ is hereditary and no module in $\R$ is
projective, we have
$$\Ext(M,N)\cong \Ext(\tau M,\tau N)$$
for all $M,N \in \R$. Hence an indecomposable module in $\R$ is rigid if and only if every module in
its $\tau$-orbit is rigid.

It is easy to check using the AR-formula that the modules
$S_2$, of quasilength $1$, and 
$\begin{tikzpicture}[baseline=1ex,xscale=0.7,yscale=1,ext/.style={black,shorten <=-1pt, shorten >=-1pt}]
\draw (0,0) node (A3) {$\scriptstyle 3$};
\draw (0,0.5) node (A2) {$\scriptstyle 2$};
\draw[ext] (A2) -- (A3);
\end{tikzpicture}
$, of quasilength $2$, are rigid, while
the module
$\begin{tikzpicture} [baseline=1.5ex,xscale=0.7,yscale=1,ext/.style={black,shorten <=-1pt, shorten >=-1pt}]
 \draw (0,0) node (A3) {$\scriptstyle 3$};
 \draw (0,0.5) node (A2) {$\scriptstyle 2$};
 \draw (0.25,1) node (A1) {$\scriptstyle 1$};
 \draw (0.5,0.5) node (A4) {$\scriptstyle 4$};
\draw[ext] (A3) -- (A2);
\draw[ext] (A2) -- (A1);
\draw[ext] (A1) -- (A4);
\end{tikzpicture}$, of quasilength $3$,
is not rigid. 
Hence every module in $\R$ of quasilength $1$ or $2$ is rigid, and no module in
$\R$ of quasilength $3$ is rigid, and we are done.
\end{proof}

We mutate (in the sense of~\cite{HU89,RS90})
the tilting module $\field Q$ at $P_2$, via the short exact sequence:
$$0\rightarrow P_2\rightarrow P_1\rightarrow T_2\rightarrow 0,$$
where $T_2=
\begin{tikzpicture} [baseline,xscale=0.7,yscale=1,ext/.style={black,shorten <=-1pt, shorten >=-1pt}]
  \draw (0,0.5) node (A1) {$\scriptstyle 1$};
  \draw (0,0) node (A4) {$\scriptstyle 4$};
\draw[ext] (A1) -- (A4);
 \end{tikzpicture}.$
We obtain the tilting module
$$P_0\oplus P_1\oplus T_2\oplus P_3\oplus P_4.$$
We mutate this tilting module at $P_3$, via the short exact sequence 
$$0\rightarrow P_3\rightarrow P_1\rightarrow T_3\rightarrow 0,$$
where
$T_3=\begin{tikzpicture} [baseline,xscale=0.7,yscale=1,ext/.style={black,shorten <=-1pt, shorten >=-1pt}]
  \draw (0,0) node (A2) {$\scriptstyle 2$};
  \draw (0.25,0.5) node (A1) {$\scriptstyle 1$};
  \draw (0.5,0) node (A4) {$\scriptstyle 4$};
\draw[ext] (A1) -- (A2);
\draw[ext] (A1) -- (A4);
 \end{tikzpicture}.$
This gives the tilting module
$$T=P_0\oplus P_1\oplus T_2\oplus T_3\oplus P_4,$$
which induces a cluster-tilting object in $\C$.

\begin{figure}
$$
\begin{tikzpicture}[scale=0.8,baseline=(bb.base),quivarrow/.style={black, -latex},translate/.style={black, dotted},relation/.style={black, dotted, thick=2pt}]

\path (0,0) node (bb) {}; 

\draw (0,0) node (X1) {\small $P_1$};
\draw (2,0) node (X3) {\small $T_3$};
\draw (0,1.5) node (X0) {\small $P_0$};
\draw (2,1.5) node (X2) {\small $T_2$};
\draw (1,-1.5) node (X4) {\small $P_4$};

\draw[quivarrow] (X1.north) -- (X0.south);
\draw[quivarrow] (X3.north) -- (X2.south);
\draw[quivarrow] (X1.east) -- (X3.west);
\draw[quivarrow] (X3.south west) -- (X4.north east);
\draw[quivarrow] (0.85,-1.3) -- (0.15,-0.2);
\draw[quivarrow] (0.7,-1.44) -- (0,-0.34);

\draw (-0.25,0.75) node {$\scriptstyle a$};
\draw (1,0.25) node {$\scriptstyle b$};
\draw (1.65,-0.85) node {$\scriptstyle c$};
\draw (0.66,-0.66) node {$\scriptstyle d$};
\draw (0.14,-0.94) node {$\scriptstyle e$};
\draw (2.3,0.75) node {$\scriptstyle f$};

\end{tikzpicture}
\quad\quad
\begin{tikzpicture}[scale=0.8,baseline=(bb.base),quivarrow/.style={black, -latex},translate/.style={black, dotted},relation/.style={black, dotted, thick=2pt}]

\path (0,0) node (bb) {}; 

\draw (0,0) node (X1) {\small $1$};
\draw (2,0) node (X3) {\small $3$};
\draw (0,1.5) node (X0) {\small $0$};
\draw (2,1.5) node (X2) {\small $2$};
\draw (1,-1.5) node (X4) {\small $4$};

\draw[quivarrow] (X0.south) -- (X1.north);
\draw[quivarrow] (X2.south) -- (X3.north);
\draw[quivarrow] (X3.west) -- (X1.east);
\draw[quivarrow] (X4.north east) -- (X3.south west);
\draw[quivarrow] (0.15,-0.2) -- (0.85,-1.3);
\draw[quivarrow] (0,-0.34) -- (0.7,-1.44);

\draw[relation] (0.7,-0.2) to[out=200,in=135] (0.6,-0.6);
\draw[relation] (1.2,-0.2) to[out=340,in=45] (1.4,-0.6);
\draw[relation] (0.8,-0.95) to[out=305,in=240] (1.32,-0.85);

\draw (0.7,-0.68) node {\small $\ast$};

\end{tikzpicture}
$$
\caption{Maps between indecomposable projective $\field Q$-modules and the quiver with relations of $\End_{\C}(T)^{\opp}$.}
\label{f:tiltedquiver}
\end{figure}
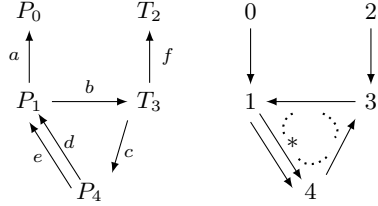

We define maps $a,b,c,d,e,f$ in $\C$ as follows (see Figure~\ref{f:tiltedquiver}).
Let $a$ be the embedding of $P_1$ into $P_0$,
$b$ a surjection of $P_1$ onto $T_3$.
We have $\Hom(T_3,P_4)=0$, while
$$\HomF{\C}(T_3,P_4)\cong \Dual\Hom(P_4,\tau^2 T_3)\cong \field,$$
since $\tau^2 T_3=
\begin{tikzpicture} [baseline=2ex,xscale=0.7,yscale=1,ext/.style={black,shorten <=-1pt, shorten >=-1pt}]
  \draw (0,0) node (A4) {$\scriptstyle 4$};
  \draw (0.25,0.5) node (A3) {$\scriptstyle 3$};
  \draw (-0.25,0.5) node (A1) {$\scriptstyle 1$};
  \draw (-0.25,1) node (A0) {$\scriptstyle 0$};
\draw[ext] (A4) -- (A3);
\draw[ext] (A4) -- (A1);
\draw[ext] (A1) -- (A0);
\end{tikzpicture}$ (see Figure~\ref{f:hereditaryAR}).
We take $c$ to be a non-zero element of $\HomF{\C}(T_3,P_4)$.
There are two embeddings of the simple module $P_4=S_4$ into $P_1$ (see~\eqref{e:hereditaryprojectives}).
We choose $d$ to be the map whose image is the lower $4$
in the $Q$-coloured quiver for $P_1$ in~\eqref{e:projectives},
and $e$ to be the map whose image is the
upper $4$. We take $f$ to be the map from $T_3$ to $T_2$
factoring out the simple $S_2$ in the socle of $T_3$.

Let $g:P_4\rightarrow P_1$ be equal to $d$ or $e$.
Applying Proposition~\ref{p:dual}(b) with $A=T_3$, $B=\tau^{-1}P_1$,
$C=\tau^{-1}P_4$, and $\beta=\tau^{-1}g$ we see that
$\Hom(T_3,\tau^{-1}g[1])=0$ if and only if
$\Hom(\tau^{-1}g,\tau T_3)=0$, which holds
if and only if the map
$$\Hom(g,\tau^2 T_3):\Hom(P_1,\tau^2 T_3) \rightarrow \Hom(P_4,\tau^2 T_3)$$
is zero. We have $\dim\Hom(P_1,\tau^2 T_3)=1$ (see Figure~\ref{f:hereditaryAR}),
so let $h:P_1\rightarrow \tau^2 T_3$ be a nonzero map.

From the explicit definition of the maps $d$ and $e$, we see that
$hd=0$, while $he\not=0$. Hence $\Hom(d,\tau^2 T_3)=0$ and $\Hom(e,\tau^2 T_3)\not=0$.
Therefore, $\Hom(T_3,\tau^{-1}d[1])=0$ and $\Hom(T_3,\tau^{-1}e[1])\not=0$.
Hence, $\Hom(T_3,\tau^{-1}d[1])(c)=(\tau^{-1}d[1])\circ c=0$, so $dc=0$ in $\C$.
Since the domain of $\Hom(T_3,\tau^{-1}e[1])$ is $\HomF{\C}(T_3,P_4)=\Hom(T_3,\tau^{-1}P_4[1])$,
which is spanned by $c$, we have $\Hom(T_3,\tau^{-1}e[1])(c)\not=0$, so $ec\not=0$ in $\C$.
Similarly, we can show that $cb=0$ and $dc=0$ and that the maps
$fb,ad,ae,be$,$bec,fbec$ and $aec$ are all nonzero in $\C$.

It follows that $\Lambda=\End_{\C}(T)^{\opp}$ is given by the
quiver $Q_{\Lambda}$ with the relations shown in
Figure~\ref{f:tiltedquiver} (where we have labelled the arrows with the
corresponding maps between indecomposable projectives in $\field Q$-mod --- note that these
go in the opposite direction).

As in the tame case (see Figure~\ref{f:tameAR}),
we shall draw modules for $\Lambda$ as $Q_{\Lambda}$-coloured quivers,
decorating the arrow between vertices $1$ and $4$ which is involved in the relations
(corresponding to the map $d$) with an asterisk.

Note that the AR-quiver of $\Lambda$-mod is the image of the AR-quiver of $\C$ under
$\Hom_{\C}(T,-)$ by~\cite[Prop.\ 3.2]{BMR07}, with the indecomposable summands of
$\tau T$ deleted; we will denote them by filled-in vertices.

Let $\PP_0,\ldots ,\PP_4$ denote the indecomposable projective
modules over $\Lambda$, $\SL_0,\ldots ,\SL_4$ their simple tops
and $\II_0,\ldots ,\II_4$ the corresponding indecomposable injective
modules. We have:
\begin{equation}
\PP_0=
 \begin{tikzpicture} [baseline=3ex,xscale=0.7,yscale=1,ext/.style={black,shorten <=-1pt, shorten >=-1pt}]
  \draw (0,0) node (P3) {$\scriptstyle 3$};
  \draw (0,0.5) node (P4A) {$\scriptstyle 4$};
\draw (0.44,0.8) node {$\scriptstyle \ast$};
  \draw (0.25,1) node (P1) {$\scriptstyle 1$};
  \draw (0.5,0.5) node (P4B) {$\scriptstyle 4$};
  \draw (0.25,1.5) node (P0) {$\scriptstyle 0$};
\draw[ext] (P0) -- (P1);
\draw[ext] (P1) -- (P4A);
\draw[ext] (P4A) -- (P3);
\draw[ext] (P1) -- (P4B);
\end{tikzpicture}, \quad
\PP_1=
 \begin{tikzpicture} [baseline=2.5ex, xscale=0.7,yscale=1,ext/.style={black,shorten <=-1pt, shorten >=-1pt}]
  \draw (0,0) node (P1) {$\scriptstyle 3$};
  \draw (0,0.5) node (P4A) {$\scriptstyle 4$};
\draw (0.44,0.8) node {$\scriptstyle \ast$}; 
  \draw (0.25,1) node (P1) {$\scriptstyle 1$};
  \draw (0.5,0.5) node (P4B) {$\scriptstyle 4$};
\draw[ext] (P1) -- (P4A);
\draw[ext] (P1) -- (P4B);
\draw[ext] (P4A) -- (P3);
 \end{tikzpicture}, \quad
\PP_2=
 \begin{tikzpicture} [baseline=3ex, xscale=0.7,yscale=1,ext/.style={black,shorten <=-1pt, shorten >=-1pt}]
  \draw (0,0)  node (P3A) {$\scriptstyle 3$};
  \draw (0,0.5) node (P4) {$\scriptstyle 4$};
  \draw (0,1) node (P1) {$\scriptstyle 1$};
  \draw (0,1.5) node (P3B) {$\scriptstyle 3$};
  \draw (0,2) node (P2) {$\scriptstyle 2$};
\draw[ext] (P2) -- (P3B);
\draw[ext] (P3B) -- (P1);
\draw[ext] (P1) -- (P4);
\draw[ext] (P4) -- (P3A);
 \end{tikzpicture}, \quad
\PP_3=
  \begin{tikzpicture} [baseline=2.5ex, xscale=0.7,yscale=1,ext/.style={black,shorten <=-1pt, shorten >=-1pt}]
  \draw (0,0) node (P3A) {$\scriptstyle 3$};
  \draw (0,0.5) node (P4) {$\scriptstyle 4$};
  \draw (0,1) node (P1) {$\scriptstyle 1$};
  \draw (0,1.5) node (P3B) {$\scriptstyle 3$};
\draw[ext] (P3A) -- (P4);
\draw[ext] (P4) -- (P1);
\draw[ext] (P1) -- (P3B);
 \end{tikzpicture}, \quad
\PP_4=
  \begin{tikzpicture} [baseline=0.6ex, xscale=0.7,yscale=1,ext/.style={black,shorten <=-1pt, shorten >=-1pt}]
  \draw (0,0) node (P3) {$\scriptstyle 3$};
  \draw (0,0.5) node (P4) {$\scriptstyle 4$};
\draw[ext] (P3) -- (P4);
 \end{tikzpicture}.
\label{e:projectives}
\end{equation}

\begin{equation}
\II_0=
 \begin{tikzpicture} [baseline=-0.5ex,xscale=0.7,yscale=1]
  \draw (0,0) node {$\scriptstyle 0$};
 \end{tikzpicture}, \quad
\II_1=
 \begin{tikzpicture} [baseline=1.5ex, xscale=0.7,yscale=1,ext/.style={black,shorten <=-1pt, shorten >=-1pt}]
  \draw (0,0) node (I1) {$\scriptstyle 1$};
  \draw (0.25,0.5) node (I0) {$\scriptstyle 0$};
  \draw (-0.25,0.5) node (I3) {$\scriptstyle 3$};
  \draw (-0.25,1) node (I2) {$\scriptstyle 2$};
\draw[ext] (I2) -- (I3);
\draw[ext] (I3) -- (I1);
\draw[ext] (I0) -- (I1);
 \end{tikzpicture}, \quad
\II_2=
 \begin{tikzpicture} [baseline=-0.5ex, xscale=0.7,yscale=1]
  \draw (0,0) node {$\scriptstyle 2$};
 \end{tikzpicture}, \quad
\II_3=
  \begin{tikzpicture} [baseline=3ex, xscale=0.7,yscale=1,ext/.style={black,shorten <=-1pt, shorten >=-1pt}]
  \draw (0,0) node (I3A) {$\scriptstyle 3$};
  \draw (-0.25,0.5) node (I4) {$\scriptstyle 4$};
  \draw (-0.25,1) node (I1) {$\scriptstyle 1$};
  \draw (-0.5,1.5) node (I3B) {$\scriptstyle 3$};
  \draw (-0.5,2) node (I2B) {$\scriptstyle 2$};
  \draw (0.25,0.5) node (I2A) {$\scriptstyle 2$};
  \draw (0,1.5) node (I0) {$\scriptstyle 0$};
\draw[ext] (I3A) -- (I2A);
\draw[ext] (I3A) -- (I4);
\draw[ext] (I4) -- (I1);
\draw[ext] (I1) -- (I3B);
\draw[ext] (I1) -- (I0);
\draw[ext] (I3B) -- (I2B);
 \end{tikzpicture}, \quad
\II_4=
  \begin{tikzpicture} [baseline=2.5ex, xscale=0.7,yscale=1,ext/.style={black,shorten <=-1pt, shorten >=-1pt}]
  \draw (0,0) node (I4) {$\scriptstyle 4$};
  \draw (-0.07,0.3) node {$\scriptstyle \ast$};
  \draw (0.25,0.5) node (I1B) {$\scriptstyle 1$};
  \draw (0.15,1) node (I3) {$\scriptstyle 3$};
  \draw (0.5,1) node (I0B) {$\scriptstyle 0$};
  \draw (0.15,1.5) node (I2) {$\scriptstyle 2$};
  \draw (-0.25,0.5) node (I1A) {$\scriptstyle 1$};
  \draw (-0.5,1) node (I0A) {$\scriptstyle 0$};
\draw[ext] (I4) -- (I1A);
\draw[ext] (I4) -- (I1B);
\draw[ext] (I1A) -- (I0A);
\draw[ext] (I1B) -- (I3);
\draw[ext] (I1B) -- (I0B);
\draw[ext] (I3) -- (I2);
 \end{tikzpicture}.
\end{equation}

\begin{lemma}
Figure~\ref{f:wildAR} illustrates part of the AR-quiver
of $\Lambda$-mod, including the image of the part of the AR-quiver
of $\C$ shown in Figure~\ref{f:hereditaryAR}. 
\end{lemma}
\begin{proof}
Firstly, note that $\Hom_{\C}(T,T_i)\cong \PP_i$ and $\Hom_{\C}(T,\tau^2 T_i)\cong \II_i$, so applying the functor
$\Hom_{\C}(T,-)$ to the first two rows in
Figure~\ref{f:hereditaryAR} gives the first two rows in Figure~\ref{f:wildAR}
except for $X_1$.

If $\alpha:P\rightarrow P'$ is a map between projective $\Lambda$-modules,
we denote by $\alpha^*$ the corresponding map between injective modules,
$\alpha^*:D\Hom_{\Lambda}(P,\Lambda)\rightarrow D\Hom_{\Lambda}(P',\Lambda)$.
A projective presentation of $\SL_2$ is:
$$\xymatrix{
\PP_3 \ar^{\alpha}[r] & \PP_2 \ar[r] & \SL_2 \ar[r] & 0,
}$$
where $\alpha$ is the embedding. So $\tau \SL_2$ is the kernel
of $\alpha^*:\II_3\rightarrow \II_2$.
Let $\beta$ be the nonzero map from $\PP_1$ to $\PP_3$. Since
$\alpha\beta\not=0$, we have $\alpha^*\beta^*\not=0$, so
$\alpha^*$ must be the map factoring out the lower $2$.
It follows that $\tau \SL_2=X_1$, completing the verification of the first two rows
in Figure~\ref{f:wildAR}.

The irreducible maps from $\II_3$ have targets given by the
indecomposable direct summands of $\II_3/\SL_3$, i.e.\ $\II_2$ and $X_3$.
The irreducible map with target $\PP_3$ must be the inclusion of its (indecomposable)
radical $X_3$.
We have:
$$\Ext(X_3,\tau X_3)\cong D\Hom(\tau X_3,\tau X_3)\cong \field,$$
so there is a unique non-split short exact sequence ending in
$X_3$, which must be as shown.

Next, we compute $\tau X_6$. From its $Q_{\Lambda}$-coloured quiver, we see that the
projective cover of $X_6$ is given by $\varphi:\PP_0\oplus \PP_1\oplus \PP_2\rightarrow X_6$.
We need to compute the kernel $\kerphi$ of $\varphi$.
Let $B=\cup_{i\in \{0,1,2,3,4\}} B_i$ be the basis of $\PP_0\oplus \PP_1\oplus \PP_2$ coming from the $Q_{\Lambda}$-coloured
quiver given by the disjoint union of the $Q_{\Lambda}$-coloured quivers of $\PP_1,\PP_2$ and $\PP_3$ 
in~\eqref{e:projectives}. As in Remark~\ref{r:redraw}, we will write the basis elements in $B_i$
as $b_{i1},b_{i2},\ldots $ (in an order taking first the basis elements for 
$\PP_0$, then $\PP_1$ and $\PP_2$). We shall also redraw each connected component of this $Q_{\Lambda}$-coloured quiver as in Remark~\ref{r:redraw}.
We do the same for $X_6$, using the notation $c_{ij}$.
The result is shown in Figure~\ref{f:modulequiver}.

\begin{figure}
$$
\begin{tikzpicture}[xscale=0.8,yscale=1,baseline=(bb.base),quivarrow/.style={black, -latex},translate/.style={black, dotted},relation/.style={black, dotted}]

\path (0,0) node (bb) {}; 

\draw (1,0.8) node (labelP0) {$\PP_0$};

\draw (0,0) node (b01) {$\scriptstyle b_{01}$};
\draw (0,-1) node (b11) {$\scriptstyle b_{11}$};
\draw (1,-2) node (b41) {$\scriptstyle b_{41}$};
\draw (1,-2.5) node (b42) {$\scriptstyle b_{42}$};
\draw (2,-1) node (b31) {$\scriptstyle b_{31}$};

\draw[quivarrow] (b01.south) -- (b11.north);
\draw[quivarrow] (b11.300) -- (b41.west);
\draw[quivarrow] (b11.south) to node[below] {$\scriptstyle \ast$} (b42.west);
\draw[quivarrow] (b41.north) -- (b31.south);

\begin{scope}[shift={(3,0)}]
\draw (1,0.8) node (labelP1) {$\PP_1$};

\draw (0,-1) node (b12) {$\scriptstyle b_{12}$};
\draw (1,-2) node (b43) {$\scriptstyle b_{43}$};
\draw (1,-2.5) node (b44) {$\scriptstyle b_{44}$};
\draw (2,-1) node (b32) {$\scriptstyle b_{32}$};

\draw[quivarrow] (b12.300) -- (b43.west);
\draw[quivarrow] (b12.south) to node[below] {$\scriptstyle \ast$} (b44.west);
\draw[quivarrow] (b43.north) -- (b32.south);
\end{scope}

\begin{scope}[shift={(6,0)}]
\draw (1,0.8) node (labelP2) {$\PP_2$};

\draw (0,-1) node (b13) {$\scriptstyle b_{13}$};
\draw (1,-2.5) node (b45) {$\scriptstyle b_{45}$};
\draw (2,-1) node (b33) {$\scriptstyle b_{33}$};
\draw (2,-1.5) node (b34) {$\scriptstyle b_{34}$};
\draw (2,0) node (b21) {$\scriptstyle b_{21}$};

\draw[quivarrow] (b13.300) -- (b45.west);
\draw[quivarrow] (b33.west) -- (b13.east);
\draw[quivarrow] (b45.north) -- (b34.south);
\draw[quivarrow] (b21.south) -- (b33.north);
\end{scope}

\begin{scope}[shift={(10,0)}]
\draw (1,0.8) node (labelE) {$X_6$};

\draw (0,0) node (c01) {$\scriptstyle c_{01}$};
\draw (0,-1) node (c11) {$\scriptstyle c_{11}$};
\draw (-0.3,-1.5) node (c12) {$\scriptstyle c_{12}$};
\draw (1,-2) node (c41) {$\scriptstyle c_{41}$};
\draw (1,-2.5) node (c42) {$\scriptstyle c_{42}$};
\draw (2,-1) node (c31) {$\scriptstyle c_{31}$};
\draw (2,-1.5) node (c32) {$\scriptstyle c_{32}$};
\draw (2,0) node (c21) {$\scriptstyle c_{21}$};

\draw[quivarrow] (c01.south) -- (c11.north);
\draw[quivarrow] (c11.330) -- (c41.north);
\draw[quivarrow] (c12.340) to node[above] {$\scriptstyle \ast$} (c41.160);
\draw[quivarrow] (c12.270) -- (c42.west);
\draw[quivarrow] (c42.east) -- (c32.south);
\draw[quivarrow] (c31.west) -- (c11.east);
\draw[quivarrow] (c21.south) -- (c31.north);
\end{scope}

\draw[->] (labelP2) -- node [above] {$\varphi$} (labelE); 
\draw (2.5,0.8) node {$\oplus$};
\draw (5.5,0.8) node {$\oplus$};

\end{tikzpicture}
$$
\caption{The projective cover of $E$}
\label{f:modulequiver}
\end{figure}

Let $\kerphi=\ker\varphi$, regarded as a representation with the
vector space $\kerphi_i$ at vertex $i$ of $Q_{\Lambda}$. We describe a basis for
each $\kerphi_i$ in the table in Figure~\ref{f:Ki}. This basis is carefully chosen to allow us to
give an explicit description of $\kerphi$ as a direct sum
of indecomposable modules.

\begin{figure}
\begin{center}
\begin{tabular}{|c|c|c|}
\hline
Vertex $i$ & Action of $\varphi$ & basis for $\kerphi_i$ \\
\hline
$0$ & $b_{01}\mapsto c_{01}$ & empty \\
\hline
    & $b_{11}\mapsto c_{11}$ &  \\
$1$ & $b_{12}\mapsto c_{12}$ & $b_{11}-b_{13}$ \\
    & $b_{13}\mapsto c_{11}$ & \\
\hline
$2$ & $b_{21}\mapsto c_{21}$ & empty \\
\hline
    & $b_{31}\mapsto$ 0 & \\
$3$ & $b_{32}\mapsto c_{32}$ & $b_{31}-b_{34}$, $b_{34}$ \\
    & $b_{33}\mapsto c_{31}$ & \\
    & $b_{34}\mapsto 0$ & \\
\hline
    & $b_{41}\mapsto c_{41}$ & \\
    & $b_{42}\mapsto 0$ & \\
$5$ & $b_{43}\mapsto c_{42}$ & $b_{41}-b_{45}$, $b_{42}$, $b_{44}-b_{45}$ \\
    & $b_{44}\mapsto c_{41}$ & \\
    & $b_{45}\mapsto c_{41}$ & \\
\hline
\end{tabular}
\end{center}
\caption{Computation of a basis for $\kerphi_i$, $i$ a vertex of $Q_{\Lambda}$.}
\label{f:Ki}
\end{figure}

Using Figure~\ref{f:modulequiver}, we can compute the restriction of the linear
maps defining $P$ to the submodule $\kerphi$ to get the description
of $\kerphi$ in Figure~\ref{f:kernelphi}. We obtain a $Q_{\Lambda}$-coloured quiver for this module,
and we obtain that $\kerphi=\ker\varphi\cong \PP_1\oplus \PP_4$.

\begin{figure}
$$
\begin{tikzpicture}[baseline=-11ex, xscale=0.8,yscale=1,quivarrow/.style={black, -latex},translate/.style={black, dotted},relation/.style={black, dotted}]

\draw (0,-1) node (k11) {$\scriptstyle b_{11}-b_{13}$};
\draw (1.5,-2) node (k42) {$\scriptstyle b_{41}-b_{45}$};
\draw (0.8,-2.5) node (k41) {$\scriptstyle b_{42}$};
\draw (2,-1) node (k31) {$\scriptstyle b_{31}-b_{34}$};

\draw[quivarrow] (k11.south) to node[below] {$\scriptstyle \ast$} (k41.north);
\draw[quivarrow] (k11.315) -- (k42.120);
\draw[quivarrow] (k42.north) -- (k31.south);

\begin{scope}[shift={(3,0)}]
\draw (1,-2) node (k43) {$\scriptstyle b_{44}-b_{45}$};
\draw (2,-1) node (k32) {$\scriptstyle b_{34}$};
\draw[quivarrow] (k43.north) -- (k32.south);
\end{scope}
\end{tikzpicture}
\quad \leftrightarrow \quad
\begin{tikzpicture}[baseline=2ex,xscale=0.7,yscale=1,ext/.style={black,shorten <=-1pt, shorten >=-1pt}]
  \draw (0,0) node (P3) {$\scriptstyle 3$};
  \draw (0,0.5) node (P4A) {$\scriptstyle 4$};
  \draw (0.44,0.8) node {$\scriptstyle \ast$}; 
  \draw (0.25,1) node (P1) {$\scriptstyle 1$};
  \draw (0.5,0.5) node (P4B) {$\scriptstyle 4$};
\draw[ext] (P1) -- (P4A);
\draw[ext] (P1) -- (P4B);
\draw[ext] (P4A) -- (P3);
\end{tikzpicture} \oplus
\begin{tikzpicture}[baseline=0.5ex, xscale=0.7,yscale=1,ext/.style={black,shorten <=-1pt, shorten >=-1pt}]
  \draw (0,0) node (P3) {$\scriptstyle 3$};
  \draw (0,0.5) node (P4) {$\scriptstyle 4$};
\draw[ext] (P3) -- (P4);
 \end{tikzpicture}.
$$
\caption{The kernel of the projective cover of $X_6$}
\label{f:kernelphi}
\end{figure}
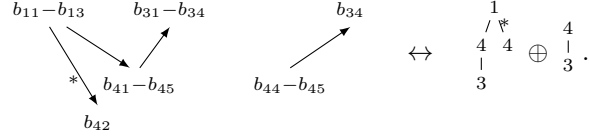

Let $\psi:\PP_1\oplus \PP_4\rightarrow \PP_0\oplus \PP_1\oplus \PP_2$ be the embedding
of $\ker\varphi$ into $\PP_0\oplus \PP_1\oplus \PP_2$. We can write $\psi$ as a
$3\times 2$ matrix $\psi=(\psi_{ij})$, and the components $\psi_{ij}$ can be
read off from the above explicit description of $\ker\varphi$. We have
$\psi^*=(\psi_{ij}^*):\II_1\oplus \II_4\rightarrow
\II_0\oplus \II_1\oplus \II_2$. Since $\psi_{11}:\PP_1\rightarrow \PP_0$ is nonzero,
$\psi_{11}^*$ is a surjection onto $\II_0\cong \SL_0$. Since $\psi_{21}:\PP_1\rightarrow \PP_1$
is the zero map, so is $\psi_{21}^*$. Since $\psi_{31}:\PP_1\rightarrow \PP_2$ is nonzero,
$\psi_{31}^*$ is a surjection onto $\II_2\cong \SL_2$. Since $\psi_{12}:\PP_4\rightarrow \PP_0$ is
the zero map, so is $\psi_{12}^*$.

Let $\gamma:\PP_1\rightarrow \PP_2$ be a nonzero map (unique up to a scalar). Then
$\gamma\psi_{22}=0$, so $\gamma^*\psi_{22}^*=0$. Hence $\psi_{22}^*$ is the map from
$\II_4$ to $\II_1$ whose image is the submodule $\begin{tikzpicture}[baseline=1ex,xscale=0.7,yscale=1,ext/.style={black,shorten <=-1pt, shorten >=-1pt}]
\draw (0,0) node (A1) {$\scriptstyle 1$};
\draw (0,0.5) node (A0) {$\scriptstyle 0$};
\draw[ext] (A1) -- (A0);
\end{tikzpicture}$.
Since $\psi_{32}\not=0$, so is $\psi_{32}^*:\II_4\rightarrow \II_2$, so it must be a surjection
onto $\II_2\cong \SL_2$. We thus have an explicit description of the map
$$\psi^*:\II_1\oplus \II_4\rightarrow \II_0\oplus \II_1\oplus \II_2.$$
Using a technique similar to the above, we can compute the kernel $\tau X_6$ of $\psi^*$
and verify that it is $X_5$.

A similar technique can be used to show that $\tau^{-1}X_3\cong X_4$.
We have
$$\Ext(X_4,X_3)\cong D\overline{\Hom}(X_3,\tau X_4)\cong D\overline{\Hom}(X_3,X_3)\cong \field.$$
Let $\varphi$ be the embedding of $X_3$ into $X_7$,
mapping it to the submodule of this form appearing on the right
hand side of the displayed $Q_{\Lambda}$-coloured quiver of this module. Then a computation similar to the above can
be done to show that $\coker\begin{pmatrix} \varphi \\ i \end{pmatrix}\cong X_4$,
where $i$ is the embedding of $X_3$ into $\PP_3$. This gives a non-split short exact sequence
$$0\rightarrow X_3\rightarrow X_7\oplus \PP_3\rightarrow X_4\rightarrow 0$$
which must be almost split.
This completes the proof.
\end{proof}

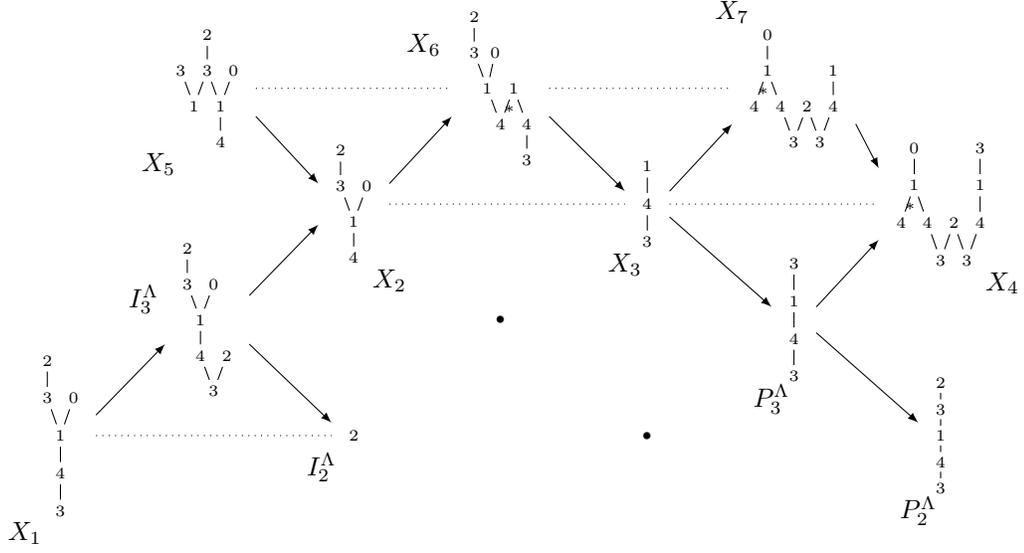
\begin{figure}
$$
\begin{tikzpicture}[xscale=0.39,yscale=0.22,baseline=(bb.base),quivarrow/.style={black, -latex},translate/.style={black, dotted}]

\path (0,0) node (bb) {}; 

\draw (0,0) node (X1) { \small
 \begin{tikzpicture} [xscale=0.7, yscale=1,ext/.style={black,shorten <=-1pt, shorten >=-1pt}]
  \draw (0,0) node (X13A) {$\scriptstyle 3$};
  \draw (0,0.5) node (X14) {$\scriptstyle 4$};
  \draw (0,1) node (X11) {$\scriptstyle 1$};
  \draw (-0.25,1.5) node (X13B) {$\scriptstyle 3$};
  \draw (0.25,1.5) node (X10) {$\scriptstyle 0$};
  \draw (-0.25,2) node (X12) {$\scriptstyle 2$};
\draw[ext] (X13A) -- (X14);
\draw[ext] (X14) -- (X11);
\draw[ext] (X11) -- (X13B);
\draw[ext] (X13B) -- (X12);
\draw[ext] (X11) -- (X10);
 \end{tikzpicture}
};

\draw (10,0) node (I2) { \small
 \begin{tikzpicture} [scale=1]
  \draw (0,0) node {$\scriptstyle 2$};
 \end{tikzpicture}
};

\draw (20,0) node { \small
 \begin{tikzpicture} [scale=1]
  \draw (0,0) node {$\scriptstyle \bullet$};
 \end{tikzpicture}
};

\draw (30,0) node (P2) { \small
 \begin{tikzpicture} [scale=0.7,yscale=1,ext/.style={black,shorten <=-1pt, shorten >=-1pt}]
  \draw (0,0) node (P23A) {$\scriptstyle 3$};
  \draw (0,0.5) node (P24) {$\scriptstyle 4$};
  \draw (0,1) node (P21) {$\scriptstyle 1$};
  \draw (0,1.5) node (P23B) {$\scriptstyle 3$};
  \draw (0,2) node (P22) {$\scriptstyle 2$};
\draw[ext] (P23A) -- (P24);
\draw[ext] (P24) -- (P21);
\draw[ext] (P21) -- (P23B);
\draw[ext] (P23B) -- (P22);
 \end{tikzpicture}
};

\draw (5,7) node (I3) { \small
 \begin{tikzpicture} [xscale=0.7,yscale=0.95,ext/.style={black,shorten <=-1pt, shorten >=-1pt}]
  \draw (0,0) node (I33A) {$\scriptstyle 3$};
  \draw (0.25,0.5) node (I32A) {$\scriptstyle 2$};
  \draw (-0.25,0.5) node (I34) {$\scriptstyle 4$};
  \draw (-0.25,1) node (I31) {$\scriptstyle 1$};
  \draw (-0.5,1.5) node (I33B) {$\scriptstyle 3$};
  \draw (0,1.5) node (I30) {$\scriptstyle 0$};
  \draw (-0.5,2) node (I32B) {$\scriptstyle 2$};
\draw[ext] (I33A) -- (I34);
\draw[ext] (I33A) -- (I32A);
\draw[ext] (I34) -- (I31);
\draw[ext] (I31) -- (I33B);
\draw[ext] (I31) -- (I30);
\draw[ext] (I33B) -- (I32B);
 \end{tikzpicture}
};

\draw (15,7) node { \small
 \begin{tikzpicture} [scale=1]
  \draw (0,0) node {$\scriptstyle \bullet$};
 \end{tikzpicture}
};

\draw (25,7) node (P3) { \small
 \begin{tikzpicture} [xscale=0.7,yscale=1,ext/.style={black,shorten <=-1pt, shorten >=-1pt}]
  \draw (0,0) node (P33A) {$\scriptstyle 3$};
  \draw (0,0.5) node (P34) {$\scriptstyle 4$};
  \draw (0,1) node (P31) {$\scriptstyle 1$};
  \draw (0,1.5) node (P33B) {$\scriptstyle 3$};
\draw[ext] (P33A) -- (P34);
\draw[ext] (P34) -- (P31);
\draw[ext] (P31) -- (P33B);
 \end{tikzpicture}
};

\draw (10,14) node (X2) { \small
 \begin{tikzpicture} [xscale=0.7,yscale=0.95,ext/.style={black,shorten <=-1pt, shorten >=-1pt}]
  \draw (-0.25,0) node (X24) {$\scriptstyle 4$};
  \draw (-0.25,0.5) node (X21) {$\scriptstyle 1$};
  \draw (-0.5,1) node (X23) {$\scriptstyle 3$};
  \draw (0,1) node (X20) {$\scriptstyle 0$};
  \draw (-0.5,1.5) node (X22) {$\scriptstyle 2$};
\draw[ext] (X24) -- (X21);
\draw[ext] (X21) -- (X23);
\draw[ext] (X23) -- (X22);
\draw[ext] (X21) -- (X20);
 \end{tikzpicture}
};

\draw (20,14) node (X3) { \small
 \begin{tikzpicture} [xscale=0.7,yscale=1,ext/.style={black,shorten <=-1pt, shorten >=-1pt}]
  \draw (0,0) node (X33) {$\scriptstyle 3$};
  \draw (0,0.5) node (X34) {$\scriptstyle 4$};
  \draw (0,1) node (X31) {$\scriptstyle 1$};
\draw[ext] (X33) -- (X34);
\draw[ext] (X34) -- (X31);
 \end{tikzpicture}
};

\draw (30,14) node (X4) {\small 
 \begin{tikzpicture} [xscale=0.7,yscale=1,ext/.style={black,shorten <=-1pt, shorten >=-1pt}]
  \draw (0,0) node (X43A) {$\scriptstyle 3$};
  \draw (-0.25,0.5) node (X44B) {$\scriptstyle 4$};
  \draw (-0.58,0.72) node {$\scriptstyle \ast$};
  \draw (-0.5,1) node (X41A) {$\scriptstyle 1$};
  \draw (-0.5,1.5) node (X40) {$\scriptstyle 0$};
  \draw (-0.75,0.5) node (X44A) {$\scriptstyle 4$};
  \draw (0.25,0.5) node (X42) {$\scriptstyle 2$};
  \draw (0.5,0) node (X43B) {$\scriptstyle 3$};
  \draw (0.75,0.5) node (X44C) {$\scriptstyle 4$}; 
  \draw (0.75,1) node (X41B) {$\scriptstyle 1$};
  \draw (0.75,1.5) node (X43C) {$\scriptstyle 3$};
\draw[ext] (X44A) -- (X41A);
\draw[ext] (X41A) -- (X40);
\draw[ext] (X44B) -- (X41A);
\draw[ext] (X43A) -- (X44B);
\draw[ext] (X43A) -- (X42);
\draw[ext] (X43B) -- (X42);
\draw[ext] (X43B) -- (X44C);
\draw[ext] (X44C) -- (X41B);
\draw[ext] (X41B) -- (X43C);
 \end{tikzpicture}
};

\draw (5,21) node (X5) { \small
 \begin{tikzpicture} [xscale=0.7,yscale=0.95,ext/.style={black,shorten <=-1pt, shorten >=-1pt}]
  \draw (0,0) node (X54) {$\scriptstyle 4$};
  \draw (0,0.5) node (X51B) {$\scriptstyle 1$};
  \draw (0.25,1) node (X50) {$\scriptstyle 0$};
  \draw (-0.25,1) node (X53B) {$\scriptstyle 3$};
  \draw (-0.25,1.5) node (X52) {$\scriptstyle 2$};
  \draw (-0.5,0.5) node (X51A) {$\scriptstyle 1$};
  \draw (-0.75,1) node (X53A) {$\scriptstyle 3$};
\draw[ext] (X54) -- (X51B);
\draw[ext] (X51B) -- (X50);
\draw[ext] (X51B) -- (X53B);
\draw[ext] (X53B) -- (X52);
\draw[ext] (X51A) -- (X53B);
\draw[ext] (X51A) -- (X53A);
 \end{tikzpicture}
};

\draw (15,21) node (X6) { \small
 \begin{tikzpicture} [xscale=0.7,yscale=0.95,ext/.style={black,shorten <=-1pt, shorten >=-1pt}]
  \draw (0,0) node (X63B) {$\scriptstyle 3$};
  \draw (0,0.5) node (X64B) {$\scriptstyle 4$};
  \draw (-0.33,0.72) node {\small $\scriptstyle \ast$}; 
  \draw (-0.25,1) node (X61B) {$\scriptstyle 1$};
  \draw (-0.5,0.5) node (X64A) {$\scriptstyle 4$};
  \draw (-0.75,1) node (X61A) {$\scriptstyle 1$};
  \draw (-0.6,1.5) node (X60) {$\scriptstyle 0$};
  \draw (-1,1.5) node (X63A) {$\scriptstyle 3$};
  \draw (-1,2) node (X62) {$\scriptstyle 2$};
\draw[ext] (X63B) -- (X64B);
\draw[ext] (X64B) -- (X61B);
\draw[ext] (X64A) -- (X61B);
\draw[ext] (X64A) -- (X61A);
\draw[ext] (X61A) -- (X63A);
\draw[ext] (X61A) -- (X60);
\draw[ext] (X63A) -- (X62);
 \end{tikzpicture}
};

\draw (25,21) node (X7) { \small
 \begin{tikzpicture} [xscale=0.7,yscale=0.95,ext/.style={black,shorten <=-1pt, shorten >=-1pt}]
  \draw (0,0) node (X73B) {$\scriptstyle 3$};
  \draw (0.25,0.5) node (X74C) {$\scriptstyle 4$};
  \draw (0.25,1) node (X71B) {$\scriptstyle 1$};
  \draw (-0.25,0.5) node (X72) {$\scriptstyle 2$};
  \draw (-0.5,0) node (X73A) {$\scriptstyle 3$};
  \draw (-0.75,0.5) node (X74B) {$\scriptstyle 4$};
  \draw (-1.08,0.72) node {\small $\scriptstyle \ast$};
  \draw (-1,1) node (X71A) {$\scriptstyle 1$};
  \draw (-1.25,0.5) node (X74A) {$\scriptstyle 4$};
  \draw (-1,1.5) node (X70) {$\scriptstyle 0$};
\draw[ext] (X74A) -- (X71A);
\draw[ext] (X71A) -- (X70);
\draw[ext] (X74B) -- (X71A);
\draw[ext] (X73A) -- (X74B);
\draw[ext] (X73A) -- (X72);
\draw[ext] (X73B) -- (X72);
\draw[ext] (X73B) -- (X74C);
\draw[ext] (X74C) -- (X71B);
 \end{tikzpicture}
};

\draw (X1.south west) node {$X_1$};
\draw ($(I2)+(240:2.2)$) node {$\II_2$};
\draw (P2.south west) node {$\PP_2$};
\draw ($(I3)+(150:2.5)$) node {$\II_3$};
\draw (P3.south west) node {$\PP_3$};
\draw (X2.south east) node {$X_2$};
\draw (X3.south west) node {$X_3$};
\draw (X4.south east) node {$X_4$};
\draw (X5.south west) node {$X_5$};
\draw ($(X6)+(135:3.7)$) node {$X_6$};
\draw (X7.north west) node {$X_{7}$};

\draw[quivarrow] (X1.30) -- (I3.210);
\draw[quivarrow] (I3.330) -- (I2.150);
\draw[quivarrow] (P3.330) -- (P2.150);
\draw[quivarrow] (I3.30) -- (X2.210);
\draw[quivarrow] (X3.330) -- (P3.150);
\draw[quivarrow] (P3.30) -- (X4.210);
\draw[quivarrow] (X5.330) -- (X2.150);
\draw[quivarrow] (X2.30) -- (X6.210);
\draw[quivarrow] (X6.330) -- (X3.150);
\draw[quivarrow] (X3.30) -- (X7.210);
\draw[quivarrow] (X7.330) -- (X4.150);

\draw[translate] (X1.east) -- (I2.west);
\draw[translate] (X2.east) -- (X3.west);
\draw[translate] (X3.east) -- (X4.west);
\draw[translate] (X5.east) -- (X6.west);
\draw[translate] (X6.east) -- (X7.west);

\end{tikzpicture}
$$
\caption{Part of the AR-quiver of $\Lambda$-mod}
\label{f:wildAR}
\end{figure}

Note that the modules in the $\tau^{\pm 1}$-orbits of $\II_2,\II_3,\PP_2,\PP_3$ are all $\tau$-rigid
(and hence rigid) by Lemma~\ref{l:hereditaryrigid} and Corollary~\ref{c:taurigid}.

\begin{proposition}
The $\Lambda$-modules $X_2,X_3,X_5$ and $X_7$ are all rigid, while $X_6$ is not rigid.
\end{proposition}
\begin{proof}
We will use Remark~\ref{r:colouredquivermorphisms} throughout.
We have
$$\Ext(X_2,X_2)\cong D\underline{\Hom}(\tau^{-1} X_2,X_2)\cong D\underline{\Hom}(X_3,X_2).$$
We have $\Hom(X_3,X_2)\cong \field$, and any nonzero map from $X_3$ to $X_2$ has image $\begin{tikzpicture}[baseline=1ex,xscale=0.7,yscale=1,ext/.style={black,shorten <=-1pt, shorten >=-1pt}]
\draw (0,0) node (A4) {$\scriptstyle 4$};
\draw (0,0.5) node (A1) {$\scriptstyle 1$};
\draw[ext] (A4) -- (A1);
\end{tikzpicture}$ and so factors through the embedding of $X_3$ into $\PP_2$.
See Figure~\ref{f:easyrigid}, where we highlight in bold the images of the
map from $X_3$ to $X_2$ and the map from $X_3$ to $\PP_2$.
It follows that $X_2$ is rigid.

We have
$$\Ext(X_3,X_3)\cong D\overline{\Hom}(X_3,\tau X_3)\cong D\overline{\Hom}(X_3,X_2).$$
In this case, any nonzero map from $X_3$ to $X_2$ factors through the embedding of
$X_3$ into $\II_3$ (see Figure~\ref{f:easyrigid}). It follows that $X_3$ is rigid.

\begin{figure}
$$
\begin{tikzpicture}[xscale=0.34,yscale=0.2,baseline=(bb.base),quivarrow/.style={black, -latex},translate/.style={black, dotted}]

\path (0,0) node (bb) {}; 

\draw (0,0) node (X3) { \small
 \begin{tikzpicture} [xscale=0.7,yscale=1,ext/.style={black,shorten <=-1pt, shorten >=-1pt}]
  \draw (0,0) node (X33) {$\scriptstyle 3$};
  \draw (0,0.5) node (X34) {$\scriptstyle 4$};
  \draw (0,1) node (X31) {$\scriptstyle 1$};
\draw[ext] (X33) -- (X34);
\draw[ext] (X34) -- (X31);
 \end{tikzpicture}
};

\draw (8,0) node (X2) { \small
 \begin{tikzpicture} [xscale=0.7,yscale=1,ext/.style={black,shorten <=-1pt, shorten >=-1pt}]
  \draw (-0.25,0) node (X24) {$\scriptstyle \mathbf{4}$};
  \draw (-0.25,0.5) node (X21) {$\scriptstyle \mathbf{1}$};
  \draw (-0.5,1) node (X23) {$\scriptstyle 3$};
  \draw (0,1) node (X20) {$\scriptstyle 0$};
  \draw (-0.5,1.5) node (X22) {$\scriptstyle 2$};
\draw[ext] (X24) -- (X21);
\draw[ext] (X21) -- (X23);
\draw[ext] (X21) -- (X20);
\draw[ext] (X23) -- (X22);
 \end{tikzpicture}
};

\draw (4,-9) node (P2) { \small
 \begin{tikzpicture} [xscale=0.7,yscale=1,ext/.style={black,shorten <=-1pt, shorten >=-1pt}]
  \draw (0,0) node (P23A) {$\scriptstyle \mathbf{3}$};
  \draw (0,0.5) node (P24) {$\scriptstyle \mathbf{4}$};
  \draw (0,1) node (P21) {$\scriptstyle \mathbf{1}$};
  \draw (0,1.5) node (P23B) {$\scriptstyle 3$};
  \draw (0,2) node (P22) {$\scriptstyle 2$};
\draw[ext] (P23A) -- (P24);
\draw[ext] (P24) -- (P21);
\draw[ext] (P21) -- (P23B);
\draw[ext] (P23B) -- (P22);
 \end{tikzpicture}
};

\draw[->] (X3.south east) -- (P2.west);
\draw[->] (P2.east) -- (X2.south west);
\draw[->] (X3.east) -- (X2.west);

\draw (X3.south west) node {$X_3$};
\draw (X2.south east) node {$X_2$};
\draw (P2.south west) node {$\PP_2$};

\begin{scope}[shift={(15,0)}]

\draw (0,0) node (X3B) { \small
 \begin{tikzpicture} [xscale=0.7,yscale=1,ext/.style={black,shorten <=-1pt, shorten >=-1pt}]
  \draw (0,0) node (X3B3) {$\scriptstyle 3$};
  \draw (0,0.5) node (X3B4) {$\scriptstyle 4$};
  \draw (0,1) node (X3B1) {$\scriptstyle 1$};
\draw[ext] (X3B3) -- (X3B4);
\draw[ext] (X3B4) -- (X3B1);
 \end{tikzpicture}
};

\draw (8,0) node (X2B) { \small
 \begin{tikzpicture} [xscale=0.7,yscale=1,ext/.style={black,shorten <=-1pt, shorten >=-1pt}]
  \draw (-0.25,0) node (X2B4) {$\scriptstyle \mathbf{4}$};
  \draw (-0.25,0.5) node (X2B1) {$\scriptstyle \mathbf{1}$};
  \draw (-0.5,1) node (X2B3) {$\scriptstyle 3$};
  \draw (0,1) node (X2B0) {$\scriptstyle 0$};
  \draw (-0.5,1.5) node (X2B2) {$\scriptstyle 2$};
\draw[ext] (X2B4) -- (X2B1);
\draw[ext] (X2B1) -- (X2B3);
\draw[ext] (X2B1) -- (X2B0);
\draw[ext] (X2B3) -- (X2B2);
 \end{tikzpicture}
};

\draw (4,-9) node (I3) { \small
  \begin{tikzpicture} [xscale=0.7,yscale=1,ext/.style={black,shorten <=-1pt, shorten >=-1pt}]
  \draw (0,0) node (I33B) {$\scriptstyle \mathbf{3}$};
  \draw (-0.25,0.5) node (I34) {$\scriptstyle \mathbf{4}$};
  \draw (-0.25,1) node (I31) {$\scriptstyle \mathbf{1}$};
  \draw (-0.5,1.5) node (I33A) {$\scriptstyle 3$};
  \draw (-0.5,2) node (I32A) {$\scriptstyle 2$};
  \draw (0.25,0.5) node (I32B) {$\scriptstyle 2$};
  \draw (0,1.5) node (I30) {$\scriptstyle 0$};
\draw[ext] (I33B) -- (I32B);
\draw[ext] (I33B) -- (I34);
\draw[ext] (I34) -- (I31);
\draw[ext] (I31) -- (I33A);
\draw[ext] (I31) -- (I30);
\draw[ext] (I33A) -- (I32A);
 \end{tikzpicture}
};

\draw[->] (X3B.south east) -- (I3.west);
\draw[->] (I3.east) -- (X2B.south west);
\draw[->] (X3B.east) -- (X2B.west);

\draw (X3B.south west) node {$X_3$};
\draw (X2B.south east) node {$X_2$};
\draw (I3.south west) node {$\II_3$};

\end{scope}
\end{tikzpicture}
$$
\caption{Rigidity of $X_2$ and $X_3$.}
\label{f:easyrigid}
\end{figure}
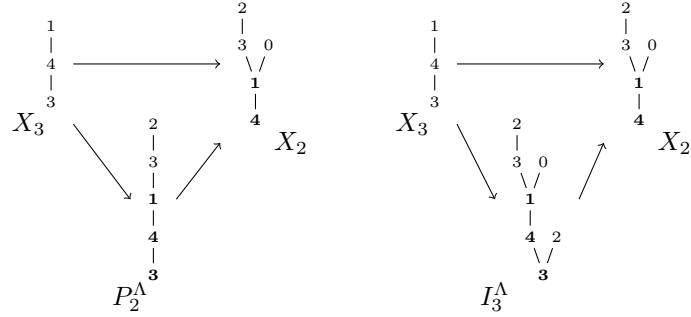

We have:
$$\Ext(X_5,X_5)\cong D\underline{\Hom}(\tau^{-1}X_5,X_5)\cong D\underline{\Hom}(X_6,X_5).$$
From the $Q_{\Lambda}$-coloured quivers of $X_5$ and $X_6$ in Figure~\ref{f:wildAR}, we see that
$\SL_1$ is a quotient of $X_6$ and is embedded into $X_5$.
Let $f_1:X_6\rightarrow X_5$ be the composition of these two maps.
From the $Q_{\Lambda}$-coloured quiver of $X_6$, we see that the module
$\begin{tikzpicture} [baseline=1ex,xscale=0.7,yscale=1,ext/.style={black,shorten <=-1pt, shorten >=-1pt}]
  \draw (0,0) node (A4) {$\scriptstyle 4$};
  \draw (0,0.5) node (A1) {$\scriptstyle 1$};
\draw[ext] (A1) -- (A4);
 \end{tikzpicture}$
is a quotient of $X_6$, and is embedded into $X_5$; let $f_2$ be the composition of the two maps.
Then it is easy to check that $\{f_1,f_2\}$ is a basis of $\Hom(X_6,X_5)$.

Furthermore, $f_1$ factors through $\PP_3$: we take the composition of the map from
$X_6$ to $\PP_3$ with image isomorphic to $X_3$ and the map from $\PP_3$ to $X_5$ whose image is the
submodule $\begin{tikzpicture}[baseline=1ex,xscale=0.7,yscale=1,ext/.style={black,shorten <=-1pt, shorten >=-1pt}]
\draw (0,0) node (A1) {$\scriptstyle 1$};
\draw (0,0.5) node (A3) {$\scriptstyle 3$};
\draw[ext] (A1) -- (A3);
\end{tikzpicture}$; see Figure~\ref{f:X5rigid}.

Note that the image of the map $f_1+f_2$ has basis given by the sum of the basis
elements of $X_5$ corresponding to the two copies of $1$ in the $Q_{\Lambda}$-coloured quiver of
$X_5$ and the basis element corresponding to the $4$; we indicate the basis elements involved
in the right hand diagram in Figure~\ref{f:X5rigid}.
The map $f_1+f_2$ factors through $\PP_2$: we take the composition of the map from $X_6$ to
$\PP_2$ with image isomorphic to $X_3$ and the map from $\PP_2$ to $X_5$ taking the basis
element corresponding to the $2$ in $\PP_2$ to the basis element corresponding to the $2$ in
$X_5$. See Figure~\ref{f:X5rigid}.
Since $\{f_1,f_1+f_2\}$ is a basis for $\Hom(X_6,X_5)$, it follows that $X_5$ is rigid. 

\begin{figure}
$$
\begin{tikzpicture}[xscale=0.34,yscale=0.2,baseline=(bb.base),quivarrow/.style={black, -latex},translate/.style={black, dotted}]

\path (0,0) node (bb) {}; 

\draw (0,0) node (X6) { \small
 \begin{tikzpicture} [xscale=0.7,yscale=1,ext/.style={black,shorten <=-1pt, shorten >=-1pt}]
  \draw (0,0) node (X63B) {$\scriptstyle 3$};
  \draw (0,0.5) node (X64B) {$\scriptstyle 4$};
  \draw (-0.33,0.72) node {\small $\scriptstyle \ast$}; 
  \draw (-0.25,1) node (X61B) {$\scriptstyle \mathbf{1}$};
  \draw (-0.5,0.5) node (X64A) {$\scriptstyle 4$};
  \draw (-0.75,1) node (X61A) {$\scriptstyle 1$};
  \draw (-0.6,1.5) node (X60) {$\scriptstyle 0$};
  \draw (-1,1.5) node (X63A) {$\scriptstyle 3$};
  \draw (-1,2) node (X62) {$\scriptstyle 2$};
\draw[ext] (X63B) -- (X64B);
\draw[ext] (X64B) -- (X61B);
\draw[ext] (X64A) -- (X61B);
\draw[ext] (X64A) -- (X61A);
\draw[ext] (X61A) -- (X60);
\draw[ext] (X61A) -- (X63A);
\draw[ext] (X63A) -- (X62);
 \end{tikzpicture}
};

\draw (8,0) node (X5) { \small
 \begin{tikzpicture} [xscale=0.7,yscale=1,ext/.style={black,shorten <=-1pt, shorten >=-1pt}]
  \draw (0,0) node (X54) {$\scriptstyle 4$};
  \draw (0,0.5) node (X51B) {$\scriptstyle 1$};
  \draw (0.25,1) node (X50) {$\scriptstyle 0$};
  \draw (-0.25,1) node (X53B) {$\scriptstyle 3$};
  \draw (-0.25,1.5) node (X52) {$\scriptstyle 2$};
  \draw (-0.5,0.5) node (X51A) {$\scriptstyle \mathbf{1}$};
  \draw (-0.75,1) node (X53A) {$\scriptstyle 3$};
\draw[ext] (X54) -- (X51B);
\draw[ext] (X51B) -- (X50);
\draw[ext] (X51B) -- (X53B);
\draw[ext] (X51A) -- (X53B);
\draw[ext] (X51A) -- (X53A);
\draw[ext] (X53B) -- (X52);
 \end{tikzpicture}
};

\draw (4,-9) node (P3) { \small
  \begin{tikzpicture} [xscale=0.7,yscale=1,ext/.style={black,shorten <=-1pt, shorten >=-1pt}]
  \draw (0,0) node (P33A) {$\scriptstyle \mathbf{3}$};
  \draw (0,0.5) node (P34) {$\scriptstyle \mathbf{4}$};
  \draw (0,1) node (P31) {$\scriptstyle \mathbf{1}$};
  \draw (0,1.5) node (P33B) {$\scriptstyle 3$};
\draw[ext] (P33A) -- (P34);
\draw[ext] (P34) -- (P31);
\draw[ext] (P31) -- (P33B);
 \end{tikzpicture}
};

\draw[->] (X6.south east) -- (P3.west);
\draw[->] (P3.east) -- (X5.south west);
\draw[->] (X6.east) to node[above] {$f_1$} (X5.west);

\draw (X6.south west) node {$X_6$};
\draw (X5.south east) node {$X_5$};
\draw (P3.south west) node {$\PP_3$};

\begin{scope}[shift={(16,0)}]

\draw (0,0) node (X6B) { \small
 \begin{tikzpicture} [xscale=0.7,yscale=1,ext/.style={black,shorten <=-1pt, shorten >=-1pt}]
  \draw (0,0) node (X6B3B) {$\scriptstyle 3$};
  \draw (0,0.5) node (X6B4B) {$\scriptstyle 4$};
  \draw (-0.33,0.72) node {\small $\scriptstyle \ast$}; 
  \draw (-0.25,1) node (X6B1B) {$\scriptstyle 1$};
  \draw (-0.5,0.5) node (X6B4A) {$\scriptstyle 4$};
  \draw (-0.75,1) node (X6B1A) {$\scriptstyle 1$};
  \draw (-0.6,1.5) node (X6B0) {$\scriptstyle 0$};
  \draw (-1,1.5) node (X6B3A) {$\scriptstyle 3$};
  \draw (-1,2) node (X6B2) {$\scriptstyle 2$};
\draw[ext] (X6B3B) -- (X6B4B);
\draw[ext] (X6B4B) -- (X6B1B);
\draw[ext] (X6B4A) -- (X6B1B);
\draw[ext] (X6B4A) -- (X6B1A);
\draw[ext] (X6B1A) -- (X6B3A);
\draw[ext] (X6B1A) -- (X6B0);
\draw[ext] (X6B3A) -- (X6B2);
 \end{tikzpicture}
};

\draw (8,0) node (X5B) { \small
 \begin{tikzpicture} [xscale=0.7,yscale=1,ext/.style={black,shorten <=-1pt, shorten >=-1pt}]
  \draw (0,0) node (X5B4) {$\scriptstyle \mathbf{4}$};
  \draw (0,0.5) node (X5B1B) {$\scriptstyle \mathbf{1}$};
  \draw (0.25,1) node (X5B0) {$\scriptstyle 0$};
  \draw (-0.25,1) node (X5B3B) {$\scriptstyle 3$};
  \draw (-0.25,1.5) node (X5B2) {$\scriptstyle 2$};
  \draw (-0.5,0.5) node (X5B1A) {$\scriptstyle \mathbf{1}$};
  \draw (-0.75,1) node (X5B3A) {$\scriptstyle 3$};
\draw[ext] (X5B4) -- (X5B1B);
\draw[ext] (X5B1B) -- (X5B0);
\draw[ext] (X5B1B) -- (X5B3B);
\draw[ext] (X5B1A) -- (X5B3B);
\draw[ext] (X5B1A) -- (X5B3A);
\draw[ext] (X5B3B) -- (X5B2);
 \end{tikzpicture}
};

\draw (4,-9) node (P2) { \small
 \begin{tikzpicture} [xscale=0.7,yscale=1,ext/.style={black,shorten <=-1pt, shorten >=-1pt}]
  \draw (0,0) node (P23A) {$\scriptstyle \mathbf{3}$};
  \draw (0,0.5) node (P24) {$\scriptstyle \mathbf{4}$};
  \draw (0,1) node (P21) {$\scriptstyle \mathbf{1}$};
  \draw (0,1.5) node (P23B) {$\scriptstyle 3$};
  \draw (0,2) node (P22) {$\scriptstyle 2$};
\draw[ext] (P23A) -- (P24);
\draw[ext] (P24) -- (P21);
\draw[ext] (P21) -- (P23B);
\draw[ext] (P23B) -- (P22);
 \end{tikzpicture}, \quad
};

\draw[->] (X6B.south east) -- (P2.west);
\draw[->] (P2.east) -- (X5B.south west);
\draw[->] (X6B.east) to node[above] {$f_1+f_2$} (X5B.west);

\draw (X6B.south west) node {$X_6$};
\draw (X5B.south east) node {$X_5$};
\draw (P2.south west) node {$\PP_2$};

\end{scope}
\end{tikzpicture}
$$
\caption{Rigidity of $X_5$.}
\label{f:X5rigid}
\end{figure}
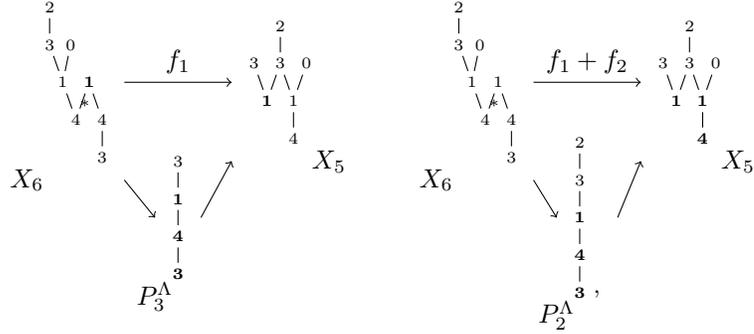

We have:
$$\Ext(X_7,X_7)\cong D\overline{\Hom}(X_7,\tau X_7)\cong D\overline{\Hom}(X_7,X_6).$$
From the $Q_{\Lambda}$-coloured quivers of $X_6$ and $X_7$ in Figure~\ref{f:wildAR}, we see that
each of the modules
$\begin{tikzpicture} [baseline=1ex,xscale=0.7,yscale=1,ext/.style={black,shorten <=-1pt, shorten >=-1pt}]
  \draw (0,0) node (A4) {$\scriptstyle 4$};
  \draw (0,0.5) node (A1) {$\scriptstyle 1$};
\draw[ext] (A1) -- (A4);
 \end{tikzpicture}$
and 
$\begin{tikzpicture} [baseline=1ex,xscale=0.7,yscale=1,ext/.style={black,shorten <=-1pt, shorten >=-1pt}]
  \draw (0,0) node (A4) {$\scriptstyle 4$};
  \draw (0,0.5) node (A1) {$\scriptstyle 1$};
  \draw (0,1) node (A0) {$\scriptstyle 0$};
\draw[ext] (A4) -- (A1);
\draw[ext] (A1) -- (A0);
\end{tikzpicture}$
is a quotient of $X_7$ and a submodule of $X_6$; we set $g_1$, $g_2$ to be the
maps from $X_7$ to $X_6$ given by the composition of the quotient map and the embedding in the
first and second case respectively. Then it is easy to check that $\{g_1,g_2\}$ is a basis of $\Hom(X_7,X_6)$.

Furthermore, $g_1$ factors through $\II_3$: we take the composition of the map from
$X_7$ to $\II_3$ with image
$\begin{tikzpicture} [baseline=2ex,xscale=0.7,yscale=1,ext/.style={black,shorten <=-1pt, shorten >=-1pt}]
  \draw (0,0) node (A3) {$\scriptstyle 3$};
  \draw (0.25,0.5) node (A4) {$\scriptstyle 4$};
  \draw (0.25,1) node (A1) {$\scriptstyle 1$};
  \draw (-0.25,0.5) node (A2) {$\scriptstyle 2$};
\draw[ext] (A1) -- (A4);
\draw[ext] (A4) -- (A3);
\draw[ext] (A2) -- (A3);
 \end{tikzpicture}$
and the map from $\II_3$ to $X_6$ with image isomorphic to $X_2$ (the composition of
the irreducible maps from $\II_3$ to $X_2$ and from $X_2$ to $X_6$);
see Figure~\ref{f:X6rigid}.

The map $g_2$ also factors through $\II_3$: we take the composition of the map from $X_7$ to
$\II_3$ with image
$\begin{tikzpicture} [baseline=2ex,xscale=0.7,yscale=1,ext/.style={black,shorten <=-1pt, shorten >=-1pt}]
  \draw (0,0) node (A3) {$\scriptstyle 3$};
  \draw (0.25,0.5) node (A4) {$\scriptstyle 4$};
  \draw (0.25,1) node (A1) {$\scriptstyle 1$};
  \draw (-0.25,0.5) node (A2) {$\scriptstyle 2$};
  \draw (0.25,1.5) node (A0) {$\scriptstyle 0$};
\draw[ext] (A0) -- (A1);
\draw[ext] (A1) -- (A4);
\draw[ext] (A4) -- (A3);
\draw[ext] (A2) -- (A3);
 \end{tikzpicture}$
and the map from $\II_3$ to $X_6$ with image isomorphic to $X_2$ considered above.
See Figure~\ref{f:X6rigid}.
Since $\{g_1,g_2\}$ is a basis for $\Hom(X_6,X_5)$, it follows that $X_6$ is rigid. 

\begin{figure}
$$
\begin{tikzpicture}[xscale=0.34,yscale=0.2,baseline=(bb.base),quivarrow/.style={black, -latex},translate/.style={black, dotted}]

\path (0,0) node (bb) {}; 

\draw (0,0) node (X7) { \small
 \begin{tikzpicture} [xscale=0.7,yscale=1,ext/.style={black,shorten <=-1pt, shorten >=-1pt}]
  \draw (0,0) node (X73B) {$\scriptstyle 3$};
  \draw (0.25,0.5) node (X74C) {$\scriptstyle 4$};
  \draw (0.25,1) node (X71B) {$\scriptstyle 1$};
  \draw (-0.25,0.5) node (X72) {$\scriptstyle 2$};
  \draw (-0.5,0) node (X73A) {$\scriptstyle 3$};
  \draw (-0.75,0.5) node (X74B) {$\scriptstyle 4$};
  \draw (-1.08,0.72) node {\small $\scriptstyle \ast$};
  \draw (-1,1) node (X71A) {$\scriptstyle 1$};
  \draw (-1.25,0.5) node (X74A) {$\scriptstyle 4$};
  \draw (-1,1.5) node (X70) {$\scriptstyle 0$};
\draw[ext] (X73B) -- (X74C);
\draw[ext] (X74C) -- (X71B);
\draw[ext] (X73B) -- (X72);
\draw[ext] (X73A) -- (X72);
\draw[ext] (X73A) -- (X74B);
\draw[ext] (X74B) -- (X71A);
\draw[ext] (X74A) -- (X71A);
\draw[ext] (X71A) -- (X70);
 \end{tikzpicture}
};

\draw (10,0) node (X6) { \small
 \begin{tikzpicture} [xscale=0.7,yscale=1,ext/.style={black,shorten <=-1pt, shorten >=-1pt}]
  \draw (0,0) node (X63B) {$\scriptstyle 3$};
  \draw (0,0.5) node (X64B) {$\scriptstyle 4$};
  \draw (-0.33,0.72) node {\small $\scriptstyle \ast$}; 
  \draw (-0.25,1) node (X61B) {$\scriptstyle 1$};
  \draw (-0.5,0.5) node (X64A) {$\scriptstyle \mathbf{4}$};
  \draw (-0.75,1) node (X61A) {$\scriptstyle \mathbf{1}$};
  \draw (-0.6,1.5) node (X60) {$\scriptstyle 0$};
  \draw (-1,1.5) node (X63A) {$\scriptstyle 3$};
  \draw (-1,2) node (X62) {$\scriptstyle 2$};
\draw[ext] (X63B) -- (X64B);
\draw[ext] (X64B) -- (X61B);
\draw[ext] (X64A) -- (X61A);
\draw[ext] (X64A) -- (X61B);
\draw[ext] (X61A) -- (X60);
\draw[ext] (X61A) -- (X63A);
\draw[ext] (X63A) -- (X62);
 \end{tikzpicture}
};

\draw (5,-10) node (I3) { \small
  \begin{tikzpicture} [xscale=0.7,yscale=1,ext/.style={black,shorten <=-1pt, shorten >=-1pt}]
  \draw (0,0) node (I33B) {$\scriptstyle \mathbf{3}$};
  \draw (-0.25,0.5) node (I34) {$\scriptstyle \mathbf{4}$};
  \draw (-0.25,1) node (I31) {$\scriptstyle \mathbf{1}$};
  \draw (-0.5,1.5) node (I33A) {$\scriptstyle 3$};
  \draw (-0.5,2) node (I32A) {$\scriptstyle 2$};
  \draw (0.25,0.5) node (I32B) {$\scriptstyle \mathbf{2}$};
  \draw (0,1.5) node (I30) {$\scriptstyle 0$};
\draw[ext] (I33B) -- (I32B);
\draw[ext] (I33B) -- (I34);
\draw[ext] (I34) -- (I31);
\draw[ext] (I31) -- (I33A);
\draw[ext] (I31) -- (I30);
\draw[ext] (I33A) -- (I32A);
 \end{tikzpicture}
};

\draw[->] (X7.south east) -- (I3.west);
\draw[->] (I3.east) -- (X6.south west);
\draw[->] (X7.east) to node[above] {$g_1$} (X6.west);

\draw (X7.south west) node {$X_7$};
\draw (X6.south east) node {$X_6$};
\draw (I3.south west) node {$\II_3$};

\begin{scope}[shift={(19,0)}]

\draw (0,0) node (X7B) { \small
 \begin{tikzpicture} [xscale=0.7,yscale=1,ext/.style={black,shorten <=-1pt, shorten >=-1pt}]
  \draw (0,0) node (X7B3B) {$\scriptstyle 3$};
  \draw (0.25,0.5) node (X7B4C) {$\scriptstyle 4$};
  \draw (0.25,1) node (X7B1B) {$\scriptstyle 1$};
  \draw (-0.25,0.5) node (X7B2) {$\scriptstyle 2$};
  \draw (-0.5,0) node (X7B3A) {$\scriptstyle 3$};
  \draw (-0.75,0.5) node (X7B4B) {$\scriptstyle 4$};
  \draw (-1.08,0.72) node {\small $\scriptstyle \ast$};
  \draw (-1,1) node (X7B1A) {$\scriptstyle 1$};
  \draw (-1.25,0.5) node (X7B4A) {$\scriptstyle 4$};
  \draw (-1,1.5) node (X7B0) {$\scriptstyle 0$};
\draw[ext] (X7B3B) -- (X7B4C);
\draw[ext] (X7B3B) -- (X7B2);
\draw[ext] (X7B3A) -- (X7B2);
\draw[ext] (X7B3A) -- (X7B4B);
\draw[ext] (X7B4C) -- (X7B1B);
\draw[ext] (X7B4B) -- (X7B1A);
\draw[ext] (X7B4A) -- (X7B1A);
\draw[ext] (X7B1A) -- (X7B0);
 \end{tikzpicture}
};

\draw (10,0) node (X6B) { \small
 \begin{tikzpicture} [xscale=0.7,yscale=1,ext/.style={black,shorten <=-1pt, shorten >=-1pt}]
  \draw (0,0) node (X6B3B) {$\scriptstyle 3$};
  \draw (0,0.5) node (X6B4B) {$\scriptstyle 4$};
  \draw (-0.33,0.72) node {\small $\scriptstyle \ast$}; 
  \draw (-0.25,1) node (X6B1B) {$\scriptstyle 1$};
  \draw (-0.5,0.5) node (X6B4A) {$\scriptstyle \mathbf{4}$};
  \draw (-0.75,1) node (X6B1A) {$\scriptstyle \mathbf{1}$};
  \draw (-0.6,1.5) node (X6B0) {$\scriptstyle \mathbf{0}$};
  \draw (-1,1.5) node (X6B3A) {$\scriptstyle 3$};
  \draw (-1,2) node (X6B2) {$\scriptstyle 2$};
\draw[ext] (X6B3B) -- (X6B4B);
\draw[ext] (X6B4B) -- (X6B1B);
\draw[ext] (X6B4A) -- (X6B1B);
\draw[ext] (X6B4A) -- (X6B1A);
\draw[ext] (X6B1A) -- (X6B0);
\draw[ext] (X6B1A) -- (X6B3A);
\draw[ext] (X6B3A) -- (X6B2);
 \end{tikzpicture}
};

\draw (5,-10) node (I3B) { \small
  \begin{tikzpicture} [xscale=0.7,yscale=1,ext/.style={black,shorten <=-1pt, shorten >=-1pt}]
  \draw (0,0) node (I3B3B) {$\scriptstyle \mathbf{3}$};
  \draw (-0.25,0.5) node (I3B4) {$\scriptstyle \mathbf{4}$};
  \draw (-0.25,1) node (I3B1) {$\scriptstyle \mathbf{1}$};
  \draw (-0.5,1.5) node (I3B3A) {$\scriptstyle 3$};
  \draw (-0.5,2) node (I3B2A) {$\scriptstyle 2$};
  \draw (0.25,0.5) node (I3B2B) {$\scriptstyle \mathbf{2}$};
  \draw (0,1.5) node (I3B0) {$\scriptstyle \mathbf{0}$};
\draw[ext] (I3B3B) -- (I3B2B);
\draw[ext] (I3B3B) -- (I3B4);
\draw[ext] (I3B4) -- (I3B1);
\draw[ext] (I3B1) -- (I3B3A);
\draw[ext] (I3B3A) -- (I3B2A);
\draw[ext] (I3B1) -- (I3B0);
 \end{tikzpicture}
};

\draw[->] (X7B.south east) -- (I3B.west);
\draw[->] (I3B.east) -- (X6B.south west);
\draw[->] (X7B.east) to node[above] {$g_2$} (X6B.west);

\draw (X7B.south west) node {$X_7$};
\draw (X6B.south east) node {$X_6$};
\draw (I3B.south west) node {$\II_3$};

\end{scope}
\end{tikzpicture}
$$
\caption{Rigidity of $X_6$.}
\label{f:X6rigid}
\end{figure}
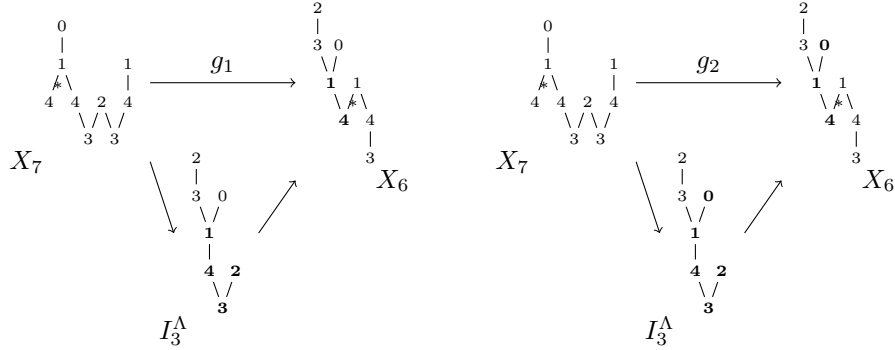

Finally, we have:
$$\Ext(X_6,X_6)\cong D\underline{\Hom}(\tau^{-1}X_6,X_6)\cong D\underline{\Hom}(X_7,X_6).$$
Consider the nonzero map $g_1:X_7\rightarrow X_6$.
The projective cover of $X_6$ is $P(X_6)\cong \PP_0\oplus \PP_1\oplus \PP_2$, so if $g_1$ factors through a
projective, it must factor through $P(X_6)$. It is easy to check directly that
$\Hom(X_7,\PP_0)=0$, $\Hom(X_7,\PP_1)=0$ and $\Hom(X_7,\PP_2)=0$, so $\Hom(X_7,P(X_6))=0$. Hence
$g_1$ does not factor through a projective and $\underline{\Hom}(X_7,X_6)\not=0$.
It follows that $X_6$ is not rigid.
\end{proof}

It is easy to check that $X_i$ is Schurian for
$i\in \{1,2,3,5,7\}$ and not Schurian for $i\in \{4,6\}$,
and that $\II_2$ and $\PP_2$ are Schurian, while $\II_3$ and
$\PP_3$ are not. This gives the picture of Schurian
and rigid modules shown on the left hand side of Figure~\ref{f:wildrigid} (using the same notation as in Figure~\ref{f:schurianrigid}),
corresponding to the modules in Figure~\ref{f:wildAR} (with $X_4$ omitted, as we have not checked if it is rigid).

In a tube of rank $3$, a module is rigid if and only if it has quasilength at most $2$, which is also the case in the
regular component $\R$. On the right hand side of Figure~\ref{f:wildrigid},
we show the pattern of $\tau$-rigid, rigid and Schurian
modules corresponding to the indecomposable objects in a tube of rank $3$.
This is from the tame case in Example~\ref{ex:running}, which was shown in Figure~\ref{f:tameAR}.

It is interesting to note the similarity of
the pattern of $\tau$-rigid, rigid and Schurian $\Lambda$-modules in these two cases, and to ask what the pattern
is for the whole of $\R$.

\begin{figure}
$$
\begin{tikzpicture}[scale=0.33,baseline=(bb.base),quivarrow/.style={black, -latex, shorten >=10pt,shorten <=10pt},translate/.style={black, dotted, shorten >=13pt, shorten <=13pt}]

\path (0,0) node (bb) {}; 

\draw (0,0) node {$\taurigid$};
\draw (2,0) node {$\taurigid$};
\draw (4,0) node {$\qhole$};
\draw (6,0) node {$\taurigid$};
\draw (1,1) node {$\notschuriantaurigid$};
\draw (3,1) node {$\qhole$};
\draw (5,1) node {$\notschuriantaurigid$};
\draw (2,2) node {$\schurianrigidnottaurigid$};
\draw (4,2) node {$\schurianrigidnottaurigid$};
\draw (1,3) node {$\schurianrigidnottaurigid$};
\draw (3,3) node {$\notschuriannotrigid$};
\draw (5,3) node {$\schurianrigidnottaurigid$};

\draw (3,-2) node {Wild case};

\begin{scope}[shift={(10,0)}]

\draw (3,-2) node {Tame case};

\draw (0,0) node {$\taurigid$};

\draw (2,0) node {$\taurigid$};
\draw (4,0) node {$\qhole$};
\draw (6,0) node {$\taurigid$};
\draw (1,1) node {$\notschuriantaurigid$};
\draw (3,1) node {$\qhole$};
\draw (5,1) node {$\notschuriantaurigid$};
\draw (2,2) node {$\schurianrigidnottaurigid$};
\draw (4,2) node {$\schurianrigidnottaurigid$};
\draw (1,3) node {$\schurianrigidnottaurigid$};
\draw (3,3) node {$\notschuriannotrigid$};
\draw (5,3) node {$\schurianrigidnottaurigid$};
\draw (2,4) node {$\notschuriannotrigid$};
\draw (4,4) node {$\notschuriannotrigid$};

\node [fill=white,inner sep=0pt] at (6,2) {$\notschuriannotrigid$};
\node [fill=white,inner sep=0pt] at (6,4) {$\schuriannotrigid$};
\node [fill=white,inner sep=0pt] at (0,2) {$\notschuriannotrigid$};
\node [fill=white,inner sep=0pt] at (0,4) {$\schuriannotrigid$};

\begin{scope}[on background layer]
\draw[dotted] (0,0) -- (0,6);
\draw[dotted] (6,0) -- (6,6);
\end{scope}

\end{scope}
\end{tikzpicture}
$$
\caption{$\tau$-rigid, rigid and Schurian $\Lambda$-modules in part of a wild example (left hand diagram).
In the right hand diagram we recall the $\tau$-rigid, rigid and Schurian modules from the tame
case in Example~\ref{ex:running} shown in Figure~\ref{f:tameAR}.}
\label{f:wildrigid}
\end{figure}
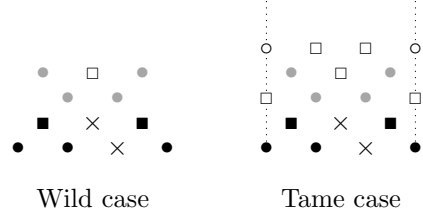

\section{A counter-example}
\label{s:counterexample}

In this section, we give the counter-example promised in the introduction.
This concerns the relationship with cluster algebras. For background on cluster
algebras, we refer to~\cite{FZ02,FZ07}. We fix a finite quiver $Q$ with no loops or $2$-cycles
and label its vertices $1,2,\ldots ,n$. Let $\mathbb{F}=\mathbb{Q}(x_1,\ldots, x_n)$ be the
field of rational functions in $n$ indeterminates over $\mathbb{Q}$. Then the associated cluster algebra
$\mathcal{A}(Q)$ is a subalgebra of $\mathbb{F}$. Here, cluster variables and clusters play a key role.
The initial cluster variables are $x_1,\ldots ,x_n$. The non-initial cluster variables can be
written in reduced form $f/m$, where $m$ is a monomial in the variables $x_1,\ldots ,x_n$, $f\in \mathbb{Q}[x_1,\ldots ,x_n]$
and $x_i\nmid f$ for all $i$.
Writing $m=x_1^{d_1}\cdots x_n^{d_n}$, where $d_i\geq 0$, we obtain a vector $(d_1,\ldots ,d_n)$, which
is called the \emph{$d$-vector} associated with the cluster variable $f/m$.

On the other hand, let $M$ be an indecomposable finite dimensional $\field Q$-module, and let $S_1,\ldots ,S_n$
be the nonisomorphic simple $\field Q$-modules. Then we have an associated \emph{dimension vector} $(d'_1,\ldots ,d'_n)$,
where $d'_i$ denotes the multiplicity of the simple module $S_i$ as a composition factor of $M$.

It was then of interest to investigate a possible relationship between the denominator vectors and the dimension
vectors of the indecomposable rigid $\field Q$-modules.
In the case where $Q$ is acyclic, the two sets coincide (see~\cite{BMRT07,BCKMRT07,CK06}).
When $Q$ is not acyclic, we do not have such a nice correspondence in general, but there are results in this
direction in~\cite{AD13,BM10,BMR09}. We have found the following example of a $d$-vector which is
not the dimension vector of an indecomposable $\field Q$-module.

\begin{example} \label{ex:counter}
Let $Q$ be the acyclic quiver from Example~\ref{ex:running}:
\vskip 0.05cm
\begin{equation}
\label{e:examplequiveragain}
\xymatrix{
1 \ar[r] \ar@/^1pc/[rrr] & 2 \ar[r] & 3 \ar[r] & 4.
}
\end{equation}
and let $\Lambda$ be the cluster-tilted algebra from this example. The quiver $Q_{\Lambda}$ of $\Lambda$ is
shown in Figure~\ref{f:tametiltedquiver}, and can be obtained from $Q$ by mutating at $2$ and then
at $3$. Recall that the AR-quiver for the largest tube in $\field Q$-mod (which has rank $3$) is shown in Figure~\ref{f:tamehereditaryAR} and the corresponding part of the AR-quiver for $\Lambda$-mod is shown in Figure~\ref{f:tameAR}. Let $M$ be the $\field Q$-module
$\begin{tikzpicture} [baseline=1ex,xscale=0.7,yscale=1,ext/.style={black,shorten <=-1pt, shorten >=-1pt}]
  \draw (0,0) node (A4) {$\scriptstyle 4$};
  \draw (-0.25,0.5) node (A1) {$\scriptstyle 1$};
  \draw (0.25,0.5) node (A3) {$\scriptstyle 3$};
\draw[ext] (A1) -- (A4);
\draw[ext] (A3) -- (A4);
 \end{tikzpicture}$, which is of quasilength $2=3-1$ in the tube
in Figure~\ref{f:tamehereditaryAR}. The corresponding $\Lambda$-module, $\clu{M}=\II_3=
\begin{tikzpicture} [baseline=3ex,scale=0.7,yscale=1,ext/.style={black,shorten <=-1pt, shorten >=-1pt}]
  \draw (0,0) node (I33A) {$\scriptstyle 3$};
  \draw (0.25,0.5) node (I34) {$\scriptstyle 4$};
  \draw (0.25,1) node (I31) {$\scriptstyle 1$};
  \draw (0.25,1.5) node (I33B) {$\scriptstyle 3$};
  \draw (0.25,2) node (I32A) {$\scriptstyle 2$};
  \draw (-0.25,0.5) node (I32B) {$\scriptstyle 2$};
\draw[ext] (I33A) -- (I34);
\draw[ext] (I34) -- (I31);
\draw[ext] (I31) -- (I33B);
\draw[ext] (I33B) -- (I32A);
\draw[ext] (I32B) -- (I33A);
 \end{tikzpicture},$
has dimension vector $(1,2,2,1)$. Then we know from~\cite[Thm.\ A]{BM10} that the denominator vector of the
corresponding cluster variable in the cluster algebra $\mathcal{A}(Q_{\Lambda})$ is $(1,2,1,1)=(1,2,2,1)-(0,0,1,0)$.
It is then easy to see that $(1,2,1,1)$ cannot occur as the dimension vector of any indecomposable
$\field Q_{\Lambda}$-module, by looking at an arbitrary representation with this dimension vector:
$$
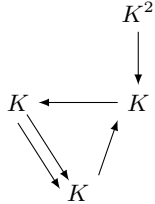
\begin{tikzpicture}[scale=0.8,baseline=(bb.base),quivarrow/.style={black, -latex},translate/.style={black, dotted},relation/.style={black, dotted, thick=2pt}]

\path (0,0) node (bb) {}; 

\draw (0,0) node (X1) {\small $\field$};
\draw (2,0) node (X3) {\small $\field$};
\draw (2,1.5) node (X2) {\small $\field^2$};
\draw (1,-1.5) node (X4) {\small $\field$};

\draw[quivarrow] (X2.south) -- (X3.north);
\draw[quivarrow] (X3.west) -- (X1.east);
\draw[quivarrow] (X4.north east) -- (X3.south west);
\draw[quivarrow] (0.15,-0.2) -- (0.85,-1.3);
\draw[quivarrow] (0,-0.34) -- (0.7,-1.44);

\end{tikzpicture}
$$
Here, a nonzero summand of $\field^2$ has to split off, so that $M$ cannot be indecomposable. Hence we have
found a $d$-vector which is not the dimension vector of any indecomposable $\field Q_{\Lambda}$-module.
Note that it cannot be the dimension vector of any indecomposable $\Lambda$-module either, by
the same argument.
\end{example}

There is another interesting class of vectors occurring in the theory of cluster algebras, known as
the $c$-vectors. They were introduced in~\cite{FZ07} (see~\cite{FZ07} for the definition).
In the case of an acyclic quiver $Q$ it is known that the set of (positive) $c$-vectors coincides
with the set of real Schur roots (see~\cite{chavez,ST13}), that is, the dimension vectors of the
indecomposable rigid $\field Q$-modules.

But the relationship between $c$-vectors and $d$-vectors is not so nice in the general case.
It is known for any finite quiver $Q$ without loops or two-cycles that each positive $c$-vector of $Q$ is the dimension vector of a finite dimensional Schurian rigid module over an appropriate
Jacobian algebra with quiver $Q$ (\cite{nagao13}; see~\cite[Thm. 14]{chavez}). As pointed out
in~\cite{NS14}, this implies that every positive $c$-vector of $Q$ is a Schur root of $Q$, hence
a root of $Q$. Then we get the following:

\begin{proposition}
There is a finite quiver $Q$ without loops or $2$-cycles for which the set of $d$-vectors associated
to $\mathcal{A}(Q)$ is not contained in the set of positive $c$-vectors of $\mathcal{A}(Q)$.
\end{proposition}
\begin{proof}
We consider the quiver $Q_{\Lambda}$ in Example~\ref{ex:counter}. In this case, the set of $d$-vectors
is not contained in the set of dimension vectors of the indecomposable $\field Q_{\Lambda}$-modules.
If the set of $d$-vectors of $Q_{\Lambda}$ was contained in the set of positive $c$-vectors of $Q_{\Lambda}$,
then we would have a contradiction, since, as we mentioned above, every positive $c$-vector of $Q_{\Lambda}$
is the dimension vector of an indecomposable $\field Q_{\Lambda}$-module.
\end{proof}

\section{Three dimension vectors}
\label{s:three}
We have seen in Section~\ref{s:counterexample} that there is a cluster-tilted algebra $\Lambda$ associated
to a quiver of tame representation type with the property that not every $d$-vector of $\mathcal{A}(Q_{\Lambda})$
is the dimension vector of an indecomposable $\Lambda$-module.
So we can ask if it is possible to express each such $d$-vector as a sum of a small number of such dimension vectors.
Our final result shows that, for a cluster-tilted algebra $\Lambda$ associated to a quiver of tame representation
type, it is always possible to write a $d$-vector for $\mathcal{A}(Q_{\Lambda})$ as the sum of at most three dimension vectors of indecomposable rigid $\Lambda$-modules.

We do not know whether it is possible to write every $d$-vector for $\mathcal{A}(Q_{\Lambda})$ as a sum of at most
two dimension vectors of indecomposable rigid $\Lambda$-modules. It would also be interesting to know whether
analogous results hold in the wild case.

As before, we fix a quiver $Q$ of tame representation type. We fix an arbitrary
cluster-tilting object $T$ in the corresponding cluster category, $\C$.
Suppose $M$ is an object in $\C$, with corresponding
$\Lambda$-module $\clu{M}$. The vertices of the quiver of $\Lambda=\End_{\C}(T)$
are indexed by the indecomposable direct summands $\ind(T)$ of $T$. The dimension
vector of $\clu{M}$ is given by the tuple $(\dimv_V(M))_V$, where $V$ varies
over the indecomposable direct summands of $T$. We have:
$$\dimv_V(M)=\dim \Hom_{\C}(V,M)=\dim \Hom(V,M)+\dim \Hom(M,\tau^2 V).$$
We shall also write $\dimv_V(\clu{M})$ for $\dimv_V(M)$.
Note that if $M$ lies in $\add(\tau T)$ then $\clu{M}=0$ and $\dimv_V(M)=0$ for all
$V\in \ind(T)$.

If $M$ is (induced by) an indecomposable module in $\T$, then there is a
mesh $\M_M$ in the AR-quiver of $\T$ corresponding to the almost split
sequence with last term $M$. This is displayed in Figure~\ref{f:mesh}, with the diagram on the left indicating the case when $M$ is on the border of $\T$. We denote the middle term whose
quasilength is greater (respectively, smaller) than that of $M$ by $M_U$
(respectively, $M_L$). Note that if $M$ is on the border of $\T$ then $M_L$ does not exist.

\begin{figure}
$$
\begin{tikzpicture}[scale=1,baseline=(bb.base),quivarrow/.style={black, -latex, shorten >=10pt,shorten <=10pt},translate/.style={black, dotted, shorten >=13pt, shorten <=13pt}]
  
\path (0,0) node (bb) {}; 

\draw (0,0) node {$\tau M$};
\draw (2,0) node {$M$};
\draw (1,1) node {$M_U$};

\draw[quivarrow] (0,0) -- (1,1);
\draw[quivarrow] (1,1) -- (2,0);

\draw[translate] (0,0) -- (2,0);

\end{tikzpicture}
\quad\quad\quad
\begin{tikzpicture}[scale=1,baseline=(bb.base),quivarrow/.style={black, -latex, shorten >=10pt,shorten <=10pt},translate/.style={black, dotted, shorten >=13pt, shorten <=13pt}]
  
\path (0,0) node (bb) {}; 

\draw (0,0) node {$\tau M$};
\draw (2,0) node {$M$};
\draw (1,1) node {$M_U$};
\draw (1,-1) node {$M_L$};

\draw[quivarrow] (0,0) -- (1,1);
\draw[quivarrow] (1,1) -- (2,0);
\draw[quivarrow] (0,0) -- (1,-1);
\draw[quivarrow] (1,-1) -- (2,0);

\draw[translate] (0,0) -- (2,0);

\end{tikzpicture}
$$
\caption{The mesh ending at an indecomposable object in $\T$.
The diagram on the left indicates the case where $M$ is on the border of $\T$.}
\label{f:mesh}
\end{figure}
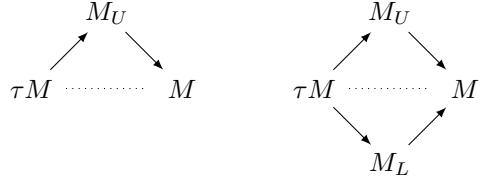

For objects $X,Y$ of $\C$ we shall write
$$\delta_{X,Y}=\begin{cases} 1, & \text{if $X\cong Y$;} \\ 0, & otherwise. \end{cases}$$

\begin{lemma} \label{l:mesh}
Let $M$ be an indecomposable object in $\T$ with mesh $\M_M$ as above.
Then:
$$\dimv_V(\clu{M})=\dimv_V(\clu{M_U})+\dimv_V(\clu{M_L})-\dimv_V(\tau \clu{M})+
\delta_{V,M},$$
where the terms involving $M_L$ do not appear if $M$ is on the border of $\T$.
\end{lemma}

\begin{proof}
If $V\not\cong M$ then the mesh ending at $\clu{M}$ in
$\Lambda$-mod is the image under $\Hom_{\C}(T,-)$ of the mesh ending at $M$ in
$\C$ (deleting zero modules corresponding to summands of $\tau T$).
If $V\cong M$ then $\clu{M}$ is an indecomposable projective module,
so $\rad(\clu{M})\cong \clu{M_L}\oplus \clu{M_U}$.
\end{proof}

We assume for the rest of this section that there is an indecomposable
direct summand $T_0$ of $T_{\T}$ with the property that every indecomposable
direct summand of $T_{\T}$ lies in the wing $\W_{T_0}$.
(In the notation at the beginning of Section~\ref{s:tube}, we have $s=1$).

We assume further that the quasilength of $T_0$ (i.e.\ $l_0$) is equal to
$r-1$. We arrange the labelling, for simplicity, so that the quasisimple modules
in $\W_{\tau T}$ are the $Q_i$ with $i\in [0,r-2]$, so in the notation from
Section~\ref{s:tube}, $i_0=0$.
Let
\begin{equation}
\label{e:Ddef}
D=\{M_{i,l}\,:\,1\leq i\leq r-1,r-i\leq l\leq 2r-2-i\}.
\end{equation}
Note that $D$ can be formed from $\wideR{0}$ and its reflection in the line $L$ through
the modules of quasilength $r-1$. It is a diamond-shaped region, with leftmost corner
$T_0\cong M_{1,r-1}$ and rightmost corner $\tau^2 T_0\cong M_{r-1,r-1}$. The lowest point is the unique
quasisimple module $Q_{r-1}$ not in $\W_{\tau T_0}$ and the highest point is the same as
the highest point $M_{1,2r-3}$ of $\wideR{0}$, immediately below $\Top_0$; see
Figure~\ref{f:regionD}.

\begin{figure}
$$
\begin{tikzpicture}[xscale=0.30,yscale=0.35,baseline=(bb.base),quivarrow/.style={black, -latex, shorten >=10pt,shorten <=10pt},translate/.style={black, dotted, shorten >=13pt, shorten <=13pt}]

\newcommand{\rank}{11}
\newcommand{\copies}{2}

\pgfmathtruncatemacro{\height}{\rank-1};
\pgfmathtruncatemacro{\rankm}{\rank-1};
\pgfmathtruncatemacro{\rankmm}{\rank-2};
\pgfmathtruncatemacro{\heightm}{\height-1};
\pgfmathtruncatemacro{\heighth}{\height/2-1};

\pgfmathparse{\rank*\copies}\let\totwidth\pgfmathresult;
\pgfmathparse{\rank*\copies-1}\let\totwidthm\pgfmathresult;

\path (0,0) node (bb) {}; 

\foreach \i in {0,...,\totwidth}{
\foreach \j in {0,...,\heighth}{
\draw (\i*2,\j*2) node[black] {$\scriptstyle \circ$};}}

\foreach \i in {0,...,\totwidthm}{
\foreach \j in {0,...,\heighth}{
\draw (\i*2+1,\j*2+1) node[black] {$\scriptstyle \circ$};}}

\foreach \i/\j in
{0/10,2/10,4/10,6/10,8/10,14/10,16/10,18/10,20/10,22/10,24/10,26/10,28/10,30/10,
36/10,38/10,40/10,42/10,44/10,
1/11,3/11,5/11,7/11,11/11,15/11,17/11,19/11,21/11,23/11,25/11,27/11,29/11,33/11,
37/11,39/11,41/11,43/11,
0/12,2/12,4/12,6/12,10/12,12/12,16/12,18/12,20/12,22/12,24/12,26/12,28/12,32/12,34/12,
38/12,40/12,42/12,44/12,
1/13,3/13,5/13,9/13,11/13,13/13,17/13,19/13,21/13,23/13,25/13,27/13,31/13,33/13,35/13,
39/13,41/13,43/13,
0/14,2/14,4/14,8/14,10/14,12/14,14/14,18/14,20/14,22/14,24/14,26/14,30/14,32/14,34/14,
36/14,40/14,42/14,44/14,
1/15,3/15,7/15,9/15,11/15,13/15,15/15,19/15,21/15,23/15,25/15,29/15,31/15,33/15,35/15,
37/15,41/15,43/15,
0/16,2/16,6/16,8/16,10/16,12/16,14/16,16/16,20/16,22/16,24/16,28/16,30/16,32/16,34/16,
36/16,38/16,42/16,44/16,
1/17,5/17,7/17,9/17,11/17,13/17,15/17,17/17,21/17,23/17,27/17,29/17,31/17,33/17,35/17,
37/17,39/17,43/17,
0/18,4/18,6/18,8/18,10/18,12/18,14/18,16/18,18/18,22/18,26/18,28/18,30/18,32/18,34/18,
36/18,38/18,40/18,44/18,
3/19,5/19,7/19,9/19,11/19,13/19,15/19,17/19,19/19,25/19,27/19,29/19,31/19,33/19,35/19,
37/19,39/19,41/19,
2/20,4/20,6/20,8/20,10/20,12/20,14/20,16/20,18/20,20/20,24/20,26/20,28/20,30/20,32/20,34/20,
36/20,38/20,40/20,42/20}
{\draw (\i,\j) node {$\scriptstyle \circ$};}

\foreach \i/\j in
{12/10,13/11,14/12,15/13,16/14,17/15,18/16,19/17,20/18,21/19,22/20,
23/19,24/18,25/17,26/16,27/15,28/14,29/13,30/12,31/11,32/10,
34/10,35/11,36/12,37/13,38/14,39/15,40/16,41/17,42/18,43/19,44/20,
0/20,1/19,2/18,3/17,4/16,5/15,6/14,7/13,8/12,9/11,10/10}
{\draw (\i,\j) node {$\circledcirc$};}

\begin{scope}[on background layer]

\draw[dotted] (0,0) -- (0,2*\height+2);
\draw[dotted] (2*\rank,0) -- (2*\rank,2*\height+2);
\draw[dotted] (2*\totwidth,0) -- (2*\totwidth,2*\height+2);


\draw[dashed] (4,0) -- (\rank+2,\rank-2) -- (2*\rank,0);
\draw[dashed] (2*\rank+4,0) -- (2*\rank+\rank+2,\rank-2) -- (2*\rank+2*\rank,0);


\draw[dashed] (0,\rankmm) -- (\totwidth*2,\rankmm);

\draw[draw=none,fill=gray!30,opacity=0.4] (\rank+2,\rank-2) -- (2*\rank,0) -- (2*\rank+\rank-2,\rank-2) -- (2*\rank,2*\rank-4);

\draw[draw=none,fill=gray!30,opacity=0.4] (0,0) -- (\rank-2,\rank-2) -- (0,2*\rank-4);
\draw[draw=none,fill=gray!30,opacity=0.4] (4*\rank,0) -- (4*\rank-\rank+2,\rank-2) -- (4*\rank,2*\rank-4);

\end{scope}

\draw (\rank+2,\rank-2+0.55) node[black] {$\scriptstyle  T_0$};
\draw (3*\rank+2,\rank-2+0.55) node[black] {$\scriptstyle T_0$};

\draw (\rank,\rank-2+0.45) node[black] {$\scriptstyle  \tau T_0$};
\draw (3*\rank,\rank-2+0.45) node[black] {$\scriptstyle \tau T_0$};

\draw (0,2*\rank-2+0.6) node[black] {$\scriptstyle \Top_0$};
\draw (2*\rank,2*\rank-2+0.6) node[black] {$\scriptstyle \Top_0$};
\draw (4*\rank,2*\rank-2+0.6) node[black] {$\scriptstyle \Top_0$};

\draw (4.4,-0.9) node {$\scriptstyle M_{1,1}$};
\draw (\rank*2+4.4,-0.9) node {$\scriptstyle M_{1,1}$};
\draw (0,-0.9) node {$\scriptstyle M_{\rankm,1}$};
\draw (2*\rank,-0.9) node {$\scriptstyle M_{\rankm,1}$};
\draw (4*\rank,-0.9) node {$\scriptstyle M_{\rankm,1}$};


\draw (2.2,0.5+\rankmm) node {$\scriptstyle L$};

\newcommand{\Mi}{4}
\newcommand{\Ml}{9}
\pgfmathtruncatemacro{\XMl}{\rank-\Mi-1};
\pgfmathtruncatemacro{\YMl}{\Mi+\Ml};


\draw (\Mi*2+2+\Ml-1,\Ml-1) node (M) {};
\draw (\Mi*2+2+\Ml-1,\Ml-1+0.5) node {$\scriptstyle M$};


\draw (\Mi*2+2+\XMl-1,\XMl-1) node (XM) {};
\draw (\Mi*2+2+\XMl-1,\XMl-1-0.55) node {$\scriptstyle X_M$};


\draw (2+\YMl-1,\YMl-1) node (YM) {};
\draw (2+\YMl-1,\YMl-1+0.7) node {$\scriptstyle Y_M$};


\draw (19,3) node (XMP) {};
\draw (18.9,2.35) node {$\scriptstyle X'_M$};


\foreach \i/\j in {1/1,1/10,3/3,3/8,4/2,5/1,7/1,7/4,9/2,10/1}
{\draw (\i*2+\j+1,\j-1) node {$\bullet$};
\draw (\i*2+\j+1+\totwidth,\j-1) node {$\bullet$};}


\draw (28,4) node (ZM) {};
\draw (28,3.5) node {$\scriptstyle Z_M$};


\draw (24,8) node (ZMP) {};
\draw (24,8.5) node {$\scriptstyle Z'_M$};

\begin{scope}[on background layer]

\draw[dashed] ($(XM)$) -- ($(M)$);


\draw[dashed] ($(YM)$) -- ($(M)$);


\draw[dashed] ($(XMP)$) -- ($(ZMP)$);


\draw[dashed] ($(ZMP)$) -- ($(ZM)$);

\draw (23,0) -- (22.5,-0.5);
\draw (21.3,-0.3) -- (19.2,1.8);
\draw (18.2,2.8) -- (16.2,4.8) -- (17.2,5.8) -- (23,0);

\draw (15.8,5.2) -- (16.8,6.2) -- (13.5,9.5);
\draw (12.5,9.5) -- (12,9) -- (15.8,5.2);

\end{scope}

\end{tikzpicture}
$$
\caption{Here the rank of the tube is $11$. The region $D$ is indicated
by the shaded area. The filled-in circles indicate the indecomposable direct
summands of $T$, and the double circles indicate the elements of $\HH_0$.
The line $L$ divides the region $D$ into two and $\wideR{0}$ consists of
the vertices in $D$ on and above the line. The upper boxed region is $\I_M$
(see~\eqref{e:IM}) and the lower boxed region is $\I'_M$ (see~\eqref{e:IMP}).}
\label{f:regionD}
\end{figure}
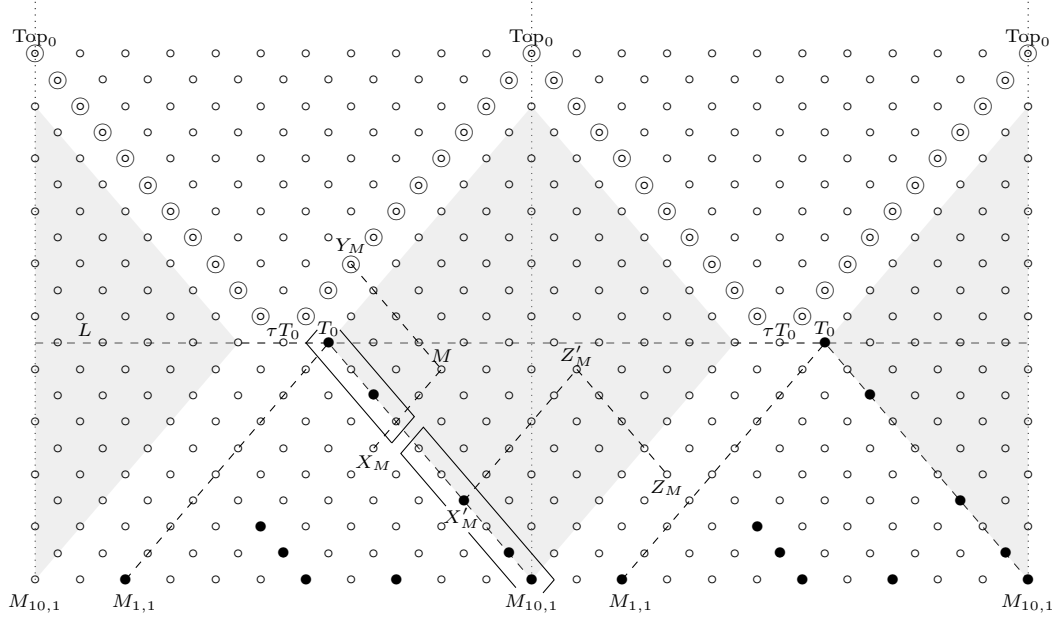

Given an indecomposable module $M=M_{i,l}\in D$, we define:
\begin{equation}
\label{e:IM}
\I_M=\{M_{j,r-j}\,:\,1\leq j\leq i\},
\end{equation}
i.e.\ the set of indecomposable modules which are injective in $\W_{T_0}$ and
lie above or on the (lowest) intersection point, $M_{i,r-i}$, between the ray through $M$ and the
coray through $T_0$. We also set
$$X_M=M_{i,r-i-1}, \quad\quad Y_M=M_{0,i+l}.$$
Note that $X_M$ is the object in the part of the ray through $M$ below $M$ which is of maximal quasilength subject to not lying in $D$.
Similarly, $Y_M$ is the nearest object to $M$ in the part of the coray through $M$ above $M$, which
is of minimal quasilength subject to not lying in $D$. See Figure~\ref{f:regionD}.

\begin{lemma} \label{l:dimvector}
Let $M\in D$ and let $V$ be an indecomposable summand of $T$. Then we have 
$$
\dimv_V(M)=
\begin{cases}
\dimv_V(X_M)+\dimv_V(Y_M)+1, & V\in \I_M; \\
\dimv_V(X_M)+\dimv_V(Y_M), & V\not\in \I_M.
\end{cases}
$$
\end{lemma}

\begin{proof}
Suppose first that $V$ is preprojective, i.e.\ $V$ is an indecomposable
direct summand of $U$. Since $X_M\in \W_{\tau T_0}$ we have $\Hom_{\C}(V,X_M)=0$.
Note that by Lemmas~\ref{l:additive} and~\ref{l:onezero},
$$\dim \Hom_{\C}(V,Q_i)=
\begin{cases} 0, & 0\leq i\leq r-2; \\
1, & i=r-1.
\end{cases}
$$
It follows from Lemma~\ref{l:additive} that $\dimv_V(X)=0$ for any module $X\in \W_{\tau T_k}$.
By Proposition~\ref{p:leq1}, $\dimv_V(Q_{r-1})=1$, noting that $Q_{r-1}$ is the unique
quasisimple in $\T$ not in $\W_{\tau T_k}$. Using additivity as in the proof of Lemma~\ref{l:additive}, we see that $\dimv_V(X)=1$ if $X\in D\cup \HH_0$. Since $X_M\in \W_{\tau T_0}$,
$Y_M\in \HH_0$ and $M\in D$, we have $\dimv_V(M)=1$, $\dimv_V(X_M)=0$ and $\dimv_V(Y_M)=1$, giving the
result in this case.

So we may assume that $V$ is an indecomposable direct summand of $T_{\T}$.
We prove the result in this case by induction on the minimal length of a
path in $\T$ from $T_0$ to $M$.
The base case is $M\cong T_0$. Then $\I_M=\{T_0\}$.
Since $\dimv_V(\tau M)=0$, the result in this case follows directly from Lemma~\ref{l:mesh}.

We assume that $M\not\cong T_0$ and that the result is proved in the case where
the minimal length of a path in $\T$ from $T_0$ to $M$ is smaller.
In particular, the result is assumed to be true for all modules in $\M_M\cap D$
other than $M$.

\noindent \textbf{Case I}:
If $M=M_{i,r-i}$, with $1\leq i\leq r-1$ lies on the lower left boundary of $D$ then
$\M_M\cap D=\{M_U,M\}$. Applying the inductive hypothesis to $M_U$ and noting that
$Y_{M_U}=Y_M$, we have:
\begin{equation}
\label{e:caseI}
\dimv_V(M_U)=
\begin{cases}
\dimv_V(X_{M_U})+\dimv_V(Y_M)+1, & V\in \I_{M_U}; \\
\dimv_V(X_{M_U})+\dimv_V(Y_M), & V\not\in \I_{M_U};
\end{cases}
\end{equation}
Note that $M_L=X_M$, $\tau M=X_{M_U}$, $Y_{M_U}=Y_M$ and $\I_M=\I_{M_U}\cup \{M\}$
(see Figure~\ref{f:dimvector}). By Lemma~\ref{l:mesh} and~\eqref{e:caseI}, we have:
\begin{align*}
\begin{split}
\dimv_V(M) &= \dimv_V(M_U)+\dimv_V(M_L)-\dimv_V(\tau M)+\delta_{V,M} \\
       &= \dimv_V(M_U)+\dimv_V(X_M)-\dimv_V(X_{M_U})+\delta_{V,M} \\
&= 
\begin{cases}
\dimv_V(X_{M_U})+\dimv_V(Y_M)+\dimv_V(X_M)-\dimv_V(X_{M_U})+\delta_{V,M}+1, & V\in \I_{M_U}; \\
\dimv_V(X_{M_U})+\dimv_V(Y_M)+\dimv_V(X_M)-\dimv_V(X_{M_U})+\delta_{V,M}, & V\not\in \I_{M_U};
\end{cases} \\
&=
\begin{cases}
\dimv_V(X_M)+\dimv_V(Y_M)+1, & V\in \I_M; \\
\dimv_V(X_M)+\dimv_V(Y_M), & V\not\in \I_M.
\end{cases}
\end{split}
\end{align*}

\noindent \textbf{Case II}: If $M=M_{1,l}$ where $r\leq l\leq 2r-3$ lies on the upper left boundary of $D$
then $\M_M\cap D=\{M_L,M\}$.
Applying the inductive hypothesis to $M_L$ and noting that $X_{M_L}=X_M$
(see Figure~\ref{f:dimvector}), we have:
\begin{equation}
\label{e:caseII}
\dimv_V(M_L)=
\begin{cases}
\dimv_V(X_M)+\dimv_V(Y_{M_L})+1, & V\in \I_{M_L}; \\
\dimv_V(X_M)+\dimv_V(Y_{M_L}), & V\not\in \I_{M_L}.
\end{cases}
\end{equation}
Note that $M_U=Y_M$, $\tau M=Y_{M_L}$, $\I_M=\I_{M_L}=\{T_0\}$ and $\delta_{V,M}=0$
(see Figure~\ref{f:dimvector}).
By Lemma~\ref{l:mesh} and~\eqref{e:caseII}, we have:
\begin{align*}
\begin{split}
\dimv_V(M) &= \dimv_V(M_U)+\dimv_V(M_L)-\dimv_V(\tau M)+\delta_{V,M} \\
       &= \dimv_V(Y_M)+\dimv_V(M_L)-\dimv_V(Y_{M_L})+\delta_{V,M} \\
&= 
\begin{cases}
\dimv_V(Y_M)+\dimv_V(X_M)+\dimv_V(Y_{M_L})-\dimv_V(Y_{M_L})+1, & V\in \I_{M_L}; \\
\dimv_V(Y_M)+\dimv_V(X_M)+\dimv_V(Y_{M_L})-\dimv_V(Y_{M_L}), & V\not\in \I_{M_L};
\end{cases} \\
&=
\begin{cases}
\dimv_V(X_M)+\dimv_V(Y_M)+1, & V\in \I_M; \\
\dimv_V(X_M)+\dimv_V(Y_M), & V\not\in \I_M.
\end{cases}
\end{split}
\end{align*}

\noindent \textbf{Case III}:
If $M=M_{i,l}$ with $1\leq i\leq r-1,r-i\leq l\leq 2r-2-i\}$, but is not in one
of the cases above, then $\M_M\cap D=\{M_L,M_U,\tau M,M\}$.
Note that $X_{M_U}=X_{\tau M}$, $Y_{M_U}=Y_M$, $X_{M_L}=X_M$, $Y_{M_L}=Y_{\tau M}$
and $\delta_{V,M}=0$. We also have that $\I_{M_U}=\I_{\tau M}$ and $\I_{M_L}=\I_M$.
Applying the inductive hypothesis to $M_L$, $M_U$ and $\tau M$, we have:
\begin{align}
\label{e:caseIIIa}
\dimv_V(M_U) &=
\begin{cases}
\dimv_V(X_{\tau M})+\dimv_V(Y_M)+1, & V\in \I_{\tau M}; \\
\dimv_V(X_{\tau M})+\dimv_V(Y_M), & V\not\in \I_{\tau M};
\end{cases} \\
\label{e:caseIIIb}
\dimv_V(M_L) &=
\begin{cases}
\dimv_V(X_M)+\dimv_V(Y_{\tau M})+1, & V\in \I_M; \\
\dimv_V(X_M)+\dimv_V(Y_{\tau M}), & V \not\in \I_M;
\end{cases} \\
\label{e:caseIIIc}
\dimv_V(\tau M) &=
\begin{cases}
\dimv_V(X_{\tau M})+\dimv_V(Y_{\tau M})+1, & V\in \I_{\tau M}; \\
\dimv_V(X_{\tau M})+\dimv_V(Y_{\tau M}), & V\not\in \I_{\tau M}.
\end{cases}
\end{align}
By Lemma~\ref{l:mesh} and~\eqref{e:caseIIIa}--\eqref{e:caseIIIc}, we obtain:
\begin{align*}
\begin{split}
\dimv_V(M) &= \dimv_V(M_U)+\dimv_V(M_L)-\dimv_V(\tau M) \\
&=\begin{cases}
\dimv_V(X_M)+\dimv_V(Y_M)+1, & V\in \I_M; \\
\dimv_V(X_M)+\dimv_V(Y_M), & V\not\in \I_M.
\end{cases}
\end{split}
\end{align*}
The result now follows by induction.
\end{proof}

\begin{figure}
$$
\begin{tikzpicture}[scale=0.5,baseline=(bb.base),quivarrow/.style={black, -latex,shorten <=-3pt, shorten >=-3pt},translate/.style={black, dotted, shorten >=13pt, shorten <=13pt}]
  
\path (0,0) node (bb) {}; 

\draw[fill=gray!30,opacity=0.4] (0,0) -- (3,3) -- (6,0) -- (3,-3) -- (0,0);

\draw (2,-2) node (M) {$\bullet$};
\draw ($(M)+(0:0.6)$) node {$\scriptstyle M$};

\draw (1,-1) node (MU) {$\bullet$};
\draw ($(MU)+(45:0.6)$) node {$\scriptstyle M_U$};

\draw (1,-3) node (ML) {$\bullet$};
\draw ($(ML)+(270:0.6)$) node {$\scriptstyle X_M=M_L$};

\draw (0,-2) node (tM) {$\bullet$};
\draw ($(tM)+(180:1.5)$) node {$\scriptstyle X_{M_U}=\tau M$};

\draw (-1,1) node (Y) {$\bullet$};

\draw (-1,1.4) node {$\scriptstyle Y_{M_U}=Y_{M}$};

\draw (0,0) node (O) {};

\draw[quivarrow] (tM) -- (MU);
\draw[quivarrow] (tM) -- (ML);
\draw[quivarrow] (ML) -- (M);
\draw[quivarrow] (MU) -- (M);
\draw[quivarrow] (Y) -- (O);

\draw (3,-5.5) node {Case I};

\end{tikzpicture}
\begin{tikzpicture}[scale=0.5,baseline=(bb.base),quivarrow/.style={black, -latex,shorten <=-3pt, shorten >=-3pt},translate/.style={black, dotted, shorten >=13pt, shorten <=13pt}]
  
\path (0,0) node (bb) {}; 

\draw[fill=gray!30,opacity=0.4] (0,0) -- (3,3) -- (6,0) -- (3,-3) -- (0,0);

\draw (2,2) node (M) {$\bullet$};
\draw ($(M)+(0:0.6)$) node {$\scriptstyle M$};

\draw (1,3) node (MU) {$\bullet$};
\draw ($(MU)+(45:0.6)$) node {$\scriptstyle Y_M=M_U$};

\draw (1,1) node (ML) {$\bullet$};
\draw ($(ML)+(270:0.6)$) node {$\scriptstyle M_L$};

\draw (0,2) node (tM) {$\bullet$};
\draw ($(tM)+(180:1.5)$) node {$\scriptstyle Y_{M_L}=\tau M$};

\draw (-1,-1) node (X) {$\bullet$};
\draw (-1,-1.5) node {$\scriptstyle X_{M_L}=X_M$};
\draw (0,0) node (O) {};

\draw[quivarrow] (tM) -- (MU);
\draw[quivarrow] (tM) -- (ML);
\draw[quivarrow] (ML) -- (M);
\draw[quivarrow] (MU) -- (M);
\draw[quivarrow] (X) -- (O);

\draw (3,-5.5) node {Case II};

\end{tikzpicture}
\begin{tikzpicture}[scale=0.5,baseline=(bb.base),quivarrow/.style={black, -latex,shorten <=-3pt, shorten >=-3pt},translate/.style={black, dotted, shorten >=13pt, shorten <=13pt}]
  
\path (0,0) node (bb) {}; 

\draw[fill=gray!30,opacity=0.4] (0,0) -- (3,3) -- (6,0) -- (3,-3) -- (0,0);

\draw (4,0) node (M) {$\bullet$};
\draw ($(M)+(0:0.6)$) node {$\scriptstyle M$};

\draw (3,1) node (MU) {$\bullet$};
\draw ($(MU)+(45:0.6)$) node {$\scriptstyle M_U$};

\draw (3,-1) node (ML) {$\bullet$};
\draw ($(ML)+(270:0.6)$) node {$\scriptstyle M_L$};

\draw (2,0) node (tM) {$\bullet$};
\draw ($(tM)+(180:0.8)$) node {$\scriptstyle \tau M$};

\draw (0,2) node (XtM) {$\bullet$};
\draw (1,3) node (XM) {$\bullet$};
\draw (0,-2) node (YtM) {$\bullet$};
\draw (1,-3) node (YM) {$\bullet$};

\draw (-1.6,2) node {$\scriptstyle Y_{\tau M}=Y_{M_L}$};
\draw (1,3.5) node {$\scriptstyle Y_{M_U}=Y_M$};
\draw (-1.68,-2) node {$\scriptstyle X_{\tau M}=X_{M_U}$};
\draw (1,-3.5) node {$\scriptstyle X_{M_L}=X_M$};

\draw (1,1) node (XtMB) {};
\draw (2,2) node (XMB) {};
\draw (1,-1) node (YtMB) {};
\draw (2,-2) node (YMB) {};

\draw[quivarrow] (tM) -- (MU);
\draw[quivarrow] (tM) -- (ML);
\draw[quivarrow] (ML) -- (M);
\draw[quivarrow] (MU) -- (M);

\draw[quivarrow] (XtM) -- (XtMB);
\draw[quivarrow] (YtM) -- (YtMB);
\draw[quivarrow] (XM) -- (XMB);
\draw[quivarrow] (YM) -- (YMB);

\draw (3,-5.5) node {Case III};

\end{tikzpicture}
$$

\caption{Proof of Lemma~\ref{l:dimvector}. The shaded region is the region $D$.}
\label{f:dimvector}
\end{figure}
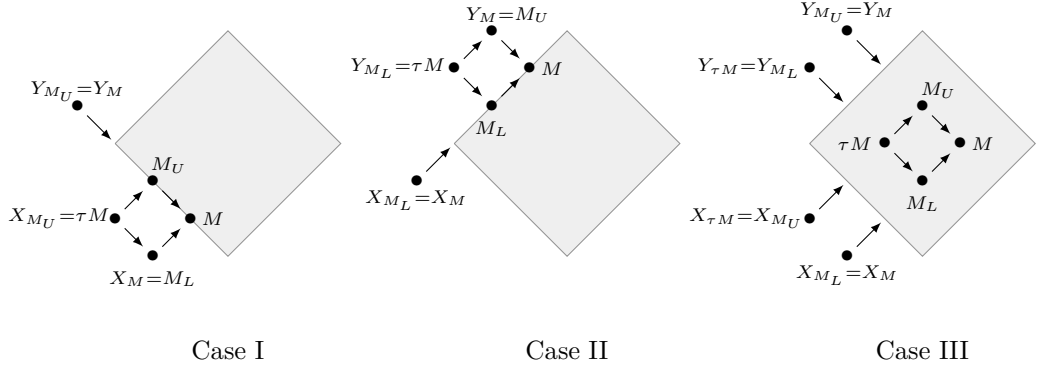

Let $\I$ denote the set of all injective objects in $\W_{T_0}$ and
set
\begin{equation}
\label{e:IMP}
\I'_M=\I\setminus \I_M,
\end{equation}
i.e.\ the set of objects in
the coray through $T_0$ which are on or below the lowest intersection point
with the ray through $M$.
Suppose that there is an indecomposable direct summand of $T$ in
$\I'_M$. Let $X'_M\cong M_{j,r-j}$ be such a summand with maximal quasilength
and set $Z_M=M_{0,j-2}$. Otherwise, we set $Z_M=M_{0,r-2}$.

\begin{remark}
In the first case above, the object $Z_M$ can be constructed geometrically as follows.
Let $Z'_M=M_{j,r-2}$ be the unique object in the ray through $X'_M$ of quasilength $r-2$.
Then $Z_M=M_{0,j-2}$ is the unique object in the coray through $Z'_M$ which is a projective in
$\W_{\tau T_0}$. See Figure~\ref{f:regionD}.
\end{remark}

\begin{lemma} \label{l:dimZ}
Let $V$ be an indecomposable direct summand of $T$.
Then
$$\dimv_V(Z_M)=
\begin{cases}
1, & V\in \I_M\setminus {T_0}; \\
0, & \text{otherwise}.
\end{cases}
$$
\end{lemma}
\begin{proof}
If $V$ is preprojective (i.e.\ an indecomposable direct summand of $U$)
then, since $Z_M\in \W_{\tau T_0}$, we have $\dimv_V(Z_M)=0$ by Lemma~\ref{l:additive}.

Suppose that $V$ is an indecomposable direct summand of $T_{\T}$.
The quasisocle of $Z_M$ is $Q_0$, which does not lie in $\W_{T_0}$.
Since $V\in \W_{T_0}$, it follows from Corollary~\ref{c:wingzero} that
$\Hom(V,Z_M)=0$. Hence (using~\eqref{e:tau2}), we have:
$$\dimv_V(Z_M)=\dim \HomF{\C}(V,Z_M)=\dim \Hom(Z_M,\tau^2 V).$$

Consider first the case where there is an indecomposable direct summand
of $T$ in $\I'_M$, so $X'_M$ is defined.
We have $\Hom(Z_M,\tau^2 V)\not=0$ if and only if $\Hom(\tau^{-2} Z_M,V)\not=0$.
By Lemma~\ref{l:homlow} and the fact that $V\in \W_{T_0}$, this holds if and only if $V$
lies in the rectangle with corners
$\tau^{-2} Z_M=M_{2,j-2}$, $M_{2,r-2}$, $M_{j-1,1}$ and $M_{j-1,r-j+1}$
In this case, $\dim\Hom(Z_M,\tau^2 V)=1$.

As $V$ and $X'_M$ are indecomposable direct summands of $T$, we have that
$\Hom(V,\tau X'_M)=0$. So, again using Lemma~\ref{l:homlow}, $V$ cannot lie
in the rectangle with corners $M_{1,j-1}$, $M_{j-1,1}$, $M_{j-1,r-j}$ and
$M_{1,r-2}$. Combining this fact with the statement in the previous paragraph, we
see that
$\Hom(Z_M,\tau^2 V)\not=0$ if and only if $V\in \I$, $V\not\cong T_0$ and
$V$ has quasilength greater than $\ql(X'_M)=r-j$.

However, by the definition of $X'_M$, there are no indecomposable direct summands
of $V$ in $\I'_M$ with quasilength greater than $\ql(X'_M)$. Hence
$\Hom(Z_M,\tau^2 V)\not=0$ if and only if $V\in \I_M\setminus \{	T_0\}$.

If there is no indecomposable direct summand of $T$ in $\I'_M$, then $Z_M=M_{0,r-2}$.
By Lemma~\ref{l:homlow} and the fact that $V\in \W_{T_0}$, we have
that
$$\dim\Hom(Z_M,\tau^2 V)=\dim\Hom(\tau^{-2} Z_M,V)$$
is $1$ if and only if
$V$ lies in the coray through $\tau^{-2} Z_M=M_{2,r-2}$, i.e.\ if and only if
$V\in \I\setminus \{T_0\}$. Since there is no indecomposable summand of $T$ in
$\I'_M$, this holds if and only if $V\in \I_M\setminus \{T_0\}$, and the proof is complete.
\end{proof}

\begin{figure}
$$
\begin{tikzpicture}[scale=0.5,baseline=(bb.base),quivarrow/.style={black, -latex,shorten <=-3pt, shorten >=-3pt},translate/.style={black, dotted, shorten >=13pt, shorten <=13pt}]
  
\path (0,0) node (bb) {}; 

\foreach \i in {0,...,9}{
\foreach \j in {0,...,\i}{
\node [fill=white,inner sep=-1pt] at (\j*2+10-\i,10-\i) {$\circ$};}}

\foreach \i in {0,1,2,3,4,5,6,7,8,9}
{\node [fill=white,inner sep=-1pt] at (21-\i,\i+1) {$\circ$};}

\node at (16,8) {$\circ$};
\node at (17,9) {$\circ$};

\draw (8,8) node (A) {};
\draw (9,9) node (B) {};
\draw (13,3) node (C) {};
\draw (14,4) node (D) {};

\draw (10,10.5) node {$\scriptstyle \tau T_0$};
\draw (12.1,10.5) node {$\scriptstyle T_0$};
\draw (1,0.5) node {$\scriptstyle M_{0,1}$};
\draw (3,0.5) node {$\scriptstyle M_{1,1}$};
\draw (13,0.5) node {$\scriptstyle M_{j-1,1}$};
\draw (19.8,4) node {$\scriptstyle X'_M=M_{j,r-j}$};
\draw (18.8,5) node {$\scriptstyle M_{j-1,r-j+1}$};
\draw (3.2,5) node {$\scriptstyle Z_M=M_{0,j-2}$};
\draw (8,5.7) node {$\scriptstyle M_{1,j-1}$};
\draw (9,4.7) node {$\scriptstyle M_{2,j-2}$};
\draw (11,9.5) node {$\scriptstyle M_{1,r-2}$};
\draw (17,3.5) node {$\scriptstyle M_{j-1,r-j}$};
\draw (17,9.5) node {$\scriptstyle M$};

\draw[dotted] (15.7,7.3) -- (14.7,6.3) -- (11,10) -- (12,11) -- (15.7,7.3);
\draw[dotted] (16.3,6.7) -- (15.3,5.7) -- (21,0) -- (22,1) -- (16.3,6.7);
\draw (14.5,9.5) node {$\I_M$};
\draw (21.8,2) node {$\I'_M$};

\begin{scope}[on background layer]
\draw[fill=gray!30,opacity=0.4,draw=none] (9,5) -- (13,9) -- (17,5) -- (13,1) -- (9,5);
\draw (9,5) -- (13,9) -- (17,5) -- (16,4);
\draw[dashed] (13,1) -- (16,4) -- (11,9) -- (8,6) -- (13,1);
\end{scope}


\foreach \i/\j in {1/1,1/10,3/3,3/8,4/2,5/1,7/1,7/4,9/2,10/1}
{\draw (\i*2+\j,\j) node {$\bullet$};}

\end{tikzpicture}
$$
\caption{Part of the proof of Lemma~\ref{l:dimZ} in the case $r=11$, $j=7$.
The dotted rectangles show $\I_M$ and $\I'_M$.
Since $V\in\W_{T_0}$, we have $\Hom(\tau^{-2}Z_M,V)\not=0$ if and only
if $V$ lies in the shaded rectangle. Since $\Ext(V,X'_M)=0$, $V$ cannot lie
in the dashed rectangle. By the definition of $X'_M$, $V$ cannot lie in
the part of $\I'_M$ above $X'_M$. Hence $\Hom(Z_M,\tau^2 V)\not=0$ if and only
if $V\in \I_M\setminus \{T_0\}$.}
\label{f:dimZ}
\end{figure}
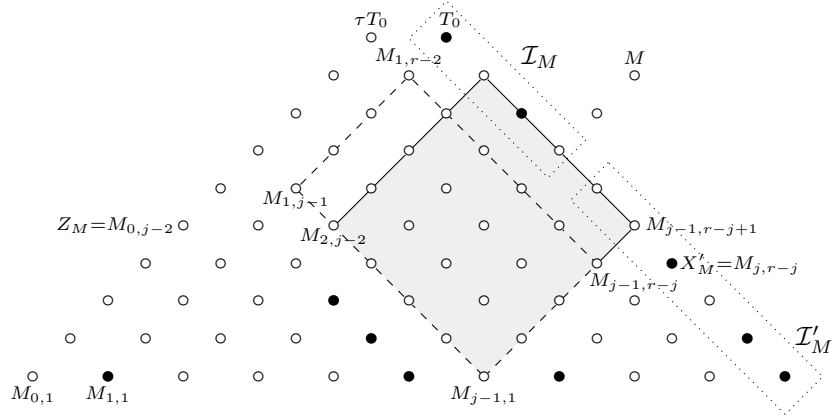

\begin{proposition} \label{p:three}
Let $M\in D$ and let $V$ be an indecomposable summand of $T$. Then we have:
$$\dimv_V(M)=\dimv_V(X_M)+\dimv_V(Y_M)+\dimv_V(Z_M)+\delta_{V,T_0}.$$
\end{proposition}
\begin{proof}
This follows from Lemmas~\ref{l:dimvector} and~\ref{l:dimZ}.
\end{proof}

\begin{theorem}
Let $Q$ be a quiver of tame representation type and let $\C$ be the corresponding cluster
category. Let $T$ be a cluster-tilting object in $\C$ and $\Lambda=\End_{\C}(T)^{\opp}$ the corresponding cluster-tilted algebra. Let $Q_{\Lambda}$ be the quiver of $\Lambda$
and $\mathcal{A}(Q_{\Lambda})$ the corresponding cluster algebra. Then
any $d$-vector of $\mathcal{A}(Q_{\Lambda})$ can be written as a sum of at most three dimension
vectors of indecomposable rigid $\Lambda$-modules.
\end{theorem}
\begin{proof}
Let $M$ be a rigid indecomposable object in $\C$ which is not an indecomposable direct summand of
$\tau T$ and $x_M$ the corresponding non-initial cluster variable of $\mathcal{A}(Q_{\Lambda})$.
By~\cite[Thm.\ A]{BM10}, if $M$ is transjective or in a tube of rank $r$ containing no indecomposable
direct summand of $T$ of quasilength $r-1$ then the $d$-vector of $x_M$ coincides with the dimension
vector of the $\Lambda$-module $\clu{M}$.

Suppose that $M$ lies in a tube which contains an indecomposable direct summand $T_0$
of $T$ of quasilength $r-1$. If $M$ is contained in the wing $\W_{\tau T_0}$ then the
$d$-vector of $x_M$ again coincides with the dimension vector of $\clu{M}$. If not, then
$M$ must lie in the region $D$ defined in~\eqref{e:Ddef} (after Lemma~\ref{l:mesh}) (note that
in addition it must have quasilength at most $r-1$, but we don't need that here).
By construction, the quasilengths of $X_M$ and $Z_M$ are both less than or equal to $r-1$,
so they are $\tau$-rigid $\Lambda$-modules by Lemma~\ref{l:lowquasilength}.
Since $Y_M$ lies in $\HH_0$, it follows from Theorem~\ref{t:classification} that $Y_M$ is a rigid $\Lambda$-module. 
By~\cite[Thm.\ A]{BM10}, the
$d$-vector $(d_V)_{V\in \ind(T)}$ of $x_M$ satisfies 
$$d_V(x_M)=\dimv_V(M)-\delta_{V,T_0}.$$
The result now follows from Proposition~\ref{p:three}.
\end{proof}

\noindent \textbf{Acknowledgements}
Both authors would like to thank the referee for very helpful comments
and would like to thank the MSRI, Berkeley for
kind hospitality during a semester on cluster algebras
in Autumn 2012. RJM was Guest Professor at the Department of 
Mathematical Sciences, NTNU, Trondheim, Norway, during the 
autumn semester of 2014 and would like to thank the 
Department for their kind hospitality.

\end{document}